\documentclass[a4paper,10pt]{article}

\usepackage{amsthm,amssymb}
\usepackage{mathtools}
\usepackage[hidelinks]{hyperref}
\usepackage[a4paper,width=160mm,top=27mm,bottom=27mm]{geometry}

\newtheorem{theorem}{Theorem}[section]
\newtheorem{lemma}[theorem]{Lemma}
\newtheorem{definition}[theorem]{Definiton}

\newtheorem{corollary}[theorem]{Corollary}

\theoremstyle{definition}
\newtheorem{remx}[theorem]{Remark}
\newenvironment{remark}
  {\pushQED{\qed}\remx}
  {\popQED\endremx}

\makeatletter
\newsavebox\myboxA
\newsavebox\myboxB
\newlength\mylenA

\newcommand*\yoverline[2][0.75]{%
    \sbox{\myboxA}{$\m@th#2$}%
    \setbox\myboxB\null
    \ht\myboxB=\ht\myboxA%
    \dp\myboxB=\dp\myboxA%
    \wd\myboxB=#1\wd\myboxA
    \sbox\myboxB{$\m@th\overline{\copy\myboxB}$}
    \setlength\mylenA{\the\wd\myboxA}
    \addtolength\mylenA{-\the\wd\myboxB}%
    \ifdim\wd\myboxB<\wd\myboxA%
       \rlap{\hskip 0.5\mylenA\usebox\myboxB}{\usebox\myboxA}%
    \else
        \hskip -0.5\mylenA\rlap{\usebox\myboxA}{\hskip 0.5\mylenA\usebox\myboxB}%
    \fi}
\makeatother

\numberwithin{equation}{section}

\begin{document}





\newcommand{\diver}{\operatorname{div}}
\newcommand{\lin}{\operatorname{Lin}}
\newcommand{\curl}{\operatorname{curl}}
\newcommand{\ran}{\operatorname{Ran}}
\newcommand{\kernel}{\operatorname{Ker}}
\newcommand{\la}{\langle}
\newcommand{\ra}{\rangle}
\newcommand{\N}{\mathbb{N}}
\newcommand{\R}{\mathbb{R}}
\newcommand{\C}{\mathbb{C}}
\newcommand{\T}{\mathbb{T}}

\newcommand{\ld}{\lambda}
\newcommand{\fai}{\varphi}
\newcommand{\0}{0}
\newcommand{\n}{\mathbf{n}}
\newcommand{\uu}{{\boldsymbol{\mathrm{u}}}}
\newcommand{\UU}{{\boldsymbol{\mathrm{U}}}}
\newcommand{\buu}{\bar{{\boldsymbol{\mathrm{u}}}}}
\newcommand{\ten}{\\[4pt]}
\newcommand{\six}{\\[-3pt]}
\newcommand{\nb}{\nonumber}
\newcommand{\hgamma}{H_{\Gamma}^1(\OO)}
\newcommand{\opert}{O_{\varepsilon,h}}
\newcommand{\barx}{\bar{x}}
\newcommand{\barf}{\bar{f}}
\newcommand{\hatf}{\hat{f}}
\newcommand{\xoneeps}{x_1^{\varepsilon}}
\newcommand{\xh}{x_h}
\newcommand{\scaled}{\nabla_{1,h}}
\newcommand{\scaledb}{\widehat{\nabla}_{1,\gamma}}
\newcommand{\vare}{\varepsilon}
\newcommand{\A}{{\bf{A}}}
\newcommand{\RR}{{\bf{R}}}
\newcommand{\B}{{\bf{B}}}
\newcommand{\CC}{{\bf{C}}}
\newcommand{\D}{{\bf{D}}}
\newcommand{\K}{{\bf{K}}}
\newcommand{\oo}{{\bf{o}}}
\newcommand{\id}{{\bf{Id}}}
\newcommand{\E}{\mathcal{E}}
\newcommand{\ii}{\mathcal{I}}
\newcommand{\sym}{\mathrm{sym}}
\newcommand{\lt}{\left}
\newcommand{\rt}{\right}
\newcommand{\ro}{{\bf{r}}}
\newcommand{\so}{{\bf{s}}}
\newcommand{\e}{{\bf{e}}}
\newcommand{\ww}{{\boldsymbol{\mathrm{w}}}}
\newcommand{\zz}{{\boldsymbol{\mathrm{z}}}}
\newcommand{\U}{{\boldsymbol{\mathrm{U}}}}
\newcommand{\G}{{\boldsymbol{\mathrm{G}}}}
\newcommand{\VV}{{\boldsymbol{\mathrm{V}}}}
\newcommand{\II}{{\boldsymbol{\mathrm{I}}}}
\newcommand{\ZZ}{{\boldsymbol{\mathrm{Z}}}}
\newcommand{\hKK}{{{\bf{K}}}}
\newcommand{\f}{{\bf{f}}}
\newcommand{\g}{{\bf{g}}}
\newcommand{\lkk}{{\bf{k}}}
\newcommand{\tkk}{{\tilde{\bf{k}}}}
\newcommand{\W}{{\boldsymbol{\mathrm{W}}}}
\newcommand{\Y}{{\boldsymbol{\mathrm{Y}}}}
\newcommand{\EE}{{\boldsymbol{\mathrm{E}}}}
\newcommand{\F}{{\bf{F}}}
\newcommand{\spacev}{\mathcal{V}}
\newcommand{\spacevg}{\mathcal{V}^{\gamma}(\Omega\times S)}
\newcommand{\spacevb}{\bar{\mathcal{V}}^{\gamma}(\Omega\times S)}
\newcommand{\spaces}{\mathcal{S}}
\newcommand{\spacesg}{\mathcal{S}^{\gamma}(\Omega\times S)}
\newcommand{\spacesb}{\bar{\mathcal{S}}^{\gamma}(\Omega\times S)}
\newcommand{\skews}{H^1_{\barx,\mathrm{skew}}}
\newcommand{\kk}{\mathcal{K}}
\newcommand{\OO}{O}
\newcommand{\bhe}{{\bf{B}}_{\vare,h}}
\newcommand{\pp}{{\mathbb{P}}}
\newcommand{\ff}{{\mathcal{F}}}
\newcommand{\mWk}{{\mathcal{W}}^{k,2}(\Omega)}
\newcommand{\mWa}{{\mathcal{W}}^{1,2}(\Omega)}
\newcommand{\mWb}{{\mathcal{W}}^{2,2}(\Omega)}
\newcommand{\twos}{\xrightharpoonup{2}}
\newcommand{\twoss}{\xrightarrow{2}}
\newcommand{\bw}{\bar{w}}
\newcommand{\bz}{\bar{{\bf{z}}}}
\newcommand{\tw}{{W}}
\newcommand{\tr}{{{\bf{R}}}}
\newcommand{\tz}{{{\bf{Z}}}}
\newcommand{\lo}{{{\bf{o}}}}
\newcommand{\hoo}{H^1_{00}(0,L)}
\newcommand{\ho}{H^1_{0}(0,L)}
\newcommand{\hotwo}{H^1_{0}(0,L;\R^2)}
\newcommand{\hooo}{H^1_{00}(0,L;\R^2)}
\newcommand{\hhooo}{H^1_{00}(0,1;\R^2)}
\newcommand{\dsp}{d_{S}^{\bot}(\barx)}
\newcommand{\LB}{{\bf{\Lambda}}}
\newcommand{\LL}{\mathbb{L}}
\newcommand{\mL}{\mathcal{L}}
\newcommand{\mhL}{\widehat{\mathcal{L}}}
\newcommand{\loc}{\mathrm{loc}}
\newcommand{\tqq}{\mathcal{Q}^{*}}
\newcommand{\tii}{\mathcal{I}^{*}}
\newcommand{\Mts}{\mathbb{M}}
\newcommand{\pot}{\mathrm{pot}}
\newcommand{\tU}{{\widehat{\bf{U}}}}
\newcommand{\tVV}{{\widehat{\bf{V}}}}
\newcommand{\pt}{\partial}
\newcommand{\bg}{\Big}
\newcommand{\hA}{\widehat{{\bf{A}}}}
\newcommand{\hB}{\widehat{{\bf{B}}}}
\newcommand{\hCC}{\widehat{{\bf{C}}}}
\newcommand{\hD}{\widehat{{\bf{D}}}}
\newcommand{\fder}{\partial^{\mathrm{MD}}}
\newcommand{\Var}{\mathrm{Var}}
\newcommand{\pta}{\partial^{0\bot}}
\newcommand{\ptaj}{(\partial^{0\bot})^*}
\newcommand{\ptb}{\partial^{1\bot}}
\newcommand{\ptbj}{(\partial^{1\bot})^*}
\newcommand{\geg}{\Lambda_\vare}
\newcommand{\tpta}{\tilde{\partial}^{0\bot}}
\newcommand{\tptb}{\tilde{\partial}^{1\bot}}
\newcommand{\ua}{u_\alpha}
\newcommand{\pa}{p\alpha}
\newcommand{\qa}{q(1-\alpha)}
\newcommand{\Qa}{Q_\alpha}
\newcommand{\Qb}{Q_\eta}
\newcommand{\ga}{\gamma_\alpha}
\newcommand{\gb}{\gamma_\eta}
\newcommand{\ta}{\theta_\alpha}
\newcommand{\tb}{\theta_\eta}


\newcommand{\mH}{{E}}
\newcommand{\mN}{{N}}
\newcommand{\mD}{{\mathcal{D}}}
\newcommand{\csob}{\mathcal{S}}
\newcommand{\mA}{{A}}
\newcommand{\mK}{{Q}}
\newcommand{\mS}{{S}}
\newcommand{\mI}{{I}}
\newcommand{\tas}{{2_*}}
\newcommand{\tbs}{{2^*}}
\newcommand{\tm}{{\tilde{m}}}
\newcommand{\tdu}{{\phi}}
\newcommand{\tpsi}{{\tilde{\psi}}}
\newcommand{\Z}{{\mathbb{Z}}}
\newcommand{\tsigma}{{\tilde{\sigma}}}
\newcommand{\tg}{{\tilde{g}}}
\newcommand{\tG}{{\tilde{G}}}
\newcommand{\mM}{{M}}
\newcommand{\mC}{\mathcal{C}}
\newcommand{\wlim}{{\text{w-lim}}\,}
\newcommand{\diag}{L_t^\ba L_x^\br}
\newcommand{\vu}{ u}
\newcommand{\vz}{ z}
\newcommand{\vv}{ v}
\newcommand{\ve}{ e}
\newcommand{\vw}{ w}
\newcommand{\vf}{ f}
\newcommand{\vh}{ h}
\newcommand{\vp}{ \vec P}
\newcommand{\ang}{{\not\negmedspace\nabla}}
\newcommand{\dxy}{\Delta_{x,y}}
\newcommand{\lxy}{L_{x,y}}
\newcommand{\gnsand}{\mathrm{C}_{\mathrm{GN},3d}}
\newcommand{\wmM}{\widehat{{M}}}
\newcommand{\wmH}{\widehat{{E}}}
\newcommand{\wmI}{\widehat{{I}}}
\newcommand{\wmK}{\widehat{{Q}}}
\newcommand{\wmN}{\widehat{{N}}}
\newcommand{\wm}{\widehat{m}}
\newcommand{\ba}{\mathbf{a}}
\newcommand{\bb}{\mathbf{b}}
\newcommand{\br}{\mathbf{r}}
\newcommand{\bs}{\mathbf{s}}
\newcommand{\bq}{\mathbf{q}}

\title{Normalized ground states and threshold scattering for focusing NLS on $\R^d\times\T$ via semivirial-free geometry}
\author{Yongming Luo \thanks{Institut f\"{u}r Wissenschaftliches Rechnen, Technische Universit\"at Dresden, Germany} \thanks{\href{mailto:yongming.luo@tu-dresden.de}{Email: yongming.luo@tu-dresden.de}}
}

\date{}
\maketitle

\begin{abstract}
We study the focusing NLS
\begin{align}\label{nls_abstract}
i\pt_t u+\Delta_{x,y} u=-|u|^\alpha u\tag{NLS}
\end{align}
on the waveguide manifold $\R_x^d\times\T_y$ in the intercritical regime $\alpha\in(\frac{4}{d},\frac{4}{d-1})$. By assuming that the \eqref{nls_abstract} is independent of $y$, it reduces to the focusing intercritical NLS on $\R^d$, which is known to have standing wave and finite time blow-up solutions. Naturally, we ask whether these special solutions with non-trivial $y$-dependence exist. In this paper we give an affirmative answer to this question. To that end, we introduce the concept of \textit{semivirial} functional
and consider a minimization problem $m_c$ on the semivirial-vanishing manifold with prescribed mass $c$. We prove that for any $c\in(0,\infty)$ the variational problem $m_c$ has a ground state optimizer $u_c$ which also solves the standing wave equation
$$-\Delta_{x,y}u_c+\beta_c u_c=|u|^\alpha u $$
with some $\beta_c>0$. Moreover, we prove the existence of a critical number $c_*\in(0,\infty)$ such that
\begin{itemize}
\item For $c\in(0,c_*)$, any optimizer $u_c$ of $m_c$ must satisfy $\pt_y u_c\neq 0$.
\item For $c\in(c_*,\infty)$, any optimizer $u_c$ of $m_c$ must satisfy $\pt_y u_c=0$.
\end{itemize}
Finally, we prove that the previously constructed ground states characterize a sharp threshold for the bifurcation of scattering and finite time blow-up solutions in dependence of the sign of the semivirial.
\end{abstract}


\section{Introduction and main results}\label{sec:Introduction and main results}
In this paper, we study the focusing nonlinear Schr\"odinger equation (NLS)
\begin{align}\label{nls}
i\pt_t u+\Delta_{x,y} u=-|u|^\alpha u
\end{align}
on the waveguide manifold $\R_x^d\times \T_y$ with $d\geq 1$, $\T=\R/2\pi\Z$ and $\alpha$ in the intercritical regime $(\frac{4}{d},\frac{4}{d-1})$. The equation \eqref{nls} arises from various physical applications such as nonlinear optics and Bose-Einstein condensation. For a more comprehensive introduction on the physical background of \eqref{nls}, we refer to \cite{waveguide_ref_1,waveguide_ref_2,waveguide_ref_3}.

Besides its physical importance, the mixed type nature of the domain $\R^d\times\T$ also makes the mathematical analysis on \eqref{nls} rather delicate and challenging. In recent years, the study of dispersive equations on product spaces has been drawing much attention from the mathematical community. A first study on local well-posedness results of NLS on product spaces might date back to Tzvetkov and Visciglia \cite{TNCommPDE}. Therein, the authors studied the cubic NLS on $\R^d\times \mathcal{M}$ with $d\geq 2$ and $\mathcal{M}$ a compact Riemannian manifold. Particularly, it was shown that the cubic NLS on $\R^d\times \mathcal{M}$ with small initial data is globally well-posed and scattering in time in certain non-isotropic Sobolev spaces. In \cite{TTVproduct2014}, Terracini, Tzvetkov and Visciglia studied \eqref{nls} on $\R^d\times \mathcal{M}$, where $\mathcal{M}$ is a $k$-dimensional compact Riemannian manifold, in the mass-subcritical regime $\alpha\in(0,\frac{4}{d+k})$. Specifically, the authors proved the existence of normalized ground states $u_c$ of the variational problem
\begin{align*}
\nu_c:=\inf_{u\in H^1(\R^d\times\mathcal{M})}\{\mH(u):\mM(u)=c\}
\end{align*}
for any $c\in(0,\infty)$, where $\mM(u)$ and $\mH(u)$ are the conventional mass and energy corresponding to \eqref{nls} (see \eqref{def of mass} and \eqref{def of mhu} for the precise definitions) respectively. Another interesting result proved in \cite{TTVproduct2014} (which is closely related to the results in this paper) is the existence of a critical number $c_*\in(0,\infty)$, such that for any $c\in(0,c_*)$ the normalized ground states $u_c$ coincide with the ones on $\R^d$, while for $c\in(c_*,\infty)$ the normalized ground states $u_c$ must have non-trivial $y$-dependence. Moreover, the authors proved that the ground states $u_c$ are orbitally stable and in the case $k=1$ a solution $u$ of \eqref{nls} is always global.

Concerning NLS with critically nonlinear potentials, the first result was given by Herr, Tataru, and Tzvetkov \cite{HerrTataruTz2}, where the authors proved local well-posedness of the cubic NLS in  $H^1(\R^{d}\times \T^{4-d})$ with $0\leq d\leq 3$. The first large data result was given by Ionescu and Pausader \cite{Ionescu2}. Therein, utilizing the concentration compactness principle initiated by Kenig and Merle \cite{KenigMerle2006} and the Black-Box-Theory developed in \cite{hyperbolic}, the authors proved that the defocusing cubic NLS is globally well-posed in $H^1(\R\times\T^3)$. Concerning the scattering results, a first breakthrough was made by Hani and Pausader \cite{HaniPausader}, where the authors proved that a solution of the defocusing quintic NLS is always global and scattering on $\R\times\T^2$. In fact, the result was originally announced based on a conjecture concerning the scattering result of the corresponding large scale resonant system, which was later proved to be true in \cite{R1T1Scattering}. Inspired by the results in \cite{HaniPausader,R1T1Scattering}, the large data scattering problems for defocusing NLS with algebraically critical potentials are now completely resolved \cite{HaniPausader,Yang_Zhao_2018,R1T1Scattering,CubicR2T1Scattering,RmT1,R2T2}. It is also worth noting that different from \cite{HaniPausader} where the concentration compactness principle was applied, by making use of the interaction Morawetz inequality originated in \cite{defocusing3d}, Tzvetkov and Visciglia \cite{TzvetkovVisciglia2016} proved that a solution of the defocusing analogue of \eqref{nls} is always global and scattering in time in the energy space. There has been nowadays a vast list of literature for the study of dispersive equations on product spaces. We refer to the papers \cite{FoprcellaHari2020,SystemProdSpace,ModifiedScattering,Barron,BarronChristPausader2021,Cheng_JMAA,ZhaoZheng2021,Luo_Waveguide_MassCritical,similar_cubic,GrossPitaevskiR1T1,4NLS} and the references therein for further readings in this direction.

Among all, we underline that the first large data scattering result for focusing NLS on product spaces was given in a previous paper \cite{Luo_Waveguide_MassCritical} of the author. Therein, by appealing to a scale-invariant  Gagliardo-Nirenberg inequality of additive type on $\R^2\times\T$, the author was able to formulate a scattering threshold, below which the solution of the focusing cubic NLS on $\R^2\times\T$ is global and scattering in time. Nevertheless, the threshold given in \cite{Luo_Waveguide_MassCritical} is possibly non-optimal since we were unable to prove the existence of standing wave solutions lying on the threshold. As we shall see, the mass-supercritical nature of \eqref{nls} shall enable us more space for manipulating the underlying analysis so that we can obtain much sharper results. Particularly, we are able to give scattering thresholds formulated in terms of ground states.

From now on we restrict our attention to \eqref{nls}. In \cite{TzvetkovVisciglia2016}, Tzvetkov and Visciglia proved that a solution of the defocusing analogue of \eqref{nls} is always global and scattering in time. Obviously, this does not hold for \eqref{nls}: by assuming that \eqref{nls} is independent of $y$, the equation reduces to the standard focusing intercritical NLS on $\R^d$ which is known to have standing wave and finite time blow-up solutions. Such simplification, however, is not quite satisfactory since the impact of the torus side on the equation is completely neglected. We therefore naturally ask the following questions:
\begin{itemize}
\item [(i)] Do standing wave or finite time blow-up solutions exist which are not necessarily independent of $y$?

\item[(ii)] As it is well-known that the standing wave solutions on $\R^d$ can be used to formulate thresholds for the bifurcation of scattering and finite time blow-up, do the standing wave solutions in (i), as long as they exist, play the same role as those on $\R^d$?
\end{itemize}
In this paper we give an affirmative answer to these questions. First notice that different from the mass-subcritical case in \cite{TTVproduct2014}, the energy $\mH(u)$ is unbounded from below on the manifold
$$ S(c)=\{u\in H^1(\R^d\times\T):\mM(u)=c\},$$
which can be easily verified by using simple scaling arguments (see for instance the proof of Lemma \ref{monotoneproperty} below). Inspired by this fact, it is natural to consider a minimization problem with extra constraints. As an educational example we consider for instance the stationary NLS
\begin{align}\label{stationary nls}
-\Delta_x u+\omega u=|u|^\alpha u
\end{align}
on $\R^d$. By testing \eqref{stationary nls} with $x\cdot\nabla_x u$ and followed by integration by parts, we obtain the well-known Pohozaev identity
\begin{align}\label{1poho}
0=\frac{d-2}{2}\|\nabla_x u\|_{L^2(\R^d)}^2+\frac{d\omega}{2}\|u\|_{L^2(\R^d)}^2-\frac{d}{\alpha+2}\| u\|_{L^{\alpha+2}(\R^d)}^{\alpha+2}.
\end{align}
As a consequence, by testing \eqref{stationary nls} with $u$ and eliminating $\|u\|_{L^2(\R^d)}^2$ in both identities we deduce the virial-vanishing condition
\begin{align*}
\wmK(u):=\|\nabla_x u\|_{L^2(\R^d)}^2-\frac{\alpha d}{2(\alpha+2)}\| u\|_{L^{\alpha+2}(\R^d)}^{\alpha+2}=0.
\end{align*}
The virial-vanishing condition indicates itself as a good candidate for being an additional constraint for the minimization problem. This suggests us to study the minimization problem
\begin{align}
\wm_c:=\inf_{u\in H^1(\R^d)}\{\wmH(u):\wmM(u)=c,\,\wmK(u)=0\},\label{wmc first def}
\end{align}
where $\wmM(u)$ and $\wmH(u)$ denote the mass and energy for a function $u$ defined on $\R^d$. Indeed, in his pioneering paper \cite{Jeanjean1997}, Jeanjean proved\footnote{In fact, existence results under more general assumptions were established in \cite{Jeanjean1997}.} that for any $c\in(0,\infty)$ the variational problem $\wm_c$ has a ground state optimizer $u_c$ which also solves \eqref{stationary nls} with some $\omega=\omega_c>0$. We intend to define an analogous minimization problem $m_c$ for \eqref{nls} on $\R^d\times\T$, following the same line in \cite{Jeanjean1997}. A first attempt would be to test the standing wave equation
\begin{align}\label{standing wave}
-\Delta_{x,y}u+\beta u=|u|^\alpha u
\end{align}
of \eqref{nls} with $(x,y)\cdot \nabla_{x,y}u$. This, however, does not lead us to the correct goal. Indeed, there would exist some boundary layer terms which can not be eliminated even by invoking the periodic boundary condition of $u$. Intuitively, in order to get reasonable scattering results, the rather weak dispersion along the $y$-direction should be compensated by the strong one along the $x$-direction. Motivated by such heuristics, we are naturally led to test \eqref{standing wave} with the virial action $x\cdot\nabla_x u$. Doing in this way and taking also the periodic boundary condition of $u$ into account, we see that a solution $u$ of \eqref{standing wave} will satisfy the so-called \textit{semivirial-vanishing} condition
\begin{align*}
\mK(u):=\|\nabla_x u\|_{L^2(\R^d\times\T)}^2-\frac{\alpha d}{2(\alpha+2)}\| u\|_{L^{\alpha+2}(\R^d\times \T)}^{\alpha+2}=0.
\end{align*}
Consequently, we will be looking for a ground state minimizer of the variational problem $m_c$ defined by
\begin{align}
m_c:=\inf_{u\in H^1(\R^d\times\T)}\{\mH(u):\mM(u)=c,\,\mK(u)=0\}.\label{mc first def}
\end{align}
At this point, we should also note that although the previous arguments are intuitively enlightening for defining a suitable minimization problem, the derivation of the Pohozaev identity for \eqref{nls} will not be a part of the proof for the main results. For the sake of simplicity we neglect the details of derivation of the Pohozaev identity and leave this as an exercise for the readers.

Our first result concerns with the existence of optimizers of $m_c$.
\begin{theorem}[Existence of ground states]\label{thm existence of ground state}
For any $c\in(0,\infty)$ the minimization problem $m_c$ defined by \eqref{mc first def} has a positive optimizer $u_c$. Moreover, $u_c$ solves the standing wave equation \eqref{standing wave} with some $\beta=\beta_c>0$.
\end{theorem}

In the Euclidean\footnote{We refer an NLS on Euclidean space to as an NLS on $\R^d$.} case, Jeanjean \cite{Jeanjean1997} proved the existence results by making use of some subtle min-max arguments, which can be briefly summarized as follows: by standard arguments, one is able to prove that the energy $\wmH(u)$ possesses a mountain pass geometry on the sphere
$$\widehat{S}(c)=\{u\in H^1(\R^d):\wmM(u)=c\}.$$
It is well-known (see for instance \cite{Ghoussoub1993}) that the mountain pass geometry already implies the existence of a Palais-Smale sequence. Nevertheless, the obtained Palais-Smale sequence is in general unbounded due to the mass-supercritical nature of the problem. To overcome this difficulty, Jeanjean modified the mountain pass geometry in a way such that the underlying functions at the left and right endpoints of the mountain pass path satisfy $\wmK(u)>0$ and $\wmK(u)<0$ respectively. In view of \cite[Thm. 4.1]{Ghoussoub1993}, one may further assume that the deduced Palais-Smale sequence also possesses asymptotically vanishing virial, which in turn implies the boundedness of the Palais-Smale sequence. Having these key ingredients in hand, the desired claim follows from a standard approximating procedure.

There are however several new issues arising when one attempts to adapt Jeanjean's arguments to the present model. First of all, we point out that a key point for identifying a non-vanishing weak limit of a minimizing sequence is the scale-invariant Gagliardo-Nirenberg inequality on $\R^d$. Indeed, using the virial-vanishing condition $\wmK(u)=0$ we obtain for a bounded minimizing sequence $(u_n)_n$
$$ \|\nabla_xu_n\|_{L^2(\R^d)}^2=\frac{\alpha d}{2(\alpha+2)}\|u_n\|_{L^{\alpha+2}(\R^d)}^{\alpha+2}
\leq C\|\nabla_x u_n\|_{L^2(\R^d)}^{\frac{\alpha d}{2}}\|u_n\|_{L^2(\R^d)}^{\frac{4-\alpha(d-2)}{2}}
\leq C\|\nabla_x u_n\|_{L^2(\R^d)}^{\frac{\alpha d}{2}},$$
which in turn implies
$$ \liminf_{n\to\infty}\|u_n\|_{L^{\alpha+2}(\R^d)}^{\alpha+2}\sim
\liminf_{n\to\infty}\|\nabla_x u_n\|_{L^{2}(\R^d)}^{2}>0.$$
Then the existence of a non-vanishing weak limit follows from the $pqr$-lemma (see \cite[Lem. 2.1]{FLL1986}). In general, on $\R^d\times \T$ we only have the inhomogeneous Gagliardo-Nirenberg inequality
$$\|u\|_{L^{\alpha+2}(\R^d\times\T)}^{\alpha+2}\leq
C\|u\|_{H^1(\R^d\times\T)}^{\frac{\alpha (d+1)}{2}}\|u\|_{L^2(\R^d\times\T)}^{\frac{4-\alpha(d-1)}{2}}
$$
which is less helpful for finding a non-vanishing weak limit. We shall overcome this difficulty by proving the following variant of the Gagliardo-Nirenberg inequality
\begin{align}
\|u\|_{L^{\alpha+2}(\R^d\times\T)}^{\alpha+2}\leq C\|\nabla_x u\|_{L^2(\R^d\times\T)}^{\frac{\alpha d}{2}}\|u\|_{L^2(\R^d\times\T)}^{\frac{4-\alpha(d-1)}{2}}
(\| u\|_{L^2(\R^d\times\T)}^{\frac{\alpha}{2}}+\|\pt_y u\|_{L^2(\R^d\times\T)}^{\frac{\alpha}{2}})\label{gn example}
\end{align}
of additive type which is particularly scale-invariant w.r.t. $x$-variable. Relying on \eqref{gn example} and a concentration compactness result from \cite{TTVproduct2014} (see Lemma \ref{lemma non vanishing limit}) we may identify a non-vanishing weak limit of a minimizing sequence, as desired. We refer to Lemma \ref{lemma gn additive} for further details.

Another difficulty is to show that the frequency parameter $\beta$ appearing in the standing wave equation \eqref{standing wave} is positive, which being essential for showing the $L^2$-strong convergence of the Palais-Smale sequence in Jeanjean's routine (the usage of the positivity of $\beta$ in our case is however different than the one in \cite{Jeanjean1997}). In the Euclidean case, the positivity of $\beta$ is a simple consequence of the Pohozaev identity. Indeed, from the Pohozaev identity of \eqref{standing wave} we obtain
\begin{align*}
\|\pt_y u\|_{L^2(\R^d\times\T)}^2+\beta\|u\|_{L^2(\R^d\times\T)}^2=\frac{2\alpha +(4-\alpha d)}{2(\alpha+2)}\|u\|_{L^{\alpha+2}(\R^d\times\T)}^{\alpha+2}.
\end{align*}
In the Euclidean case we simply have $\pt_y u=0$, thus the positivity of $\beta$ follows already from $u\neq 0$, which holds true for a non-vanishing weak limit that we have obtained previously. In general we can not exclude the possibility $\pt_y u\neq 0$ and the proof for finding a positive frequency $\beta$ is more technical in our case. As we will see, instead of using the $L^2$-type identities (in the sense that the identities contain the $L^2$-norm of $\nabla_x u$) such as the Pohozaev identity, we shall prove the positivity of $\beta$ by appealing to certain $L^1$-type identities (in the sense that the identities contain the $L_{\rm loc}^1$-norm of $\nabla_x u$) in the case $d=1$ and the Liouville-type theorems from \cite{Liouville_type} in the case $d\geq 2$.

The above mentioned points thus make our proof for Theorem \ref{thm existence of ground state} very different than the classical ones on $\R^d$. To make the difference clear, we briefly summarize the steps applied in the proof of Theorem \ref{thm existence of ground state} as follows:
\begin{itemize}
\item[(i)] In a first step, we prove that any minimizer $u_c$ of $m_c$ is automatically a solution of \eqref{standing wave}.
\item[(ii)] Afterwards, we show that if $u_c$ is a minimizer, so is $|u_c|$. This particularly implies that we can w.l.o.g. assume that a minimizer is non-negative.
\item[(iii)] Using the non-negativity of $u_c$ and some subtle $L^1$-type identity ($d=1$) and the Liouville-type theorems from \cite{Liouville_type} ($d\geq 2$), we prove that the frequency $\beta$ in \eqref{standing wave} is positive.

\item[(iv)] We finally adapt an energetic method due to Le Coz \cite{LeCoz2008} into the context of normalized ground states to close the proof of Theorem \ref{thm existence of ground state}.
\end{itemize}
We also emphasize that the proof of Theorem \ref{thm existence of ground state} is almost purely energetic and does not make use of arguments involving a converging Palais-Smale sequence. However, we still need to apply certain deformation arguments for the proof of (i). Note that in the Euclidean case, the proof of (i) is a simple application of the Pohozaev identity (see for instance \cite[Lem. 3.1]{BellazziniJeanjean2016}). On the contrary, the terms $-\Delta_x u_c$ and $-\pt_y^2 u_c$ appearing in the Lagrange multiplier equation of $u_c$ as an optimizer of $m_c$ have generally different prefactors, thus \textit{a priori} we do not know whether the underlying Lagrange multiplier equation is elliptic and the Pohozaev identity is inapplicable\footnote{To be more precise, the ellipticity principally implies the higher regularity of $u_c$, which being necessary in a standard proof of Pohozaev identity.}. Nevertheless, by a straightforward modification of the proof of Theorem \ref{thm existence of ground state} we may alternatively give a more fundamental and purely energetic proof for the existence results in the Euclidean case, which requires no prerequisite knowledge on the min-max theory and might be of independent interest.

Borrowing an idea from Terracini-Tzvetkov-Visciglia \cite{TTVproduct2014} we obtain the following result concerning the $y$-dependence of the ground state optimizers $u_c$.

\begin{theorem}[Ground states with non-trivial $y$-dependence]\label{thm threshold mass}
Let $m_c$ and $\wm_c$ be the quantities defined by \eqref{mc first def} and \eqref{wmc first def} respectively. Then there exists some $c_*\in(0,\infty)$ such that
\begin{itemize}
\item For all $c \in (0,c_*)$ we have $m_{c}<2\pi \wm_{(2\pi)^{-1}c}$. Moreover, for $c\in(0,c_*)$ any minimizer $u_c$ of $m_{c}$ satisfies $\pt_y u_c\neq 0$.

\item For all $c\in[c_*,\infty)$ we have $m_{c}=2\pi \wm_{(2\pi)^{-1}c}$. Moreover, for $c\in(c_*,\infty)$ any minimizer $u_c$ of $m_{c}$ satisfies $\pt_y u_c=0$.
\end{itemize}
\end{theorem}

The proof of Theorem \ref{thm threshold mass} follows the same steps in \cite{TTVproduct2014}. We loosely describe the underlying strategy: for $\ld>0$ we introduce the modified energy $\mH_{\ld}(u)$ defined by \eqref{def modified energy} and consider the minimization problem
\begin{align*}
m_{1,\ld}:=\inf_{u\in H^1(\R^d\times\T)}\{\mH_{\ld}(u):\mM(u)=1,\,\mK(u)=0\}.
\end{align*}
Proceeding as in the proof of Theorem \ref{thm existence of ground state} we know that for any $\ld>0$ the minimization problem $m_{1,\ld}$ has a positive optimizer $u_\ld$. Using certain perturbation arguments we prove the crucial statement
\begin{align*}
\lim_{\ld\to \infty}m_{1,\ld}=2\pi\wm_{(2\pi)^{-1}},\quad\lim_{\ld\to 0}m_{1,\ld}<2\pi\wm_{(2\pi)^{-1}},
\end{align*}
from which we obtain a first result concerning the $y$-dependence of the minimizers $u_\ld$ in dependence of the size of $\ld$. Theorem \ref{thm threshold mass} then follows from a simple rescaling argument. Although the arguments for the proof of Theorem \ref{thm threshold mass} are very similar to the ones in \cite{TTVproduct2014}, the additional semivirial-vanishing condition $\mK(u)=0$ complicates the proof at several places. Particularly, we shall construct a different and more complicated test function than the one in \cite{TTVproduct2014} in order to prove the statement $\lim_{\ld\to 0}m_{1,\ld}<2\pi\wm_{(2\pi)^{-1}}$.

Finally, we prove that the ground states deduced from Theorem \ref{thm existence of ground state} characterize a sharp threshold for the bifurcation of scattering and finite time blow-up solutions.

\begin{theorem}[Scattering below ground states]\label{main thm}
Let $u$ be a solution of \eqref{nls}. If $\mH(u)<m_{\mM(u)}$ and $\mK(u(0))>0$, then $u$ is a global solution of \eqref{nls}. If additionally $d\leq 4$, then $u$ also scatters in time in the sense that there exist $\phi^\pm\in H^1(\R^d\times\T)$ such that
\begin{align}
\lim_{t\to\pm\infty}\|u(t)-e^{it\Delta_{x,y}}\phi^{\pm}\|_{H^1(\R^d\times\T)}=0.
\end{align}
\end{theorem}

\begin{theorem}[Finite time blow-up below ground states]\label{thm blow up}
Let $u$ be a solution of \eqref{nls}. If $|x|u(0)\in L^2(\R^d\times\T)$, $\mH(u)<m_{\mM(u)}$ and $\mK(u(0))<0$, then $u$ blows-up in finite time.
\end{theorem}

The proofs of Theorem \ref{main thm} and \ref{thm blow up} rely on the classical concentration compactness arguments initiated by Kenig and Merle \cite{KenigMerle2006} and the virial arguments by Glassey \cite{Glassey1977} respectively. The restriction $d\leq 4$ given in the scattering result lies in the fact that the $y$-derivative of the nonlinear potential $|u|^\alpha u$ is no longer Lipschitz in $d\geq5$, hence the proof of the stability result (Lemma \ref{lem stability cnls}) does not work in higher dimensions ($d\geq 5$). We believe that with help of the Lebesgue-type Strichartz estimates given in \cite{HaniPausader,Barron,BarronChristPausader2021} the scattering result should extend to all dimensions $d\geq 1$. The adaption is however technical and out of the scope of the present paper. We decide to leave this problem open and plan to tackle it in a forthcoming paper.

\subsubsection*{Outline of the paper}
The rest of the paper is organized as follows: In Section \ref{sec: Existence} we give the proof of Theorem \ref{thm existence of ground state}. In Section \ref{sec: Dependence} we prove Theorem \ref{thm threshold mass}. Finally, Theorem \ref{main thm} and \ref{thm blow up} are shown in Section \ref{sec: Scattering}.

\subsection{Notation and definitions}
We use the notation $A\lesssim B$ whenever there exists some positive constant $C$ such that $A\leq CB$. Similarly we define $A\gtrsim B$ and we use $A\sim B$ when $A\lesssim B\lesssim A$.

For simplicity, we ignore in most cases the dependence of the function spaces on their underlying domains and hide this dependence in their indices. For example $L_x^2=L^2(\R^d)$, $H_{x,y}^1= H^1(\R^d\times \T)$
and so on. However, when the space is involved with time, we still display the underlying temporal interval such as $L_t^pL_x^q(I)$, $L_t^\infty L_{x,y}^2(\R)$ etc. The norm $\|\cdot\|_p$ is defined by $\|\cdot\|_p:=\|\cdot\|_{L_{x,y}^p}$.

Next, we define the quantities such as mass and energy etc. that will be frequently used in the proof of the main results. For $u\in H_{x,y}^1$, define
\begin{align}
\mM(u)&:=\|u\|^2_{2},\label{def of mass}\\
\mH(u)&:=\frac{1}{2}\|\nabla_{x,y} u\|^2_{2}-\frac{1}{\alpha+2}\|u\|^{\alpha+2}_{\alpha+2},\label{def of mhu}\\
\mK(u)&:=\|\nabla_{x} u\|^2_{2}-\frac{\alpha d}{2(\alpha+2)}\|u\|^{\alpha+2}_{\alpha+2},\\
\mI(u)&:=\frac{1}{2}\|\pt_y u\|_{2}^2+\frac{\alpha d-4}{4(\alpha+2)}\|u\|^{\alpha+2}_{\alpha+2}=\mH(u)-\frac{1}{2}
\mK(u).\label{def of mI}
\end{align}
For $\ld\in(0,\infty)$, define
\begin{align}\label{def modified energy}
\mH_{\ld}(u)&:=\frac{\ld}{2}\|\nabla_{y} u\|^2_{2}+\frac{1}{2}\|\nabla_{x} u\|^2_{2}
-\frac{1}{\alpha+2}\|u\|^{\alpha+2}_{\alpha+2},\\
\mI_{\ld}(u)&:=\frac{\ld}{2}\|\pt_y u\|_{2}^2+\frac{\alpha d-4}{4(\alpha+2)}\|u\|^{\alpha+2}_{\alpha+2}.\label{def of I ld}
\end{align}
For $u\in H_x^1$, define
\begin{align}
\wmM(u)&:=\|u\|^2_{L_x^2},\\
\wmH(u)&:=\frac{1}{2}\|\nabla_{x} u\|^2_{L_x^2}-\frac{1}{\alpha+2}\|u\|^{\alpha+2}_{L_x^{\alpha+2}},\\
\wmI(u)&:=\frac{\alpha d-4}{4(\alpha+2)}\|u\|^{\alpha+2}_{L_x^{\alpha+2}},\label{def of wmI}\\
\wmK(u)&:=\|\nabla_{x} u\|^2_{L_x^2}-\frac{\alpha d}{2(\alpha+2)}\|u\|^{\alpha+2}_{L_x^{\alpha+2}}
\end{align}
We also define the sets
\begin{align}
S(c)&:=\{u\in H_{x,y}^1:\mM(u)=c\},\\
V(c)&:=\{u\in S(c):\mK(u)=0\},\\
\widehat{S}(c)&:=\{u\in H_x^1:\wmM(u)=c\},\\
\widehat{V}(c)&:=\{u\in \widehat{S}(c):\wmK(u)=0\}
\end{align}
and the variational problems
\begin{align}
m_c&:=\inf\{\mH(u):u\in V(c)\},\label{def of mc}\\
m_{1,\ld}&:=\inf\{\mH_\ld(u):u\in V(1)\},\label{def of auxiliary problem}\\
\wm_c&:=\inf\{\wmH(u):u\in \widehat{V}(c)\}\label{def of wmc}.
\end{align}
Finally, for a function $u\in H_{x,y}^1$, the scaling operator $u\mapsto u^t$ for $t\in(0,\infty)$ is defined by
\begin{align}\label{def of scaling op}
u^t(x,y):=t^{\frac d2}u(tx,y).
\end{align}
We shall also frequently use results arising from the variational problem $\wm_c$, which we summarize in the following. The results are classical and we refer e.g. to \cite{Cazenave2003,Jeanjean1997,Bellazzini2013,BellazziniJeanjean2016} for details of the corresponding proofs.

\begin{lemma}\label{lem wmc property}
The following statements hold true:
\begin{itemize}
\item[(i)]For any $c>0$ the variational problem $\wm_c$ has an optimizer $P_c\in \widehat{S}(c)$. Moreover, $P_c$ satisfies the standing wave equation
\begin{align}\label{standing wave on rd}
-\Delta_x P_c+\omega_c P=|P_c|^\alpha P_c
\end{align}
with some $\omega_c>0$.
\item[(ii)] Any solution $P_c\in H^1(\R^d)$ of \eqref{standing wave on rd} with $\omega_c>0$ is of class $W^{3,p}(\R^d)$ for all $p\in[2,\infty)$.
\item[(iii)] Any solution $P_c\in H^1(\R^d)$ of \eqref{standing wave on rd} satisfies $\wmK(P_c)=0$.
\item[(iv)] The mapping $c\mapsto \wm_c$ is strictly monotone decreasing and continuous on $(0,\infty)$.
\end{itemize}
\end{lemma}
We also record a very useful lemma from \cite{TTVproduct2014} (see the fourth step of the proof of \cite[Thm. 1.1]{TTVproduct2014}), which plays an important role for identifying a non-vanishing weak limit of a minimizing sequence. The original proof was given for $\alpha$ in the mass-subcritical regime $(0,\frac{4}{d+1})$, which extends verbatim to the whole energy-subcritical regime $(0,\frac{4}{d-1})$.

\begin{lemma}\label{lemma non vanishing limit}
Let $(u_n)_n$ be a bounded sequence in $H_{x,y}^1$. Assume also that there exists some $\alpha\in(0,\frac{4}{d-1})$ such that
\begin{align}
\liminf_{n\to\infty}\|u_n\|_{\alpha+2}>0.
\end{align}
Then there exist $(x_n)_n\subset \R^d$ and some $u\in H_{x,y}^1\setminus\{0\}$ such that up to a subsequence
\begin{align}
u_n(x+x_n,y)\rightharpoonup u(x,y)\quad\text{weakly in $H_{x,y}^1$}.
\end{align}
\end{lemma}

Next we introduce the concept of an \textit{admissible} pair on $\R^d$. A pair $(q,r)$ is said to be $\dot{H}^s$-admissible for $s\in[0,\frac d2)$ if $q,r\in[2,\infty]$, $\frac{2}{q}+\frac{d}{r}=\frac{d}{2}-s$ and $(q,d)\neq(2,2)$. For any $L^2$-admissible pairs $(q_1,r_1)$ and $(q_2,r_2)$ we have the following Strichartz estimate: if $u$ is a solution of
\begin{align*}
i\pt_t u+\Delta_x u=F
\end{align*}
on $I\subset\R$ with $t_0\in I$ and $u(t_0)=u_0$, then
\begin{align}
\|u\|_{L_t^q L_x^r(I)}\lesssim \|u_0\|_{L_x^2}+\|F\|_{L_t^{q_2'} L_x^{r_2'}(I)},
\end{align}
where $(q_2',r_2')$ is the H\"older conjugate of $(q_2,r_2)$. For a proof, we refer to \cite{EndpointStrichartz,Cazenave2003}. For $d\geq 3$, we define the space $S_x$ by
\begin{align}
S_x:=L_t^\infty L_x^2\cap L^{2}_t L_x^{\frac{2d}{d-2}}\label{def of sx}.
\end{align}
For $d\in\{1,2\}$, the space $L^{2}_t L_x^{\frac{2d}{d-2}}$ in the definition of $S_x$ is replaced by
\begin{align*}
d=1:\quad L^{4}_t L_x^{\infty}\quad\text{and}\quad
d=2:\quad L^{2^+}_t L_x^{\infty^-}
\end{align*}
respectively, where $(2^{+},\infty^-)$ is an $L^2$-admissible pair with some $2^+\in(2,\infty)$ sufficiently close to $2$. In the following, an admissible pair is always referred to as an $L^2$-admissible pair if not otherwise specified.

Finally, we define the Fourier transformation of a function $f$ w.r.t. $x\in\R^d$ or $y\in\T$ by
\begin{gather*}
\hat{f}_y(x,k)=\mathcal{F}_yf(x,k):=(2\pi)^{-\frac{1}{2}}\int_{\T}f(x,y)e^{-iky}\,dy,\\
\hat{f}_x(\xi,y)=\mathcal{F}_xf(\xi,y):=(2\pi)^{-\frac{d}{2}}\int_{\R^d}f(x,y)e^{-i\xi\cdot x}\,dx,\\
\hat{f}_{\xi,k}(\xi)=\mathcal{F}_{x,y}f(\xi,k):=(2\pi)^{-\frac{d+1}{2}}\int_{\R^d\times\T}f(x,y)e^{-i(\xi\cdot x+ky)}\,dxdy.
\end{gather*}
\section{Existence of normalized ground states}\label{sec: Existence}

\subsection{A scale-invariant Gagliardo-Nirenberg inequality on $\R^d\times\T$}
We establish in the following a scale-invariant (w.r.t. $x$-variable) Gagliardo-Nirenberg inequality on $\R^d\times\T$ which plays the same role as the conventional scale-invariant Gagliardo-Nirenberg inequalities on $\R^d$. As we shall see in the upcoming proofs, such scale-invariant Gagliardo-Nirenberg inequality serves as a substitute for the standard inhomogeneous ones and is more useful for identifying a non-vanishing weak limit of a minimizing sequence.
\begin{lemma}[Scale-invariant Gagliardo-Nirenberg inequality on $\R^d\times\T$]\label{lemma gn additive}
There exists some $C>0$ such that for all $u\in H_{x,y}^1$ we have
\begin{align}
\|u\|_{\alpha+2}^{\alpha+2}\leq C\|\nabla_x u\|_2^{\frac{\alpha d}{2}}\|u\|_2^{\frac{4-\alpha(d-1)}{2}}
(\| u\|_{2}^{\frac{\alpha}{2}}+\|\pt_y u\|_{2}^{\frac{\alpha}{2}})
\end{align}
\end{lemma}

\begin{proof}
For a function $u$ define $m(u):=(2\pi)^{-1}\int_{\T}u(y)\,dy$. By triangular inequality we obtain
\begin{align}\label{gn ineq 1}
\|u\|_{\alpha+2}\leq \|m(u)\|_{\alpha+2}+\|u-m(u)\|_{\alpha+2}.
\end{align}
We then estimate $m(u)$ and $u-m(u)$ separately. Since $m(u)$ is independent of $y$, using the standard Gagliardo-Nirenberg inequality on $\R^d$ and Jensen's inequality we conclude that
\begin{align}\label{gn ineq 2}
\|m(u)\|_{\alpha+2}\lesssim \|m(u)\|_{L_x^{\alpha+2}}\lesssim \|m(u)\|_{L_x^2}^{1-\frac{\alpha d}{2(\alpha+2)}}
\|\nabla_x m(u)\|_{L_x^2}^{\frac{\alpha d}{2(\alpha+2)}}
\lesssim \|u\|_2^{1-\frac{\alpha d}{2(\alpha+2)}}\|\nabla_x u\|_2^{\frac{\alpha d}{2(\alpha+2)}}.
\end{align}
For $u-m(u)$, we recall the following well-known Sobolev's inequality on $\T$ for functions with zero mean (see for instance \cite{sobolev_torus}):
\begin{align}\label{sobolev torus 1}
\|u-m(u)\|_{L_y^{\alpha+2}}\lesssim\|u\|_{\dot{H}_y^{\frac{\alpha}{2(\alpha+2)}}}.
\end{align}
Writing $u$ into the Fourier series $u(x,y)=\sum_{k} e^{iky}u_k(x)$ w.r.t $y$ and followed by \eqref{sobolev torus 1}, Minkowski, Gagliardo-Nirenberg on $\R^d$ and H\"older (through the identity
$$1=\frac{\alpha}{2(\alpha+2)}+\frac{\alpha+4-\alpha d}{2(\alpha+2)}+\frac{\alpha d}{2(\alpha+2)})$$
we obtain
\begin{align}\label{gn ineq 3}
&\,\|u-m(u)\|_{\alpha+2}\nonumber\\
\lesssim&\,\bg(\int_{\R^d}\bg(\sum_{k\in\Z}|k|^{\frac{\alpha}{\alpha+2}}|u_k(x)|^2
\bg)^{\frac{\alpha+2}{2}}\,dx\bg)^{\frac{1}{\alpha+2}}\nonumber\\
\lesssim&\, \bg(\sum_{k\in\Z}|k|^{\frac{\alpha}{\alpha+2}}\|u_k\|_{L_x^{\alpha+2}}^2\bg)^{\frac12}
\nonumber\\
\lesssim&\,\bg(\sum_{k\in\Z}|k|^{\frac{\alpha}{\alpha+2}}\|u_k\|_{L_x^2}^{2-\frac{\alpha d}{\alpha+2}}
\|\nabla_x u_k\|_{L_x^2}^{\frac{\alpha d}{\alpha+2}}\bg)^{\frac12}\nonumber\\
\lesssim&\,\|(\|u_k\|_{L_x^2})_k\|_{\ell_k^2}^{\frac{4-\alpha(d-1)}{2(\alpha+2)}}
\|(k \|u_k\|_{L_x^2})_k\|_{\ell_k^2}^{\frac{\alpha}{2(\alpha+2)}}
\|(\|\nabla_xu_k\|_{L_x^2})_k\|_{\ell_k^2}^{\frac{\alpha d}{2(\alpha+2)}}\nonumber\\
=&\,\|u\|_{2}^{\frac{4-\alpha(d-1)}{2(\alpha+2)}}
\|\pt_y u\|_{2}^{\frac{\alpha}{2(\alpha+2)}}
\|\nabla_x u\|_{2}^{\frac{\alpha d}{2(\alpha+2)}}.
\end{align}
The desired result follows from summing \eqref{gn ineq 1}, \eqref{gn ineq 2} and \eqref{gn ineq 3} up.
\end{proof}

The following result is an immediate consequence of Lemma \ref{lemma gn additive}.
\begin{corollary}\label{cor lower bound}
For any $c\in(0,\infty)$ we have $m_c\in(0,\infty)$, where $m_c$ is defined by \eqref{def of mc}.
\end{corollary}

\begin{proof}
That $m_c<\infty$ follows from $V(c)\neq\varnothing$. Next, let  $(u_n)_n\subset V(c)$ such that $\mH(u_n)=m_c+o_n(1)$. Then
\begin{align*}
\infty>m_c+o_n(1)=\mH(u_n)=\mH(u_n)-\frac{2}{\alpha d}\mK(u_n)=\frac{1}{2}\|\pt_y u_n\|_2^2+\bg(\frac{1}{2}-\frac{2}{\alpha d}\bg)\|\nabla_x u_n\|_2^2.
\end{align*}
Combining with $\alpha>4/d$ (which in turn implies $\frac{1}{2}-\frac{2}{\alpha d}>0$) we infer that $(u_n)_n$ is a bounded sequence in $H_{x,y}^1$. Now by Lemma \ref{lemma gn additive} and the fact that $\alpha<4/(d-1)$ we deduce
\begin{align*}
\|\nabla_xu_n\|_2^2=\frac{\alpha d}{2(\alpha+2)}\|u_n\|_{\alpha+2}^{\alpha+2}
\lesssim \|\nabla_x u_n\|_2^{\frac{\alpha d}{2}}\|u_n\|_2^{\frac{4-\alpha(d-1)}{2}}
(\| u_n\|_{2}^{\frac{\alpha}{2}}+\|\pt_y u_n\|_{2}^{\frac{\alpha}{2}})\lesssim \|\nabla_x u_n\|_2^{\frac{\alpha d}{2}},
\end{align*}
which combining with $\alpha>4/d$ implies
\begin{align}
\liminf_{n\to\infty}\|u_n\|_{\alpha+2}^{\alpha+2}\sim\liminf_{n\to\infty}\|\nabla_x u_n\|_2^2>0.\label{key gn inq2}
\end{align}
Summing up, we obtain
\begin{align*}
m_c=\lim_{n\to\infty}\bg(\frac{1}{2}\|\pt_y u_n\|_2^2+\bg(\frac{1}{2}-\frac{2}{\alpha d}\bg)\|\nabla_x u_n\|_2^2\bg)
\geq \bg(\frac{1}{2}-\frac{2}{\alpha d}\bg)\liminf_{n\to\infty}\|\nabla_x u_n\|_2^2\gtrsim 1,
\end{align*}
which completes the proof.
\end{proof}

\subsection{Dynamical properties of the mappings $t\mapsto \mK(u^t)$ and $c\mapsto m_c$}
In the following, we prove some useful properties of the mappings $t\mapsto \mK(u^t)$ and $c\mapsto m_c$, where $u^t$ is the scaling operator defined through \eqref{def of scaling op}. These properties will play a central role in the proof of our main results.
\begin{lemma}[Property of the mapping $t\mapsto \mK(u^t)$]\label{monotoneproperty}
Let $c>0$ and $u\in S(c)$. Then the following statements hold true:
\begin{enumerate}
\item[(i)] $\frac{\partial}{\partial t}\mH(u^t)=t^{-1} Q(u^t)$ for all $t>0$.
\item[(ii)] There exists some $t^*=t^*(u)>0$ such that $u^{t^*}\in V(c)$.
\item[(iii)] We have $t^*<1$ if and only if $\mK(u)<0$. Moreover, $t^*=1$ if and only if $\mK(u)=0$.
\item[(iv)] Following inequalities hold:
\begin{equation*}
Q(u^t) \left\{
\begin{array}{lr}
             >0, &t\in(0,t^*) ,\\
             <0, &t\in(t^*,\infty).
             \end{array}
\right.
\end{equation*}
\item[(v)] $\mH(u^t)<\mH(u^{t^*})$ for all $t>0$ with $t\neq t^*$.
\end{enumerate}
\end{lemma}
\begin{proof}
(i) follows from direct calculation. Now define $y(t):= \frac{\partial}{\partial t}\mH(u^t)$. Then
\begin{align*}
y(t)&=t\|\nabla_x u\|_2^2-\frac{\alpha d}{2(\alpha+2)}t^{\frac{\alpha d}{2}-1}\|u\|_{\alpha+2}^{\alpha+2},\\
y'(t)&=\|\nabla_x u\|_2^2-\frac{2\alpha d(\alpha d-2)}{4(\alpha+2)}t^{\frac{\alpha d}{2}-2}\|u\|_{\alpha+2}^{\alpha+2}.
\end{align*}
Using $\alpha>4/d$ we infer that $y'(0)=\|\nabla_x u\|_2^2>0$, $y'(t)\to -\infty$ as $t\to\infty$ and $y'(t)$ is strictly monotone decreasing on $(0,\infty)$. Thus there exists a $t_0>0$ such that $y'(t)$ is positive on $(0,t_0)$ and negative on $(t_0,\infty)$. Consequently, we conclude that $y(t)$ has a zero at $t^*>t_0$, $y(t)$ is positive on $(0,t^*)$ and negative on $(t^*,\infty)$. (ii) and (iv) now follow from the fact
$$y(t)=\frac{\partial \mH(u^t)}{\partial t}=\frac{Q(u^t)}{t}.$$
For (iii), we first let $\mK(u)<0$. Then
$$0>\mK(u)=\frac{Q(u^1)}{1}=y(1),$$
which is only possible as long as $t^*<1$. Conversely, let $t^*<1$. Then using the fact that $y(t)$ is monotone decreasing on $(t^*,\infty)$ we obtain
$$\mK(u)=y(1)<y(t^*)<0. $$
This completes the proof of (iii). To see (v), integration by parts yields
$$ \mH(u^{t^*})=\mH(u^t)+\int_t^{t^*}y(s)\,ds.$$
Then (v) follows from the fact that $y(t)$ is positive on $(0,t^*)$ and $y(t)$ is negative on $(t^*,\infty)$.
\end{proof}

\begin{lemma}[Property of the mapping $c\mapsto m_c$]\label{monotone lemma}
The mapping $c\mapsto m_c$ is continuous and monotone decreasing on $(0,\infty)$.
\end{lemma}

\begin{proof}
The proof follows the same line in \cite{Bellazzini2013}. We firstly show that the function $f$ defined by
\begin{align*}
f(a,b):=\max_{t>0}\{at^2-bt^{\frac{\alpha d}{2}}\}
\end{align*}
is continuous on $(0,\infty)^2$. In fact, the global maxima can be calculated explicitly. Let
\begin{align*}
g(t,a,b):=at^2-bt^{\frac{\alpha d}{2}}
\end{align*}
and let $t^*\in (0,\infty)$ such that $\pt_t g(t^*,a,b)=0$. Then $t^*=\bg(\frac{4a}{\alpha d b}\bg)^{\frac{2}{\alpha d-4}}$. Particularly, $\pt_t g(t,a,b)$ is positive on $(0,t^*)$ and negative on $(t^*,\infty)$. Thus
$$ f(a,b)=g(t^*,a,b)=\bg(\bg(\frac{4}{\alpha d}\bg)^{\frac{4}{\alpha d-4}}
-\bg(\frac{4}{\alpha d}\bg)^{\frac{\alpha d}{\alpha d-4}}\bg)a^{\frac{\alpha d}{\alpha d-4}}b^{-\frac{4}{\alpha d-4}}
$$
and we conclude the continuity of $f$ on $(0,\infty)^2$.

We now show the monotonicity of $c\mapsto m_c$. It suffices to show that for any $0<c_1<c_2<\infty$ and $\vare>0$ we have
\begin{align*}
m_{c_2}\leq m_{c_1}+\vare.
\end{align*}
By the definition of $m_{c_1}$ there exists some $u_1\in V(c_1)$ such that
\begin{align}\label{pert 2}
\mH(u_1)\leq m_{c_1}+\frac{\vare}{2}.
\end{align}
Let $\eta\in C^{\infty}_c(\R^d;[0,1])$ be a cut-off function such that $\eta=1$ for $|x|\leq 1$ and $\eta=0$ for $|x|\geq 2$. For $\delta>0$, define
\begin{equation*}
\tilde{u}_{1,\delta}(x,y):= \eta(\delta x)\cdot u_1(x,y).
\end{equation*}
Using dominated convergence theorem it is easy to verify that $\tilde{u}_{1,\delta}\to u_1$ in $H_{x,y}^1$ as $\delta\to 0$. Therefore,
\begin{align*}
\|\nabla_{x,y}\tilde{u}_{1,\delta}\|_2&\to \|\nabla_{x,y} u_1\|_2, \\
\|\tilde{u}_{1,\delta}\|_p&\to \| u_1\|_p
\end{align*}
for all $p\in[2,2+\frac{4}{d-1})$ as $\delta\to 0$. Combining with the continuity of $f$ we conclude that
\begin{align}\label{pert 1}
\max_{t>0}\mH(\tilde{u}^t_{1,\delta})
&=\max_{t>0}\bg\{\frac{t^2}{2}\|\nabla_{x}\tilde{u}_{1,\delta}\|_2^2-\frac{t^{\frac{\alpha d}{2}}}{\alpha+2}\|\tilde{u}_{1,\delta}\|_{\alpha+2}^{\alpha+2}
\bg\}+\frac{1}{2}\|\pt_y \tilde{u}_{1,\delta}\|_2^2\nonumber\\
&\leq \max_{t>0}\bg\{\frac{t^2}{2}\|\nabla_{x}u_1\|_2^2-\frac{t^{\frac{\alpha d}{2}}}{\alpha+2}\|u_1\|_{\alpha+2}^{\alpha+2}
\bg\}+\frac{1}{2}\|\pt_y u_1\|_2^2+\frac{\vare}{4}=\max_{t>0}\mH( u^t_1)+\frac{\vare}{4}
\end{align}
for sufficiently small $\delta>0$. Now let $v\in C_c^\infty(\R^d)$ with $\mathrm{supp}\,v\subset
B(0,4\delta^{-1}+1)\backslash B(0,4\delta^{-1})$ and define
\begin{equation*}
v_0:= \frac{(c_2-\mM(\tilde{u}_{1,\delta}))^{\frac{1}{2}}}{\mM(v)^{\frac{1}{2}}}\,v.
\end{equation*}
Notice that $v_0$ and $\tilde{u}_{1,\delta}$ have compact supports, which also implies $\mM(v_0)=c_2-\mM(\tilde{u}_{1,\delta})$. Define
\begin{align*}
w_\ld:=\tilde{u}_{1,\delta}+v_0^\ld
\end{align*}
with some to be determined $\ld>0$. Then
\begin{align*}
\|w_\ld\|^p_p=\|\tilde{u}_{1,\delta}\|^p_p+\| v_0^\ld\|^p_p
\end{align*}
for all $p\in [2,2+\frac{4}{d-1})$. Particularly, $\mM(w_\ld)=c_2$. Since $v_0$ is independent of $y\in\T$, we also infer that
\begin{align*}
\|\nabla_{x} w_\ld\|_2&\to\|\nabla_{x} \tilde{u}_{1,\delta}\|_2,\\
\|\pt_y w_\ld\|_2&=\|\pt_y \tilde{u}_{1,\delta}\|_2,\\
\| w_\ld\|_p&\to\| \tilde{u}_{1,\delta}\|_p
\end{align*}
for all $p\in(2,2+\frac{4}{d-1})$ as $\ld\to 0$. Using the continuity of $f$ once again we obtain
\begin{align*}
\max_{t>0}\mH(w^t_\ld)\leq \max_{t>0}\mH(\tilde{u}^t_{1,\delta})+\frac{\vare}{4}
\end{align*}
for sufficiently small $\ld>0$. Finally, combing with \eqref{pert 2} and \eqref{pert 1} we infer that
\begin{align*}
m_{c_2}\leq \max_{t>0}\mH(w^t_\lambda)\leq \max_{t>0}\mH(\tilde{u}^t_{1,\delta})+\frac{\varepsilon}{4}\leq
\max_{t>0}\mH(u^t_1)+\frac{\varepsilon}{2}=\mH(u_1)+\frac{\varepsilon}{2}\leq m_{c_1}+\varepsilon,
\end{align*}
which implies the monotonicity of $c\mapsto m_c$ on $(0,\infty)$.

Next, we show the continuity of the curve $c\mapsto m_c$. Since $c\mapsto m_c$ is non-increasing, it suffices to show that for any $c\in(0,\infty)$ and any sequence $c_n\downarrow c$ we have
\begin{align*}
m_c\leq \lim_{n\to \infty}m_{c_n}.
\end{align*}
Let $\vare>0$ be an arbitrary positive number. By the definition of $m_{c_n}$ we can find some $u_n\in V(c_n)$ such that
\begin{align}\label{pert 3}
\mH(u_n)\leq m_{c_n}+\frac{\vare}{2}\leq m_c+\frac{\vare}{2}.
\end{align}
We define $\tilde{u}_n=(c_n^{-1}c)^{\frac{1}{2}} \cdot u_n:=\rho_n u_n$. Then $\mM(\tilde{u}_n)=c$ and $\rho_n\uparrow 1$. Since $u_n\in V(c_n)$, we obtain
\begin{align*}
m_{c}+\frac{\vare}{2}\geq m_{c_n}+\frac{\vare}{2}\geq \mH(u_n)=\mH(u_n)-\frac{2}{\alpha d}\mK(u_n)
=\frac{1}{2}\|\pt_y u_n\|_2^2+\bg(\frac{1}{2}-\frac{2}{\alpha d}\bg)\|\nabla_x u_n\|_2^2.
\end{align*}
Thus $(u_n)_n$ is bounded in $H^1_{x,y}$ and up to a subsequence we infer that there exist $A,B,C\geq 0$ such that
\begin{align*}
\|\nabla_xu_n\|_2^2=A+o_n(1),\quad\|\pt_y u_n\|_2^2=B+o_n(1),\quad\|u_n\|_{\alpha+2}^{\alpha+2}=C+o_n(1).
\end{align*}
Arguing as in the proof of Corollary \ref{cor lower bound}, we may use the fact $\mK(u_n)=0$ and Lemma \ref{lemma gn additive} to deduce $A,C>0$ and by previous arguments we know that $f$ is continuous at the point $(A,C)$. Using also the fact that $\rho_n\uparrow 1$ we conclude that
\begin{align*}
m_{c}&\leq \max_{t>0}\mH(\tilde{u}^t_n)
=\max_{t>0}\bg\{\frac{t^2\rho_n^2}{2}\|\nabla_x{u}_{n}\|_2^2
- \frac{t^{\frac{\alpha d}{2}}\rho_n^{\alpha+2}}{\alpha+2}\|{u}_{n}\|_{\alpha+2}^{\alpha+2}\bg\}
+\frac{\rho_n^2}{2}\|\pt_y u_n\|_2^2\nonumber\\
&\leq\max_{t>0}\bg\{\frac{t^2A}{2}
- \frac{t^{\frac{\alpha d}{2}}C}{\alpha+2}\bg\}
+\frac{1}{2}\|\pt_y u_n\|_2^2+\frac{\vare}{4}\nonumber\\
&\leq \max_{t>0}\bg\{\frac{t^2}{2}\|\nabla_x u_n\|_2^2
- \frac{t^{\frac{\alpha d}{2}}}{\alpha+2}\| u_n\|_{\alpha+2}^{\alpha+2}\bg\}
+\frac{1}{2}\|\pt_y u_n\|_2^2+\frac{\vare}{2}\nonumber\\
&=\max_{t>0}\mH(u^t_n)+\frac{\vare}{2}=\mH(u_n)+\frac{\vare}{2}\leq m_{c_n}+\vare
\end{align*}
by choosing $n$ sufficiently large. The continuity claim follows from the arbitrariness of $\vare$.
\end{proof}

\subsection{Characterization of an optimizer as a standing wave equation}
In this subsection we prove that any minimizer $u_c$ of $m_c$ is automatically a solution of \eqref{standing wave}. As already mentioned in the introductory section, since in general the prefactors of $-\Delta_x u_c$ and $-\pt_y^2 u_c$ do not coincide in the Lagrange multiplier equation of $u_c$, we do not know whether a simple proof via the Pohozaev identity is applicable. In this case, we appeal to a subtle deformation argument (see for instance \cite{WillemBook,Bellazzini2013} and specifically \cite{Bellazzini2013} in the context of normalized ground states) to overcome this difficulty.

First, we introduce the concept of \textit{mountain pass geometry}.
\begin{definition}[Mountain pass geometry]\label{definiton of mp geometry}
We say that $\mH(u)$ has a mountain pass geometry on $S(c)$ at the level $\gamma_c$ if there exists some $k>0$ and $\vare\in(0,m_c)$ such that
\begin{align}
\gamma_c:=\inf_{g\in\Gamma(c)}\max_{t\in[0,1]}\mH(g(t))>\max\{\sup_{g\in\Gamma(c)}\mH(g(0)),\sup_{g\in\Gamma(c)}\mH(g(1))\},\label{def of gammac}
\end{align}
where
\begin{align*}
\Gamma(c):=\{g\in C([0,1];S(c)):g(0)\in A_{k,\vare},\mH(g(1))<0\}
\end{align*}
and
\begin{align*}
A_{k,\vare}:=\{u\in S(c):\|\nabla_{x} u\|_2^2<k,\|\pt_y u\|_2^2 \leq 2(m_c-\vare)\}.
\end{align*}
\end{definition}

We next show that one obtains mountain geometry by choosing $k$ and $\vare$ sufficiently small.

\begin{lemma}\label{lemma existence mountain pass}
There exist $k>0$ and $\vare\in(0,m_c)$ such that
\begin{itemize}
\item[(i)]$m_c=\gamma_c$, where $m_c$ and $\gamma_c$ are defined by \eqref{def of mc} and \eqref{def of gammac} respectively.
\item[(ii)]$\mH(u)$ has a mountain pass geometry on $S(c)$ at the level $m_c$ in the sense of Definition \ref{definiton of mp geometry}.
\end{itemize}
\end{lemma}

\begin{proof}
We firstly prove that by choosing $k$ sufficiently small we have $\mK(u)>0$ for all $u\in A_{k,\vare}$, where $k$ is independent of the choice of  $\vare\in (0,m_c)$. Indeed, by Lemma \ref{lemma gn additive}, the fact that $\mM(u)=c$, $\|\pt_y u\|_2^2< 2m_c$ for $u\in A_{k,\vare}$ and $\alpha>4/d$ we obtain
\begin{align*}
\mK(u)=\frac{1}{2}\|\nabla_x u\|_2^2-\frac{\alpha d}{2(\alpha+2)}\|u\|_{\alpha+2}^{\alpha+2}
\geq\frac{1}{2}\|\nabla_x u\|_2^2-C\|\nabla_x u\|_2^{\frac{\alpha d}{2}}>0
\end{align*}
as long as $\|\nabla_x u\|_2^2\in(0,k)$ for some sufficiently small $k$. Next we construct the number $\vare$. Arguing as in \eqref{key gn inq2} we know that there exists some $\beta=\beta(c)>0$ such that if $(u_n)_n\subset V(c)$ is a minimizing sequence for $m_c$, then
\begin{align}
\liminf_{n\to\infty}\bg(\bg(\frac{1}{2}-\frac{2}{\alpha d}\bg)\|\nabla_x u_n\|_2^2\bg)\geq \beta.\label{lower half beta}
\end{align}
We may shrink $\beta$ further such that $\beta<4m_c$. Hence for any minimizing sequence $(u_n)_n$ we must have
\begin{align}\label{bounded mc sec2}
m_c+\frac{\beta}{4}\geq \frac{1}{2}\|\pt_y u_n\|_2^2+\bg(\frac{1}{2}-\frac{2}{\alpha d}\bg)\|\nabla_x u_n\|_2^2
>\frac{1}{2}\|\pt_y u_n\|_2^2+\frac{\beta}{2}
\end{align}
for all sufficiently large $n$. Now we set $\vare=\frac{\beta}{4}$. Therefore for $u\in A_{k,\vare}$, using also Lemma \ref{lemma gn additive} we infer that
\begin{align*}
\mH(u)\leq \frac{1}{2}\|\pt_y u\|_2^2+\frac{1}{2}\|\nabla_x u\|_2^2+C\|\nabla_x u\|_2^{\frac{\alpha d}{2}}
\leq m_c-\vare+\bg(\frac{1}{2}k+Ck^{\frac{\alpha d}{4}}\bg).
\end{align*}
We now choose $k=k(\vare)$ sufficiently small (without changing the fact that $Q(u)>0$ for $u\in A_{k,\vare}$) such that $\frac{1}{2}k+Ck^{\frac{\alpha d}{4}}<\frac{\vare}{2}$. Thus for this choice of $k$ and $\vare$, we have $\mH(u)<m_c-\vare/2<m_c$ for all $u\in A_{k,\vare}$ and by definition, (ii) follows immediately from (i).

It is therefore left to show (i). Let $(u_n)_n$ be the given minimizing sequence satisfying \eqref{lower half beta}. Let $u=u_n$ for some (to be determined) sufficiently large $n\in\N$. For any $\kappa\in(0,\beta/4)$ we can choose $n$ sufficiently large such that $\mH(u)\leq m_c+\kappa$ and $(\frac{1}{2}-\frac{2}{\alpha d})\|\nabla_x u\|_2^2\geq \beta/2$. Then by \eqref{bounded mc sec2} we know that $\|\pt_y u\|_2^2\leq 2(m_c-\vare)$ for all $\kappa\in(0,\beta/4)$. It is easy to check that $\|\pt_y(u^t)\|_2^2=\|\pt_y u\|_2^2$ for all $t\in(0,\infty)$ and $\|\nabla_x (u^t)\|_2^2=t^2\|\nabla_x u\|_2^2\to 0$ as $t\to 0$. We then fix some $t_0>0$ sufficiently small such that $\|\nabla_x(u^{t_0})\|_2^2<k$, which in turn implies that $(u^{t_0})\in A_{k,\vare}$. On the other hand,
\begin{align*}
\mH(u^t)=\frac{1}{2}\|\pt_y u\|_2^2 +\frac{t^2}{2}\|u\|_2^2-\frac{t^{\frac{\alpha d}{2}}}{\alpha+2}\|u\|_{\alpha+2}^{\alpha+2}\to -\infty
\end{align*}
as $t\to\infty$. We then fix some $t_1$ sufficiently large such that $\mH(u^t)<0$. Now define $g\in C([0,1];S(c))$ by
\begin{align}\label{curve}
g(t):=u^{t_0+(t_1-t_0)t}.
\end{align}
Then $g\in \Gamma(c)$. By definition of $\gamma_c$ and Lemma \ref{monotoneproperty} we have
\begin{align*}
\gamma_c\leq \max_{t\in[0,1]}\mH(g(t))=\mH(u)\leq m_c+\kappa.
\end{align*}
Since $\kappa$ can be chosen arbitrarily small, we conclude that $\gamma_c\leq m_c$. On the other hand, by our choice of $k$ we already know that for any $g\in \Gamma(c)$ we have $\mK(g(0))>0$. We now prove $\mK(g(1))<0$ for any $g\in \Gamma(c)$. Assume the contrary that there exists some $g\in\Gamma(c)$ such that $\mK(g(1))\geq 0$. Then
$$ 0>\mH(g(1))\geq \frac{1}{2}\|\nabla_x g(1)\|_2^2-\frac{1}{\alpha+2}\|g(1)\|_{\alpha+2}^{\alpha+2}
\geq \bg(\frac{1}{2}-\frac{2}{\alpha d}\bg)\|\nabla_x g(1)\|_2^2\geq 0,$$
a contradiction. Next, by continuity of $g$ there exists some $t\in(0,1)$ such that $\mK(g(t))=0$. Therefore
\begin{align*}
\max_{t\in[0,1]}\mH(g(t))\geq m_c.
\end{align*}
Taking infimum over $g\in\Gamma(c)$ we deduce $\gamma_c\geq m_c$, which completes the desired proof.
\end{proof}

\begin{remark}\label{technical reason}
By technical reason we will also shrink $\vare$ in the Definition \ref{definiton of mp geometry} if necessary such that $\vare\leq \frac{d}{2}-\frac{2}{\alpha}$. The purpose of this choice of $\vare$ will become clear in the upcoming proof of Lemma \ref{minimizer is solution}.
\end{remark}

Before we finally turn to the proof of Lemma \ref{minimizer is solution}, we shall firstly introduce some preliminary concepts given by \cite{Berestycki1983}. First recall that $S(c)$ is a submanifold of $H_{x,y}^1$ with codimension $1$. Moreover, the tangent space $T_uS(c)$ for a point $u\in S(c)$ is given by
$$ T_u S(c)=\{v\in H_{x,y}^1:\la u,v\ra_{L_{x,y}^2}=0\}.$$
We also denote the tangent bundle of $S(c)$ by $TS(c)$. Next, the energy functional $\mH|_{S(c)}$ restricted on $S(c)$ is a $C^1$-functional on $S(c)$ and for any $u\in S(c)$ and $v\in T_uS(c)$ we have
\begin{align*}
\la \mH'|_{S(c)}(u),v\ra =\la \mH'(u),v\ra .
\end{align*}
We use $\|\mH'|_{S(c)}(u)\|_*$ to denote the dual norm of $\mH'|_{S(c)}(u)$ in the cotangent space $(T_uS(c))^*$, i.e.
$$\|\mH'|_{S(c)}(u)\|_*:=\sup_{v\in T_u S(c),\,\|v\|_{H^1_{x,y}}\leq 1}|\la \mH'|_{S(c)}(u),v\ra|. $$
Let now
$$\tilde{S}(c):=\{u\in S(c):\mH'|_{S(c)}(u)\neq 0\}.$$
According to \cite[Lem. 4]{Berestycki1983} there exists a locally Lipschitz pseudo gradient vector field $Y:\tilde{S}(c)\to TS(c)$ such that $Y(u)\in T_uS(c)$,
\begin{align}\label{pseudo gradient vector}
\|Y(u)\|_{H_{x,y}^1}\leq 2\|\mH'|_{S(c)}(u)\|_*\quad\text{and}\quad \la \mH'|_{S(c)}(u),Y(u)\ra\geq \|\mH'|_{S(c)}(u)\|_*^2
\end{align}
for $u\in \tilde{S}(c)$.

Having all the preliminaries we are ready to prove the claimed statement.

\begin{lemma}\label{minimizer is solution}
For any $c>0$ an optimizer $u_c$ of $m_c$ is a solution of \eqref{standing wave} for some $\beta\in\R$.
\end{lemma}

\begin{proof}
We borrow an idea from the proof of \cite[Lem. 6.1]{Bellazzini2013} to show the claim. By Lagrange multiplier theorem we know that $u_c$ solves \eqref{standing wave} is equivalent to $\mH'|_{S(c)}(u)=0$. We hence assume the contrary $\|\mH'|_{S(c)}(u)\|_*\neq 0$, which implies that there exists some $\delta>0$ and $\mu>0$ such that
\begin{align}\label{gtr mu}
v\in B_{u_c}(3\delta)\Rightarrow\|\mH'|_{S(c)}(v)\|_*\geq \mu,
\end{align}
where $B_{u_c}(\delta):=\{v\in S(c):\|u-v\|_{H_{x,y}^1}\leq \delta\}$. Let $k$ and $\vare$ be given according to Lemma \ref{lemma existence mountain pass} and Remark \ref{technical reason}. Define
\begin{align*}
\vare_1&:=\frac{1}{4}\bg(m_c-\max\{\sup_{g\in\Gamma(c)}\mH(g(0)),\sup_{g\in\Gamma(c)}\mH(g(1))\}\bg),\\
\vare_2&:=\min\{\vare_1,m_c/4,\mu\delta/4\}.
\end{align*}
We now define the deformation mapping $\eta$ as follows: Let the sets $A,B$ and function $h:S(c)\to [0,\delta]$ be given by
\begin{align*}
A&:= S(c)\cap \mH^{-1}([m_c-2\vare_2,m_c+2\vare_2]),\\
B&:=B_{u_c}(2\delta)\cap \mH^{-1}([m_c-\vare_2,m_c+\vare_2]),\\
h(u)&:=\frac{\delta\,\mathrm{dist}(u,S(c)\setminus A)}{\mathrm{dist}(u,S(c)\setminus A)+\mathrm{dist}(u,B)}.
\end{align*}
Next, we define the pseudo gradient flow $W:S(c)\to H_{x,y}^1$ by
\begin{align*}
W(u):=
\left\{
\begin{array}{ll}
-h(u)\|Y(u)\|^{-1}_{H_{x,y}^1}Y(u),&\text{if $u\in\tilde{S}(c)$},\\
0,&\text{if $u\in S(c)\setminus \tilde{S}(c)$}.
\end{array}
\right.
\end{align*}
One easily verifies that $W$ is a locally Lipschitz function from $S(c)$ to $H_{x,y}^1$. Then by standard arguments (see for instance \cite[Lem. 6]{Berestycki1983}) there exists a mapping $\eta:\R\times S(c)\to S(c)$ such that $\eta(1,\cdot)\in C(S(c);S(c))$ and $\eta$ solves the differential equation
\begin{align*}
\frac{d}{dt}\eta(t,u)=W(\eta(t,u)),\quad\eta(0,u)=u
\end{align*}
for any $u\in S(c)$. We now claim that $\eta$ satisfies the following properties:
\begin{itemize}
\item[(i)]$\eta(1,v)=v$ if $v\in S(c)\setminus \mH^{-1}([m_c-2\vare_2,m_c+2\vare_2])$.
\item[(ii)]$\eta(1,\mH^{m_c+\vare_2}\cap B_{u_c}(\delta))\subset\mH^{m_c-\vare}$.
\item[(iii)]$\mH(\eta(1,v))\leq \mH(v)$ for all $v\in S(c)$.
\end{itemize}
Here, the symbol $\mH^\kappa$ denotes the set $\mH^\kappa:=\{v\in S(c):\mH(v)\leq \kappa\}$. For (i), by definition we see that $h(v)=0$, thus $\frac{d}{dt}\eta(t,v)|_{t=0}=W(v)=0$ and $\eta(t,v)\equiv \eta(0,v)=v$. For (iii), using \eqref{pseudo gradient vector} and the non-negativity of $h$ we obtain
\begin{align*}
\mH(\eta(1,v))&=\mH(v)+\int_0^1\frac{d}{ds}\mH(\eta(s,v))\,ds\nonumber\\
&=\mH(v)-\int_{s\in[0,1],\eta(s,v)\in \tilde{S}(c)}\la\mH'(\eta(s,v)),h(\eta(s,v))\|Y(\eta(s,v))\|^{-1}_{H_{x,y}^1}Y(\eta(s,v))\ra\,ds\nonumber\\
&\leq \mH(v)-\frac12\int_{s\in[0,1],\eta(s,v)\in \tilde{S}(c)}h(\eta(s,v))\|\mH'|_{S(c)}(\eta(s,v))\|_*\,ds\leq\mH(v).
\end{align*}
It is left to prove (ii). We first show that for $v\in \mH^{m_c+\vare_2}\cap B_{u_c}(\delta)$ one has $\eta(t,v)\in B_{u_c}(2\delta)$ for all $t\in[0,1]$. This follows from
\begin{align*}
\|\eta(t,v)-v\|_{H_{x,y}^1}=\|\int_{0}^t h(v)\|Y(v)\|^{-1}_{H_{x,y}^1}Y(v)\,ds\|_{H_{x,y}^1}
\leq th(v)\leq \delta.
\end{align*}
By \eqref{gtr mu} this implies particularly that $\|\mH'|_{S(c)}(\eta(t,v))\|_*\geq \mu$. Consequently, using \eqref{pseudo gradient vector}, $0\leq h\leq \delta$ and $\vare_2\leq \mu\delta/4$ we obtain
\begin{align*}
\mH(\eta(1,v))
&\leq \mH(v)-\frac12\int_{s\in[0,1],\eta(s,v)\in \tilde{S}(c)}h(\eta(s,v))\|\mH'|_{S(c)}(\eta(s,v))\|_*\,ds\nonumber\\
&\leq m_c+\vare_2-\frac{\mu\delta}{2}\leq m_c-\vare_2.
\end{align*}
Next, we recall the function $g$ defined by \eqref{curve} by setting $u=u_c$ therein. We claim that there exist $t_0\ll 1$ and $t_1\gg 1$ such that $g\in\Gamma(c)$. Indeed, from the proof of Lemma \ref{lemma existence mountain pass} it suffices to show $\|\pt_y u_c\|_2^2\leq 2(m_c-\vare)$. Define the scaling operator $T_\ld$ by
\begin{align}\label{def of t ld}
T_\ld u(x,y):=\ld^{\frac{2}{\alpha}}u(\ld x,y).
\end{align}
Then
\begin{align*}
\|T_\ld(\nabla_x u)\|_2^2&=\ld^{2+\frac4\alpha-d}\|\nabla_x u\|_2^2,\\
\|T_\ld u\|_{\alpha+2}^{\alpha+2}&=\ld^{2+\frac4\alpha-d}\|u\|_{\alpha+2}^{\alpha+2},\\
\mK(T_\ld u)&=\ld^{2+\frac4\alpha-d}\mK(u),\\
\|T_\ld (\pt_y u)\|_2^2&=\ld^{\frac4\alpha-d}\|\pt_y u\|_2^2,\\
\|T_\ld u\|_2^2&=\ld^{\frac4\alpha-d}\|u\|_2^2.
\end{align*}
Combining with Lemma \ref{monotone lemma} and the fact that $u_c$ is an optimizer of $m_c$ we infer that $\frac{d}{d\ld}\mH(T_\ld u_c)|_{\ld=1}\geq 0$, or  equivalently
\begin{align}\label{identity abcd}
\|\pt_y u_c\|_{2}^2\leq \frac{2\alpha +(4-\alpha d)}{2(\alpha+2)}\|u_c\|_{\alpha+2}^{\alpha+2}.
\end{align}
Hence
\begin{align*}
m_c=\mH(u_c)=\mI(u_c)=\frac{1}{2}\|\pt_y u_c\|_{2}^2+\frac{\alpha d-4}{4(\alpha+2)}\|u_c\|^{\alpha+2}_{\alpha+2}
\geq\frac{\alpha}{2\alpha+4-\alpha d}\|\pt_y u_c\|_2^2,
\end{align*}
where $\mI(u)$ is defined by \eqref{def of mI}. Combining with Remark \ref{technical reason} we obtain
\begin{align*}
\|\pt_y u_c\|_2^2\leq 2m_c\bg(1-\bg(\frac{d}{2}-\frac{2}{\alpha}\bg)\bg)\leq 2m_c(1-\vare),
\end{align*}
as desired. By (i) and Lemma \ref{lemma existence mountain pass} we know that $\eta(1,g(t))\in \Gamma(c)$. We shall finally prove
\begin{align*}
\max_{t\in[0,1]}\mH(\eta(1,g(t)))<m_c,
\end{align*}
which contradicts the characterization $m_c=\gamma_c$ deduced from Lemma \ref{lemma existence mountain pass} and closes the desired proof. Notice by definition of $g$ and Lemma \ref{monotoneproperty} we have $\mH(g(t))\leq \mH(u_c)=m_c$ for all $t\in[0,1]$. Thus only the following scenarios can happen:
\begin{itemize}
\item[(a)] $g(t)\in S(c)\setminus B_{u_c}(\delta)$. By (iii) and Lemma \ref{monotoneproperty} (v) we have
\begin{align*}
\mH(\eta(1,g(t)))\leq \mH(g(t))<\mH(u_c)=m_c.
\end{align*}

\item[(b)]$g(t)\in\mH^{m_c-\vare_2} $. By (iii) we have
\begin{align*}
\mH(\eta(1,g(t)))\leq \mH(g(t))\leq m_c-\vare_2<m_c.
\end{align*}

\item[(c)]$g(t)\in\mH^{-1}([m_c-\vare_2,m_c+\vare_2])\cap B_{u_c}(\delta) $. By (ii) we have
\begin{align*}
\mH(\eta(1,g(t)))\leq  m_c-\vare_2<m_c.
\end{align*}
\end{itemize}
This completes the desired proof.
\end{proof}

\subsection{Proof of Theorem \ref{thm existence of ground state}}
Having all the preliminaries we are in a position to prove Theorem \ref{thm existence of ground state}.

\begin{proof}[Proof of Theorem \ref{thm existence of ground state}]
We split our proof into four steps.
\subsubsection*{Step 1: A Le Coz characterization of $m_c$}
We firstly give a different and much handier characterization for $m_c$ due to Le Coz \cite{LeCoz2008} that is more useful for our analysis. Define
\begin{align}
\tilde{m}_c:=\inf\{\mI(u):u\in S(c),\mK(u)\leq 0\},\label{mtilde equal m}
\end{align}
where $\mI(u)$ is the energy functional defined by \eqref{def of mI}. We now prove $m_c=\tilde{m}_c$. Let $(u_n)_n\subset S(c)$ be a minimizing sequence for the variational problem $\tilde{m}_c$, i.e.
\begin{align}\label{le coz charac}
\mI(u_n)=\tilde{m}_c+o_n(1),\quad
\mK(u_n)\leq 0\quad\forall\,n\in\N.
\end{align}
By Lemma \ref{monotoneproperty} we know that there exists some $t_n\in(0,1]$ such that $\mK(u_n^{t_n})$ is equal to zero. Thus
\begin{align*}
m_c\leq \mH(u_n^{t_n})=\mI(u_n^{t_n})\leq \mI(u_n)=\tilde{m}_c+o_n(1).
\end{align*}
Sending $n\to\infty$ we infer that $m_c\leq \tilde{m}_c$. On the other hand,
\begin{align*}
\tilde{m}_c
\leq\inf\{\mI(u):u\in V(c)\}
=\inf\{\mH(u):u\in V(c)\}=m_c,
\end{align*}
which closes the proof of Step 1.

\subsubsection*{Step 2: Existence of a non-negative optimizer of $m_c$}
We next prove the existence of a non-negative optimizer of $m_c$ following the idea from \cite{Akahori2013}. A new difficulty arises from the fact that we need to guarantee the weak limit of a minimizing sequence to be an element of $S(c)$. In general, by weakly lower semicontinuity of norms it only follows that $u\in S(c_1)$ for some $c_1\in[0,c]$, which is obviously insufficient for our purpose. We shall appeal to Lemma \ref{monotone lemma} to overcome this difficulty. Let $(u_n)_n\subset S(c)$ be a minimizing sequence satisfying \eqref{le coz charac}. By diamagnetic inequality we know that the variational problem $\tilde{m}_{c}$ is stable under the mapping $u\mapsto |u|$, thus we may w.l.o.g. assume that $u_n\geq 0$. By Corollary \ref{cor lower bound} and Step 1 we know that $\tilde{m}_c<\infty$, hence
\begin{align*}
\infty>\tilde{m}_c\gtrsim \mI(u_n)=\frac{1}{2}\|\pt_y u\|_{2}^2+\frac{\alpha d-4}{4(\alpha+2)}\|u\|^{\alpha+2}_{\alpha+2}.
\end{align*}
Combining with $Q(u_n)\leq 0$ and $(u_n)\subset S(c)$ we conclude that $(u_n)_n$ is a bounded sequence in $H_{x,y}^1$. Now using \eqref{key gn inq2} and Lemma \ref{lemma non vanishing limit} we may find some $H_{x,y}^1\setminus\{0\}\ni u\geq 0$ such that $u_n\rightharpoonup u$ weakly in $H_{x,y}^1$. By weakly lower semicontinuity of norms we deduce
\begin{align}
\mM(u)=:c_1\in(0,c],\quad\mI(u)\leq \tilde{m}_c.\label{xia jie}
\end{align}
We next show $\mK(u)\leq 0$. Assume the contrary $\mK(u)>0$. By Brezis-Lieb lemma, $\mK(u_n)\leq 0$ and the fact that $L_{x,y}^2$ is a Hilbert space we infer that
\begin{align*}
\mM(u_n-u)&=c-c_1+o_n(1),\\
\mK(u_n-u)&\leq -\mK(u)+o_n(1).
\end{align*}
Therefore, for all sufficiently large $n$ we know that $\mM(u_n-u)\in(0,c)$ and $\mK(u_n-u)<0$. By Lemma \ref{monotoneproperty} we know that there exists some $t_n\in(0,1)$ such that $\mK((u_n-u)^{t_n})=0$. Consequently, Lemma \ref{monotone lemma}, Brezis-Lieb lemma and Step 1 yield
\begin{align*}
\tilde{m}_c\leq \mI((u_n-u)^{t_n})<\mI(u_n-u)=\mI(u_n)-\mI(u)+o_n(1)=\tilde{m}_c-\mI(u)+o_n(1).
\end{align*}
Sending $n\to\infty$ and using the non-negativity of $\mI(u)$ we obtain $\mI(u)=0$. This in turn implies $u=0$, which is a contradiction and thus $\mK(u)\leq 0$. If $\mK(u)<0$, then again by Lemma \ref{monotoneproperty} we find some $s\in(0,1)$ such that $\mK(u^s)=0$. But then using Lemma \ref{monotone lemma}, Step 1 and the fact $c_1\leq c$
\begin{align*}
\tilde{m}_{c_1}\leq \mI(u^s)<\mI(u)\leq \tilde{m}_c\leq \tilde{m}_{c_1},
\end{align*}
a contradiction. We conclude therefore $\mK(u)=0$

Thus $u$ is a minimizer of $m_{c_1}$. From Lemma \ref{minimizer is solution} we know that $u$ is a solution of \eqref{standing wave} and it remains to show that the corresponding $\beta$ in \eqref{standing wave} is positive and $\mM(u)=c$.

\subsubsection*{Step 3: Positivity of $\beta$}
First we prove that $\beta$ is non-negative. Testing \eqref{standing wave} with $u$ and followed by eliminating $\|\nabla_x u\|_2^2$ using $\mK(u)=0$ we obtain
\begin{align}
\|\pt_y u\|_2^2+\beta\mM(u)=\frac{2\alpha +(4-\alpha d)}{2(\alpha+2)}\|u\|_{\alpha+2}^{\alpha+2}.\label{corrected1}
\end{align}
Combining with \eqref{identity abcd} we infer that $\beta\mM(u)\geq 0$. Since $u\neq 0$ we conclude $\beta\geq 0$. We next show that $\beta=0$ leads to a contradiction, which completes the proof of Step 3. Assume therefore that $u$ satisfies the equation
\begin{align}\label{liouville}
-\Delta_{x,y}u=u^{\alpha+1}.
\end{align}
First consider the case $d\geq 2$. By the Brezis-Kato estimate \cite{BrezisKato} (see also \cite[Lem. B.3]{Struwe1996}) and the local $L^p$-elliptic regularity (see for instance \cite[Lem. B.2]{Struwe1996}) we know that $u\in W^{2,p}_{\rm loc}(\R^{d+1})$ for all $p\in[1,\infty)$. Hence by Sobolev embedding we also know that $u$ and $\nabla u$ are of class $L^\infty_{\rm loc}(\R^{d+1})$. Taking $\pt_{j}$ to \eqref{liouville} with $j\in\{1,\cdots,d+1\}$ we obtain
\begin{align*}
-\Delta_{x,y}\pt_j u=(\alpha+1)u^{\alpha}\pt_j u\in L^\infty_{\rm loc}(\R^{d+1}).
\end{align*}
Hence by applying the local $L^p$-elliptic regularity again we deduce $u\in W^{3,p}_{\rm loc}(\R^{d+1})$ for all $p\in[1,\infty)$. Consequently, by Sobolev embedding we infer that $u\in C^2(\R^{d+1})$. But then by \cite[Thm. 1]{Liouville_type} we know that any nonnegative $C^2$-solution of \eqref{liouville} must be zero, a contradiction.

Next we consider the case $d=1$. For $n\in\N$ let $\phi_n\in C_c^\infty(\R;[0,1])$ be a radially symmetric decreasing function such that $\phi_n(t)\equiv 1$ on $|t|\leq n$, $\phi_n(t)\equiv 0$ for $|t|\geq n+1$ and $\sup_{n\in\N}\|\phi_n\|_{L_t^\infty}\lesssim 1$. Define also
$$ D_n:=\{x\in\R:|x|\in(n,n+1)\}.$$
Since $u\neq 0$ is non-negative, by monotone convergence theorem we know that
\begin{align*}
\int_{\R\times\T}u^{\alpha+1}\phi_n\,dxdy>0
\end{align*}
for all $n\gg 1$. On the other hand, using the fact that $\mathrm{supp}\,\nabla_x\phi_n\subset D_n$ and H\"older we see that
\begin{align}
\int_{\R\times \T}\nabla_x u\cdot\nabla_x \phi_n\,dxdy
&\leq \|\nabla_x u\|_{L^{2}(D_n\times\T)}\|\nabla_x \phi_n\|_{L^2(D_n\times\T)}\nonumber\\
&\lesssim \|\nabla_x u\|_{L^{2}(D_n\times\T)}((n+1)-n)^{\frac12}\lesssim\|\nabla_x u\|_{L^{2}(D_n\times\T)}.\label{contra1}
\end{align}
But then
$\infty>\|\nabla_x u\|_{2}^{2}\geq \sum_{n\geq 1}\|\nabla_x u\|^{2}_{L^{2}(D_n\times\T)}$
yields $\|\nabla_x u\|_{L^{2}(D_n\times\T)}=o_n(1)$. By the fact that $\phi_n$ is independent of $y$ we also know $\int_{\R\times\T}(-\pt_y^2 u)\phi_n\,dxdy=0$.
Summing up, by testing \eqref{standing wave} with $\phi_n$ and rearranging terms we obtain
\begin{align*}
0=\int_{\R\times\T}u^{\alpha+1}\phi_n\,dxdy+o_n(1)>0
\end{align*}
for $n$ sufficiently large, a contradiction. This completes the proof of Step 3.

\subsubsection*{Step 4: $\mM(u)=c$ and conclusion}
Finally, we prove $\mM(u)=c$. Assume therefore $c_1<c$. By Lemma \ref{monotone lemma} and \eqref{xia jie} we know that $m_{c_1}$ is a local minimizer of the mapping $c\mapsto m_c$, which in turn implies that the inequality in \eqref{identity abcd} is in fact an equality. Now using \eqref{identity abcd} (as an equality) and \eqref{corrected1} we infer that $\beta\mM(u)=0$, which is a contradiction since $\beta>0$ and $u\neq 0$. We thus conclude $\mM(u)=c$. That $u$ is positive follows immediately from the strong maximum principle. This completes the desired proof.
\end{proof}

\section{Non-trivial $y$-dependence of the ground states}\label{sec: Dependence}
In this section we prove Theorem \ref{thm threshold mass}. In a first step we consider the auxiliary problem $m_{1,\ld}$ defined by \eqref{def of auxiliary problem}. Proceeding as in the proof of Theorem \ref{thm existence of ground state} we know that the variational problem $m_{1,\ld}$ has an optimizer $u_\ld\in V(1)$ for any $\ld\in(0,\infty)$. Following the same lines in \cite{TTVproduct2014}, we prove the following characterization of $m_{1,\ld}$ for varying $\ld$.

\begin{lemma}\label{lemma auxiliary}
Let $\wm_c$ be the quantity defined through \eqref{def of wmc}. Then there exists some $\ld_*\in(0,\infty)$ such that
\begin{itemize}
\item For all $\ld \in (0,\ld_*)$ we have $m_{1,\ld}<2\pi \wm_{(2\pi)^{-1}}$. Moreover, for $\ld\in(0,\ld_*)$ any minimizer $u_\ld$ of $m_{1,\ld}$ satisfies $\pt_y u_\ld\neq 0$.

\item For all $\ld\in(\ld_*,\infty)$ we have $m_{1,\ld}=2\pi \wm_{(2\pi)^{-1}}$. Moreover, for $\ld\in(\ld_*,\infty)$ any minimizer $u_\ld$ of $m_{1,\ld}$ satisfies $\pt_y u_\ld=0$.
\end{itemize}
\end{lemma}

The proof of Theorem \ref{thm threshold mass} follows then from Lemma \ref{lemma auxiliary} by a simple rescaling argument.

\subsection{Proof of Lemma \ref{lemma auxiliary}}

In order to prove Lemma \ref{lemma auxiliary}, we firstly collect some useful auxiliary lemmas.

\begin{lemma}\label{auxiliary lemma 1}
We have
\begin{align}\label{limit ld to infty energy sec4}
\lim_{\ld\to\infty}m_{1,\ld}=2\pi\wm_{(2\pi)^{-1}}.
\end{align}
Additionally, let $u_\ld\in V(1)$ be a positive optimizer of $m_{1,\ld}$ which also satisfies
\begin{align}
-\Delta_x u_\ld-\ld \pt_y^2 u_\ld+\beta_\ld u_\ld=|u_\ld|^\alpha u_\ld\quad\text{on $\R^d\times \T$}\label{vanishing 3 sec4}
\end{align}
for some $\beta_\ld>0$ (whose existence is guaranteed by Theorem \ref{thm existence of ground state}). Then
\begin{align}
\lim_{\ld\to\infty}\ld\|\pt_y u_\ld\|_2^2=0.\label{vanishing sec4}
\end{align}
\end{lemma}

\begin{proof}
By assuming that a candidate in $V(1)$ is independent of $y$ we already conclude
\begin{align}
m_{1,\ld}\leq 2\pi \wm_{(2\pi)^{-1}}.\label{upper bound sec4}
\end{align} Next we prove
\begin{align}
\lim_{\ld\to\infty}\|\pt_y u_\ld\|_2^2=0.\label{vanishing1 sec4}
\end{align}
Suppose that \eqref{vanishing1 sec4} does not hold. Then we must have
\begin{align*}
\lim_{\ld\to\infty}\ld\|\pt_y u_\ld\|_2^2=\infty.
\end{align*}
Since $\mK(u_\ld)=0$ and $\alpha>4/d$,
\begin{align}
m_{1,\ld}&=\mH_\ld(u_\ld)-\frac{2}{\alpha d}\mK_{\ld}(u_\ld)\nonumber\\
&=\ld\|\pt_y u_\ld\|_2^2
+\bg(\frac12-\frac{2}{\alpha d}\bg)\|\nabla_x u_\ld\|_2^2
\geq \ld\|\pt_y u_\ld\|_2^2\to\infty\label{contradiction sec4}
\end{align}
as $\ld\to\infty$, which contradicts \eqref{upper bound sec4} and in turn proves \eqref{vanishing1 sec4}. Using \eqref{upper bound sec4} and \eqref{contradiction sec4} we infer that
\begin{align}\label{upper bound 2 sec4}
\|\nabla_x u_\ld\|_2^2\lesssim  m_{1,\ld}\leq  2\pi\wm_{(2\pi)^{-1}}<\infty.
\end{align}
Therefore $(u_\ld)_\ld$ is a bounded sequence in $H_{x,y}^1$, whose weak limit is denoted by $u$. Using \eqref{key gn inq2} and Lemma \ref{lemma non vanishing limit} we may also assume that $u\neq 0$. By \eqref{vanishing1 sec4} we know that $u$ is independent of $y$ and thus $u\in H_x^1$. Moreover, using weakly lower semicontinuity of norms we know that $\wmM(u)\in(0,(2\pi)^{-1}]$. On the other hand, using $\mK(u_\ld)=0$, \eqref{vanishing 3 sec4} and $\mM(u_\ld)=1$ we obtain
\begin{align*}
\beta_\ld=\frac{2\alpha +(4-\alpha d)}{2(\alpha+2)}\|u_\ld\|_{\alpha+2}^{\alpha+2}-\ld\|\pt_y u_\ld\|_2^2\lesssim \|u_\ld\|_{\alpha+2}^{\alpha+2}.
\end{align*}
Thus $(\beta_\ld)_\ld$ is a bounded sequence in $(0,\infty)$, whose limit is denoted by $\beta$. We now test \eqref{vanishing 3 sec4} with $\phi\in C_c^\infty(\R^d)$ and integrate both sides over $\R^d\times\T$. Notice particularly that the term $\int_{\R^d\times\T}  \pt_y^2 u_\ld \phi\,dxdy=0$ for any $\ld>0$ since $\phi$ is independent of $y$. Using the weak convergence of $u_\ld$ to $u$ in $H_{x,y}^1$, by sending $\ld\to\infty$ we obtain
\begin{align}
-\Delta_x u+\beta u=|u|^\alpha u\quad\text{in $\R^d$}.\label{vanishing 4 sec4}
\end{align}
In particular, by Lemma \ref{lem wmc property} we know that $\wmK(u)=0$ and $\beta>0$. Combining with weakly lower semicontinuity of norms we deduce
$$2\pi\wmH(u)=2\pi\wmI(u)\leq\liminf_{\ld\to\infty}\mI_\ld(u_\ld)= \liminf_{\ld\to\infty}\mH_\ld(u_\ld)\leq 2\pi \wm_{(2\pi)^{-1}},$$
where $\mI_\ld(u)$ and $\wmI(u)$ are the quantities defined by \eqref{def of I ld} and \eqref{def of wmI} respectively. However, by Lemma \ref{lem wmc property} the mapping $c\mapsto \wm_c$ is strictly monotone decreasing on $(0,\infty)$, from which we conclude that $\wmM(u)=(2\pi)^{-1}$ and $u$ is an optimizer of $\wm_{(2\pi)^{-1}}$. Finally, using the weakly lower semicontinuity of norms we conclude
\begin{align}
m_{1,\ld}&=\mH_\ld(u_\ld)=\mH_\ld(u_\ld)-\frac{2}{\alpha d}\mK(u)
=\frac{\ld}{2}\|\pt_y u_\ld\|_2^2+(\frac{1}{2}-\frac{2}{\alpha d})\|\nabla_x u_\ld\|_2^2\nonumber\\
&\geq (\frac{1}{2}-\frac{2}{\alpha d})\|\nabla_x u_\ld\|_2^2
\geq (\frac{1}{2}-\frac{2}{\alpha d})2\pi\|\nabla_x u\|_{L_x^2}^2+o_\ld(1)\nonumber\\
&=2\pi \wmH(u)+o_\ld(1)\geq 2\pi\wm_{(2\pi)^{-1}}+o_\ld(1).\label{vanishing 5 sec4}
\end{align}
Letting $\ld\to\infty$ and taking \eqref{upper bound sec4} into account we conclude \eqref{limit ld to infty energy sec4}. Finally, \eqref{vanishing sec4} follows directly from the previous calculation by not neglecting $\ld\|u_\ld\|_2^2$ therein. This completes the desired proof.
\end{proof}

\begin{lemma}\label{strong convergence u ld}
Let $u_\ld$ and $u$ be the functions given in the proof of Lemma \ref{auxiliary lemma 1}. Then $u_\ld\to u$ strongly in $H_{x,y}^1$.
\end{lemma}

\begin{proof}
In fact, in the proof of Lemma \ref{auxiliary lemma 1} we see that all the inequalities involving the weakly lower semicontinuity of norms are in fact equalities. This particularly implies $\|u_\ld\|_{H_{x,y}^1}\to \|u\|_{H_{x,y}^1}=2\pi\|u\|_{H_x^1}$ as $\ld\to\infty$, which in turn implies the strong convergence of $u_\ld$ to $u$ in $H_{x,y}^1$.
\end{proof}

\begin{lemma}\label{lemma no dependence}
There exists some $\ld_0$ such that $\pt_y u_\ld=0$ for all $\ld>\ld_0$.
\end{lemma}

\begin{proof}
Let $w_\ld:=\pt_y u_\ld$. Then taking $\pt_y$-derivative to \eqref{vanishing 3 sec4} we obtain
\begin{align}\label{no dependence 1}
-\Delta_x w_\ld-\ld \pt^2_y w_\ld+\beta_\ld w_\ld=\pt_y(|u_\ld|^\alpha u_\ld)=(\alpha+1)|u_\ld|^\alpha w_\ld.
\end{align}
Testing \eqref{no dependence 1} with $w_\ld$ and rewriting suitably, we infer that
\begin{align}
0&=\|\nabla_x w_\ld\|_2^2+\ld\|\pt_y w_\ld\|_2^2+\beta_\ld \|w_\ld\|_2^2-(\alpha+1)\int_{\R^d\times \T}|u_\ld|^\alpha |\bar{w}_\ld|^2\,dxdy\nonumber\\
&=(\ld-1)\|\pt_y w_\ld\|_2^2-(\alpha+1)\int_{\R^d\times \T}|u|^\alpha |w_\ld|^2\,dxdy\label{no dependence 2}\\
&+\beta_\ld \|w_\ld\|_2^2+\|\nabla_{x,y}w_\ld\|_2^2\label{no dependence 4}\\
&-(\alpha+1)\int_{\R^d\times \T}(|u_\ld|^\alpha -|u|^\alpha)|w_\ld|^2\,dxdy.\label{no dependence 3}
\end{align}
For \eqref{no dependence 2}, we firstly point out that by Lemma \ref{lem wmc property} (ii) and Sobolev embedding we have $u\in L^\infty(\R^d)$. On the other hand, since $\int_\T w_\ld\,dy=0$, we have $\|w_\ld\|_2\leq \|\pt_y w_\ld\|_2$. Summing up, we conclude that
\begin{align*}
\eqref{no dependence 2}\geq (\ld-1-(\alpha+1)\|u\|^{\alpha}_{L_x^\infty})\|\pt_y w_\ld\|_2^2\geq 0
\end{align*}
for all sufficiently large $\ld$. For \eqref{no dependence 3}, we discuss the cases $\alpha\leq 1$ and $\alpha>1$. For $\alpha\leq 1$, we estimate the second term in \eqref{no dependence 3} using subadditivity of concave function, H\"older's inequality, Lemma \ref{strong convergence u ld} and the Sobolev embedding $H_{x,y}^1\hookrightarrow L_{x,y}^{\alpha+2}$:
\begin{align*}
&\,\int_{\R^d\times \T}(|u_\ld|^\alpha -|u|^\alpha)|w_\ld|^2\,dxdy\nonumber\\
\leq&\,\int_{\R^d\times \T}|u_\ld-u|^\alpha|w_\ld|^2\,dxdy\nonumber\\
\leq& \,\|u_\ld-u\|_{\alpha+2}^\alpha\|w_\ld\|_{\alpha+2}^2\leq o_\ld(1)\|w_\ld\|^2_{H_{x,y}^1}.
\end{align*}
The case $\alpha>1$ can be similarly estimated as follows:
\begin{align*}
&\,\int_{\R^d\times \T}(|u_\ld|^\alpha -|u|^\alpha)|w_\ld|^2\,dxdy\nonumber\\
\lesssim&\,\int_{\R^d\times \T}|u_\ld-u||w_\ld|^2(|u_\ld|^{\alpha-1}+|u|^{\alpha-1})\,dxdy\nonumber\\
\leq& \,\|u_\ld-u\|_{\alpha+2}(\|u_\ld\|^{\alpha-1}_{\alpha+2}+\|u\|^{\alpha-1}_{\alpha+2})\|w_\ld\|_{\alpha+2}^2\leq o_\ld(1)\|w_\ld\|^2_{H_{x,y}^1}.
\end{align*}
Therefore, \eqref{no dependence 2}, \eqref{no dependence 3}, $\beta_\ld=\beta+o_\ld(1)$ and the fact $\beta>0$ imply
\begin{align*}
0\gtrsim \|w_\ld\|_{H_{x,y}^1}^2(1-o_\ld(1))\gtrsim \|w_\ld\|_{H_{x,y}^1}^2
\end{align*}
for all $\ld>\ld_0$ with some sufficiently large $\ld_0$. We therefore conclude that $0=w_\ld=\pt_y u_\ld$ for all $\ld>\ld_0$ and the desired proof is complete.
\end{proof}

Having all the preliminaries we are in a position to prove Lemma \ref{lemma auxiliary}.

\begin{proof}[Proof of Lemma \ref{lemma auxiliary}]
Define
\begin{align*}
\ld_*:=\inf\{\tilde{\ld}\in(0,\infty):m_{1,\ld}=2\pi \wm_{(2\pi)^{-1}}\,\forall\,\ld\geq \tilde{\ld}\}.
\end{align*}
From Lemma \ref{lemma no dependence} we already know $\ld_*<\infty$ and it is left to show $\ld_*>0$. It suffices to show
\begin{align}\label{strictly small}
\lim_{\ld\to 0}m_{1,\ld}<2\pi\wm_{(2\pi)^{-1}}.
\end{align}
To see this, we firstly define the function $\rho:[0,2\pi]\to[0,\infty)$ as follows: Let $a\in(0,\pi)$ such that $a>\pi-3\pi\bg(\frac{3}{\alpha+3}\bg)^{\frac2\alpha}$. This is always possible for $a$ sufficiently close to $\pi$. Then we define $\rho$ by
\begin{align*}
\rho(y)=\left\{
\begin{array}{ll}
0,&y\in[0,a]\cup[2\pi-a,2\pi],\\
(\pi-a)^{-1}\bg(\frac{\alpha+3}{3}\bg)^{\frac{1}{\alpha}}(y-a),&y\in[a,\pi],\\
(\pi-a)^{-1}\bg(\frac{\alpha+3}{3}\bg)^{\frac{1}{\alpha}}(2\pi-a-y),&y\in[\pi,2\pi-a].
\end{array}
\right.
\end{align*}
By direct calculation we see that $\rho\in H_y^1$ and
\begin{align*}
2\pi>\|\rho\|_{L_y^2}^2=\|\rho\|_{L_y^{\alpha+2}}^{\alpha+2}.
\end{align*}
Next, let $P\in H_x^1$ be an optimizer of $\wm_{\|\rho\|_{L_y^2}^{-2}}$. Since by Lemma \ref{lem wmc property} the mapping $c\mapsto\wm_c$ is strictly decreasing on $(0,\infty)$ and $\|\rho\|_{L_y^2}^{-2}>(2\pi)^{-1}$, we infer that $\wm_{\|\rho\|_{L_y^2}^{-2}}<\wm_{(2\pi)^{-1}}$. Furthermore, by Lemma \ref{lem wmc property} (iii) we have $\|\nabla_x P\|_{L_x^2}^2=\frac{\alpha d}{2(\alpha+2)}\|P\|_{L_x^{\alpha+2}}^{\alpha+2}$. Now define $\psi(x,y):=\rho(y)P(x)$. Then we conclude that $\psi\in H_{x,y}^1$, $\mM(\psi)=\|\rho\|_{L_y^2}^2\wmM(P)=1$,
\begin{align*}
\mK(\psi)&=\|\nabla_x\psi\|_2^2-\frac{\alpha d}{2(\alpha+2)}\|\psi\|_{\alpha+2}^{\alpha+2}\nonumber\\
&=\|\rho\|_{L_y^2}^2\|\nabla_x P\|_{L_x^2}^2
-\frac{\alpha d}{2(\alpha+2)}\|\rho\|_{L_y^{\alpha+2}}^{\alpha+2}\|P\|_{L_x^{\alpha+2}}^{\alpha+2}\nonumber\\
&=\|\rho\|_{L_y^2}^2\bg(\|\nabla_x P\|_{L_x^2}^2-\frac{\alpha d}{2(\alpha+2)}\|P\|_{L_x^{\alpha+2}}^{\alpha+2}\bg)=0
\end{align*}
and
\begin{align*}
\mH_{*}(\psi)
&:=\frac{1}{2}\|\nabla_x \psi\|_{2}^2-\frac{1}{\alpha+2}\|\psi\|_{{\alpha+2}}^{\alpha+2}\nonumber\\
&=\frac{1}{2}\|\rho\|_{L_y^2}^2\|\nabla_x P\|_{L_x^2}^2
-\frac{1}{\alpha+2}\|\rho\|_{L_y^{\alpha+2}}^{\alpha+2}\|P\|_{L_x^{\alpha+2}}^{\alpha+2}\nonumber\\
&=\|\rho\|_{L_y^2}^2\bg(\frac{1}{2}\|\nabla_x P\|_{L_x^2}^2-\frac{1}{\alpha+2}\|P\|_{L_x^{\alpha+2}}^{\alpha+2}\bg)
=\|\rho\|_{L_y^2}^2\wm_{\|\rho\|_{L_y^2}^{-2}}<2\pi\wm_{(2\pi)^{-1}}.\label{star energy}
\end{align*}
Consequently,
\begin{align}
\lim_{\ld\to 0}m_{1,\ld}\leq \lim_{\ld\to 0}\mH_{\ld}(\psi)=\mH_*(\psi)<2\pi\wm_{(2\pi)^{-1}}
\end{align}
and \eqref{strictly small} follows. That a minimizer of $m_{1,\ld}$ has non-trivial $y$-dependence for $\ld<\ld_*$ follows already from the definition of $\ld_*$ and the fact that $\ld\mapsto m_{1,\ld}$ is monotone increasing. We borrow an idea from \cite{GrossPitaevskiR1T1} to show that any minimizer of $m_{1,\ld}$ for $\ld>\ld_*$ must be $y$-independent. Assume the contrary that an optimizer $u_\ld$ of $m_{1,\ld}$ satisfies $\|\pt_y u_\ld\|_2^2\neq 0$. Then there exists some  $\mu\in(\ld_*,\ld)$. Then
\begin{align*}
2\pi \wm_{(2\pi)^{-1}}=m_{1,\mu}\leq \mH_{\mu}(u_\ld)=\mH_{\ld}(u_\ld)+\frac{\mu-\ld}{2}\|\pt_y u_\ld\|_2^2<\mH_{\ld}(u_\ld)=m_{1,\ld}=2\pi \wm_{(2\pi)^{-1}},
\end{align*}
a contradiction. This completes the desired proof.
\end{proof}

\subsection{Proof of Theorem \ref{thm threshold mass}}
We now prove Theorem \ref{thm threshold mass} by using Lemma \ref{lemma auxiliary} and a simple rescaling argument.

\begin{proof}[Proof of Theorem \ref{thm threshold mass}]
For $c>0$ let $\kappa_c:=c^{\frac{1}{d-\frac4\alpha}}$. Then $u\mapsto T_{\kappa_c} u$ (where $T_{\kappa_c} u$ is given by \eqref{def of t ld}) defines a bijection between $V(c)$ and $V(1)$. Direct calculation also shows that $\mH(u)=c^{\frac{d-4/\alpha-2}{d-4/\alpha}}\mH_{\kappa_c^{2}}(T_{\kappa_c} u)$, thus $m_{c}=c^{\frac{d-4/\alpha-2}{d-4/\alpha}}m_{1,\kappa_c^{2}}$. By same rescaling arguments we also infer that $\wm_{(2\pi)^{-1}c}=c^{\frac{d-4/\alpha-2}{d-4/\alpha}}\wm_{(2\pi)^{-1}}$ for $c>0$. Notice also that the mapping $c\mapsto \kappa_c$ is strictly monotone increasing on $(0,\infty)$. Thus by Lemma \ref{lemma auxiliary} there exists some $c_*\in(0,\infty)$ such that
\begin{itemize}
\item For all $c\in(0,c_*)$ we have
$$m_{c}=c^{\frac{d-4/\alpha-2}{d-4/\alpha}}m_{1,\kappa_c^2}<c^{\frac{d-4/\alpha-2}{d-4/\alpha}}2\pi \wm_{(2\pi)^{-1}}=2\pi \wm_{(2\pi)^{-1}c}.$$

\item For all $c\in(c_*,\infty)$ we have
$$m_{c}=c^{\frac{d-4/\alpha-2}{d-4/\alpha}}m_{1,\kappa_c^2}=c^{\frac{d-4/\alpha-2}{d-4/\alpha}}2\pi \wm_{(2\pi)^{-1}}=2\pi \wm_{(2\pi)^{-1}c}.$$
\end{itemize}
The statements concerning the $y$-dependence of the minimizers follow also from Lemma \ref{lemma auxiliary} simultaneously. Finally, that $m_{c_*}=2\pi \wm_{(2\pi)^{-1}c_*}$ follows immediately from the continuity of the mappings $c\mapsto m_c$ and $c\mapsto \wm_c$ deduced from Lemma \ref{monotone lemma} and Lemma \ref{lem wmc property} respectively. This completes the proof.
\end{proof}

\section{Scattering and finite time blow-up below ground states}\label{sec: Scattering}

\subsection{Some useful inequalities}
In this subsection we collect some useful inequalities which will be used throughout the rest of the paper.

\begin{lemma}[Strichartz estimates on $\R^d\times\T$, \cite{TzvetkovVisciglia2016}]\label{strichartz}
Let $\gamma\in\R$, $s\in[0,d/2)$ and $p,q,\tilde{p},\tilde{q}$ satisfy $p,\tilde{p}\in(2,\infty]$ and
\begin{align}
\frac{2}{p}+\frac{d}{q}=\frac{2}{\tilde{p}}+\frac{d}{\tilde{q}}
=\frac{d}{2}-s.
\end{align}
Then for a time interval $I\ni t_0$ we have
\begin{align}
\|e^{i(t-t_0)\dxy}f\|_{L_t^p L_x^q H^\gamma_y(I)}&\lesssim\|f\|_{{H}_x^sH_y^\gamma}.
\end{align}
Moreover, the Strichartz estimate for the nonlinear term
\begin{align}
\|\int_{t_0}^t e^{i(t-s)\dxy}F(s)\,ds\|_{L_t^p L_x^q H^\gamma_y(I)}&\lesssim \|F\|_{L_t^{\tilde{p}'} L_x^{\tilde{q}'} H^\gamma_y(I)}
\end{align}
holds in the case $s=0$.
\end{lemma}

\begin{lemma}[Exotic Strichartz estimates on $\R^d\times\T$, \cite{TzvetkovVisciglia2016}]\label{exotic strichartz}
There exist $\ba,\br,\bb,\bs\in(2,\infty)$ such that
\begin{gather*}
(\alpha+1)\bs'=\br,\quad(\alpha+1)\bb'=\ba,\quad\alpha/\br<\min\{1,\frac{2}{d}\},\quad
\frac{2}{\ba}+\frac{d}{\br}=\frac{2}{\alpha}.
\end{gather*}
Moreover, for any $\gamma\in\R$ we have the following exotic Strichartz estimate:
\begin{align}
\|\int_{t_0}^t e^{i(t-s)\dxy}F(s)\,ds\|_{L_t^\ba L_x^\br H^\gamma_y(I)}&\lesssim \|F\|_{L_t^{\bb'} L_x^{\bs'} H^\gamma_y(I)}.
\end{align}
We may also impose the following restrictions on $(\ba,\br)$: When $d=1$, then $\br$ can be chosen arbitrarily close to $\alpha+1$; When $d\geq 2$, then $\br$ can be chosen arbitrarily close to $\frac{\alpha(\alpha+1)d}{\alpha+2}$.
\end{lemma}

\begin{remark}
As we shall see in the following, the exotic admissible pairs $(\ba,\br)$ and $(\bb,\bs)$ are ``computationally friendly'' for the remaining analysis. Moreover, from the original statement of \cite[Lem. 8.1]{TzvetkovVisciglia2016} we only have $\alpha/\br<1$. Nevertheless, in the proof of \cite[Lem. 8.1]{TzvetkovVisciglia2016} the authors indeed deduced the sharper upper bound $2/d$ for $\alpha/\br$ when $d\geq 3$ and that $\br$ can be chosen to be arbitrarily close to the endpoints $\alpha+1$ ($d=1$) resp. $\frac{\alpha(\alpha+1)d}{\alpha+2}$ ($d\geq 2$). These will become useful for the proofs of Lemma \ref{lemma l2 admissible special} and Lemma \ref{lem stability cnls} below.
\end{remark}

\begin{lemma}\label{lemma l2 admissible special}
Let $(\tilde{\ba},\tilde{\br})=(\frac{4\br}{\alpha d},\frac{2\br}{\br-\alpha})$. Then $(\tilde{\ba},\tilde{\br})$ is an $L^2$-admissible pair and
\begin{align}\label{l2 admissible exotic}
\frac{1}{\tilde{\ba}'}=\frac{\alpha}{\ba}+\frac{1}{\tilde{\ba}},\quad\frac{1}{\tilde{\br}'}=\frac{\alpha}{\br}+\frac{1}{\tilde{\br}}
\end{align}
is satisfied.
\end{lemma}

\begin{proof}
It is easy to see that \eqref{l2 admissible exotic} is satisfied for the given $(\ba,\br)$ and $\frac{2}{\ba}+\frac{d}{\br}=\frac{d}{2}$. It remains to show that $\br\in(2,2^*)$, where $2^*=\infty$ when $d=1,2$ and $2^*=\frac{2d}{d-2}$ when $d\geq 3$, which in turn implies that $(\ba,\br)$ is an $L^2$-admissible pair. We discuss different cases:
\begin{itemize}
\item For $d\in\{1,2\}$, we have $\tilde{\br}=\frac{2\br}{\br-\alpha}\in(2,\infty)$ and we obtain an admissible choice.
\item For $d\geq 3$, we may rewrite $\tilde{\br}=\frac{2\br}{\br-\alpha}$ to $\tilde{\br}=\frac{2(2\br/\alpha)}{(2\br/\alpha)-2}$. Notice that the function $l\mapsto\frac{2l}{l-2}$ is monotone decreasing on $(2,\infty)$. From Lemma \ref{exotic strichartz} we know that $2\ro/\alpha\in(d,\infty)$. Thus $\tilde{\br}\in(2,\frac{2d}{d-2})$ and we obtain an admissible choice.
\end{itemize}
\end{proof}

\begin{lemma}[Fractional calculus on $\T$, \cite{TzvetkovVisciglia2016}]\label{fractional on t}
For $s\in(0,1)$ and $\alpha>0$ we have
\begin{align}
\|u^{\alpha+1}\|_{\dot{H}_y^s}+\||u|^{\alpha+1}\|_{\dot{H}_y^s}+\||u|^\alpha u\|_{\dot{H}_y^s}\lesssim_{\alpha,s} \|u\|_{\dot{H}_y^s}\|u\|_{L_y^\infty}^\alpha
\end{align}
for $u\in H_y^s$.
\end{lemma}

\subsection{Small data and stability theories}
In the following we prove some useful lemmas such as the small data and stability theories for \eqref{nls} which appear as default preliminaries in a standard rigidity proof based on the concentration compactness principle.

\begin{lemma}[Small data well-posedness]\label{lemma small data}
Let $I$ be an open interval containing $0$. Define the space $X(I)$ through the norm
\begin{align}
\|u\|_{X(I)}:=\|u\|_{S_x L_y^2(I)}+\|\nabla_x u\|_{S_x L_y^2(I)}+\|\pt_y u\|_{S_x L_y^2(I)},
\end{align}
where the space $S_x$ is defined by \eqref{def of sx}. Let also $s\in (\frac12,1-s_\alpha)$ be some given number, where $s_\alpha:=\frac{d}{2}-\frac{2}{\alpha}\in(0,\frac12)$. Assume now
\begin{align}
\|u_0\|_{H_{x,y}^1}\leq A
\end{align}
for some $A>0$. Then there exists some $\delta=\delta(A)$ such that if
\begin{align}
\|e^{it\Delta}u_0\|_{L_t^\ba L_x^\br  H_y^s(I) }\leq \delta,
\end{align}
then there exists a unique solution $u\in X(I)$ of \eqref{nls} with $u(0)=u_0$ such that
\begin{align}
\|u\|_{X(I)}&\lesssim A,\\
\|u\|_{L_t^\ba L_x^\br  H_y^s(I)}&\leq 2\|e^{it\Delta}u_0\|_{L_t^\ba L_x^\br  H_y^s(I)}.
\end{align}
\end{lemma}

\begin{proof}
We define the space
\begin{align*}
B(I):=\{u\in X(I):\|u\|_{X(I)}\leq 2CA,\|u\|_{L_t^\ba L_x^\br  H_y^s(I)}\leq 2\delta\}.
\end{align*}
Then $B(I)$ is a complete metric space with the metric $\rho(u,v):=\|u-v\|_{S_x L_y^2(I)}$. Define also the Duhamel's mapping $\Psi(u)$ by
\begin{align*}
\Psi(u):=e^{it\Delta_{x,y}}u_0+i\int_0^t e^{i(t-s)\Delta_{x,y}}(|u|^\alpha u)(s)\,ds.
\end{align*}
We firstly show that $\Psi(B(I))\subset B(I)$ by choosing $\delta$ samll. Let $D\in\{1,\pt_{x_i},\pt_y\}$. Let $(\tilde{\ba},\tilde{\br})$ be the $L^2$-admissible pairs given by Lemma \ref{lemma l2 admissible special}. Using the Strichartz estimates given in Lemma \ref{strichartz}, the embedding $H_y^s\hookrightarrow L_y^\infty$ and H\"older we obtain
\begin{align*}
\|D(\Psi(u))\|_{S_x L_y^2(I)}&\leq CA+C\|u^{\alpha}Du\|_{L_t^{\tilde{\ba}'}L_x^{\tilde{\br}'}L_y^2(I)}
\leq CA+C\|\|u\|^{\alpha}_{L_y^{\infty}}\|Du\|_{L_y^2}\|_{L_t^{\tilde{\ba}'}L_x^{\tilde{\br}'}(I)}\nonumber\\
&\leq CA+C\|\|u\|^{\alpha}_{H_y^s}\|Du\|_{L_y^2}\|_{L_t^{\tilde{\ba}'}L_x^{\tilde{\br}'}(I)}
\leq CA+C\|u\|^{\alpha}_{L_t^{{\ba}}L_x^{{\br}} H_y^s(I)}\|Du\|_{L_t^{\tilde{\ba}}L_x^{\tilde{\br}}L_y^2(I)}\nonumber\\
&
\leq CA+(2\delta)^{\alpha} CA\leq 2CA
\end{align*}
by choosing $\delta$ sufficiently small. Similarly, using Lemma \ref{exotic strichartz}, the embedding $H_y^s\hookrightarrow L_y^\infty$ and Lemma \ref{fractional on t} we obtain
\begin{align*}
\|\Psi(u)\|_{L_t^\ba L_x^\br  H_y^s(I)}
&\leq \delta+\||u|^{\alpha}u\|_{L_t^{\bb'} L_x^{\bs'} H^s_y(I)}
\leq \delta+C(\||u|^{\alpha}u\|_{L_t^{\bb'} L_x^{\bs'} L^2_y(I)}+\||u|^{\alpha}u\|_{L_t^{\bb'} L_x^{\bs'} \dot{H}^s_y(I)})\nonumber\\
&\leq \delta+C(\|\|u\|^{\alpha+1}_{L_y^\infty}\|_{L_t^{\bb'} L_x^{\bs'}(I)}+
\|\|u\|^{\alpha}_{L_y^\infty}\| u\|_{\dot{H}_y^s}\|_{L_t^{\bb'} L_x^{\bs'}(I)})\nonumber\\
&\leq \delta+C\|\|u\|^{\alpha+1}_{H_y^s}\|_{L_t^{\bb'} L_x^{\bs'}(I)}=\delta+C\|u\|^{\alpha+1}_{L_t^{\ba} L_x^{\br}H_y^s(I)}
\leq \delta(1+C\delta^\alpha)\leq 2\delta.
\end{align*}
Finally, we show that $\Psi$ is a contraction on $B(I)$. By direct calculation we infer that
\begin{align*}
\|\Psi(u)-\Psi(v)\|_{S_x L_y^2}&\leq C\||u|^\alpha u-|v|^\alpha v\|_{L_t^{\tilde{\ba}'}L_x^{\tilde{\br}'}L_y^2(I)}
\leq C\|(|u|^{\alpha}+|v|^\alpha)(u-v)\|_{L_t^{\tilde{\ba}'}L_x^{\tilde{\br}'}L_y^2(I)}\nonumber\\
&\leq C\|(\|u\|^{\alpha}_{L_y^{\infty}}+\|v\|^{\alpha}_{L_y^{\infty}})\|u-v\|_{L_y^2}\|_{L_t^{\tilde{\ba}'}L_x^{\tilde{\br}'}(I)}\nonumber\\
&\leq C(\|u\|_{L_t^{{\ba}}L_x^{{\br}} H_y^s(I)}+\|v\|_{L_t^{{\ba}}L_x^{{\br}} H_y^s(I)})^\alpha\|u-v\|_{S_x L_y^2(I)}\leq C\delta^\alpha
\|u-v\|_{S_x L_y^2(I)}.
\end{align*}
The desired claim follows by choosing $\delta$ small and using the Banach's fixed point theorem.
\end{proof}

\begin{lemma}[Scattering norm]\label{lemma scattering norm}
Let $u\in X_{\rm loc}(\R)$ be a global solution of \eqref{nls} such that
\begin{align}
\|u\|_{L_t^\ba L_x^\br  H_y^s(\R)}+\|u\|_{L_t^\infty H_{x,y}^1(\R)}<\infty.
\end{align}
Then $u$ scatters in time. Moreover, for all $s'\in(\frac12,1-s_\alpha)$ we have
\begin{align}\label{uniform bound}
\|u\|_{X(\R)}+\|u\|_{L_t^\ba L_x^\br H_y^{s'}(\R)}\lesssim_{\|u\|_{L_t^\ba L_x^\br  H_y^s(\R)},\,\|u\|_{L_t^\infty H_{x,y}^1(\R)}}1.
\end{align}
\end{lemma}

\begin{proof}
We first show \eqref{uniform bound}. Partition $\R$ into $\R=\cup_{j=1}^m I_j$ such that $\|u\|_{L_t^\ba L_x^\br  H_y^s(I_j)}=O(m^{-1})$. Let $D\in\{1,\pt_{x_i},\pt_y\}$. Arguing as in the proof of Lemma \ref{lemma small data}, on the interval $I_j=[t_j,t_{j+1}]$ with $s_j\in(t_j,t_{j+1})$ we have
\begin{align*}
\|u\|_{X(I_j)}&\lesssim \|e^{i(t-s_j)\Delta_{x,y}}u(s_j)\|_{X(I_j)}+\sum_{D\in\{1,\pt_{x_i},\pt_y\}}\|u\|^{\alpha}_{L_t^{{\ba}}L_x^{{\br}} H_y^s(I_j)}\|Du\|_{L_t^{\tilde{\ba}}L_x^{\tilde{\br}}L_y^2(I_j)}\nonumber\\
&\lesssim \|u(s_j)\|_{ H_{x,y}^1}+O(m^{-\alpha})\|u\|_{X(I_j)}.
\end{align*}
Choosing $m$ sufficiently large we can absorb the term $O(m^{-\alpha})\|u\|_{X(I_j)}$ to the l.h.s. and conclude that $\|u\|_{X(I_j)}\lesssim \|u(s_j)\|_{ H_{x,y}^1}$. The upper bound of $\|u\|_{X(\R)}$ given by \eqref{uniform bound} follows then by summing up the partial estimates on $I_j$ over $j=1,\cdots m$. The upper bound of $\|u\|_{L_t^\ba L_x^\br H_y^{s'}(\R)}$ follows very similarly by also appealing to Lemma \ref{fractional on t}:
\begin{align*}
\|u\|_{L_t^\ba L_x^\br H_y^{s'}(I_j)}&\lesssim \|u(s_j)\|_{ H_{x,y}^1}+\|u\|^{\alpha}_{L_t^{{\ba}}L_x^{{\br}} H_y^s(I_j)}\|u\|_{L_t^{\ba}L_x^{\br}H_y^{s'}(I_j)}\lesssim \|u(s_j)\|_{ H_{x,y}^1}+O(m^{-\alpha})\|u\|_{L_t^{{\ba}}L_x^{{\br}} H_y^{s'}(I_j)}.
\end{align*}
Next, define
\begin{align*}
\phi:=u_0+i\int_0^\infty e^{-is\Delta_{x,y}}(|u|^\alpha u)(s)\,ds.
\end{align*}
Then
\begin{align*}
u-e^{it\Delta_{x,y}}\phi=-i\int_t^\infty e^{i(t-s)\Delta}(|u|^\alpha u)(s)\,ds.
\end{align*}
Consequently, using Strichartz estimates we infer that
\begin{align*}
\|D(u-e^{it\Delta_{x,y}}\phi)\|_{L_{x,y}^2}&\lesssim\|u^\alpha Du\|_{L_t^{\tilde{\ba}'}L_x^{\tilde{\br}'}L_y^2(t,\infty)}
\lesssim \|u\|^{\alpha}_{L_t^{{\ba}}L_x^{{\br}} H_y^s(t,\infty)}\|Du\|_{L_t^{\tilde{\ba}}L_x^{\tilde{\br}}L_y^2(t,\infty)}\nonumber\\
&\lesssim\|u\|^{\alpha}_{L_t^{{\ba}}L_x^{{\br}} H_y^s(t,\infty)}\|u\|_{X(\R)}\lesssim\|u\|^{\alpha}_{L_t^{{\ba}}L_x^{{\br}} H_y^s(t,\infty)}\to 0
\end{align*}
as $t\to\infty$, since $\|u\|_{L_t^{{\ba}}L_x^{{\br}} H_y^s(\R)}<\infty$. This shows that $u$ scatters in positive time. That $u$ scatters in negative time follows verbatim, we omit the details here.
\end{proof}

\begin{lemma}[Criterion for maximality of lifespan]
Let $u\in X_{\rm loc}(I_{\max})$ be a solution of \eqref{nls} defined on its maximal lifespan $I_{\max}$. Then if $t_{\max}:=\sup I_{\max}<\infty$, we have $\|u\|_{L_t^\ba L_x^\br  H_y^s(t,t_{\max})}=\infty$ for all $t\in I_{\max}$. A similar result holds for $t_{\min}:=\inf I_{\max}>-\infty$.
\end{lemma}

\begin{proof}
Assume the contrary that there exists some $t_0\in I_{\max}$ such that
\begin{align}\label{contradiction criterion}
\|u\|_{L_t^\ba L_x^\br  H_y^s(t_0,t_{\max})}<\infty.
\end{align}
We will show that $\lim_{t\to t_{\max}}u(t)$ exists in $H_{x,y}^1$. By Lemma \ref{lemma small data} this would mean that we can extend $u(t)$ beyond $t_{\max}$ for $t$ sufficiently close to $t_{\max}$, which contradicts the maximality of $I_{\max}$. By the strong continuity of the linear group $(e^{it\Delta_{x,y}})_{t\in\R}$ it suffices to show that $(e^{-it\Delta_{x,y}}u(t))_{t\nearrow t_{\max}}$ is a Cauchy sequence. Let $t_m<t_n<t_{\max}$. By Duhamel's formula and Strichartz estimate we have
\begin{align}
&\,\|e^{-it_m\Delta_{x,y}}u(t_m)-e^{-it_n\Delta_{x,y}}u(t_n)\|_{H_{x,y}^1}\nonumber\\
=&\,\|\int_{t_m}^{t_n}e^{i(t_n-s)\Delta_{x,y}}(|u|^\alpha u)(s)\,ds\|_{H_{x,y}^1}
\leq \|\int_{t_m}^{t}e^{i(t-s)\Delta_{x,y}}(|u|^\alpha u)(s)\,ds\|_{L_t^\infty H_{x,y}^1(t_m,t_{\max})}\nonumber\\
\leq &\,\sum_{D\in\{1,\pt_{x_i},\pt_y\}}\||u|^{\alpha} Du\|_{L_t^{\tilde{\ba}'}L_x^{\tilde{\br}'}L_y^2(t_m,t_{\max})}
\lesssim \sum_{D\in\{1,\pt_{x_i},\pt_y\}}\|u\|^{\alpha}_{L_t^{{\ba}}L_x^{{\br}} H_y^s(t_m,t_{\max})}\|Du\|_{L_t^{\tilde{\ba}}L_x^{\tilde{\br}}L_y^2(t_m,t_{\max})}.\label{cauchy seq}
\end{align}
From the proof of Lemma \ref{lemma scattering norm} and \eqref{contradiction criterion} we know that
\begin{gather*}
\|u\|_{X(t_0,t_{\max})}<\infty,\\
\lim_{t\nearrow t_{\max}}\|u\|_{L_t^\ba L_x^\br  H_y^s(t,t_{\max})}=0
\end{gather*}
Combining with \eqref{cauchy seq} we deduce the desired claim.
\end{proof}

\begin{lemma}[Stability theory]\label{lem stability cnls}
Let $d\leq 4$. Let $u$ be a solution of \eqref{nls} defined on some interval $0\in I\subset\R$ and let $\tilde{u}$ be a solution of the perturbed NLS
\begin{align}
i\pt_t\tilde{u}+\Delta_{x,y}\tilde{u}+|\tilde{u}|^\alpha \tilde{u}+e=0.
\end{align}
Assume that there exists some $A>0$ such that
\begin{align}
\|\tilde{u}\|_{L_t^\ba L_x^\br H_y^s(I)}&\leq A.\label{cond 1}
\end{align}
Then there exist $\vare(A)>0$ and $C(A)>0$ such that if
\begin{align}
\|e\|_{L_t^{\bb'}L_x^{\bs'}H_y^s(I)}&\leq \vare\leq\vare(A),\label{cond 2}\\
\|e^{it\Delta_{x,y}}(u(0)-\tilde{u}(0))\|_{L_t^{\ba}L_x^{\br}H_y^s(I)}&\leq \vare\leq\vare(A),\label{cond 3}
\end{align}
then $\|u-\tilde{u}\|_{L_t^\ba L_x^\br H_y^s(I)}\leq C(A)\vare$.
\end{lemma}

\begin{proof}
W.l.o.g. let $I=(-T,T)$ for some $T\in(0,\infty]$. We first recall the following well-known identity (see for instance the proof of \cite[Lem. 4.1]{TzvetkovVisciglia2016}):
\begin{align*}
c\|u\|_{\dot{H}_y^s}^2=\int_0^{2\pi}\int_{\R}\frac{|u(x+h)-u(x)|^2}{|h|^{1+2s}}\,dhdx
\end{align*}
for some positive constant $c>0$. Set $w:=u-\tilde{u}$ and $F(z):=|z|^\alpha z$ for $z\in\C$. Then
\begin{align}\label{long1}
&\,c\||\tilde{u}+w|^\alpha(\tilde{u}+w)-|\tilde{u}|^\alpha \tilde{u}\|_{\dot{H}_y^s}^2\nonumber\\
=&\,\int_0^{2\pi}\int_{\R}\frac{|(F((\tilde{u}+w)(x+h))-F((\tilde{u})(x+h)))-(F((\tilde{u}+w)(x))-F((\tilde{u})(x)))|^2}{|h|^{1+2s}}\,dhdx.
\end{align}
Writing $z=a+bi$ we may identify $F(z)$ with the two-dimensional function $F(a,b)$ through $F(a,b)=F(z)$. Define the usual complex derivatives by
\begin{align*}
F_z:=\frac{1}{2}\bg(\frac{\pt F}{\pt a}-i\frac{\pt F}{\pt b}\bg),\quad F_{\bar{z}}:=\frac{1}{2}\bg(\frac{\pt F}{\pt a}+i\frac{\pt F}{\pt b}\bg).
\end{align*}
Using chain rule we obtain that for $z,z'\in\C$
\begin{align*}
F(z)-F(z')=(z-z')\int_0^1 F_z(z'+\theta(z-z'))\,d\theta+\overline{(z-z')}\int_0^1 F_{\bar{z}}(z'+\theta(z-z'))\,d\theta.
\end{align*}
Combining with standard telescoping argument and the fact that $\alpha>1$ in the case $d\leq 4$ we see that
\begin{align}
&\,|(F((\tilde{u}+w)(x+h))-F((\tilde{u})(x+h)))-(F((\tilde{u}+w)(x))-F((\tilde{u})(x)))|\nonumber\\
\lesssim&\,|w(x+h)-w(x)|(|\tilde{u}(x+h)|^\alpha+|w(x+h)|^\alpha)\nonumber\\
+&\,|w(x)|(|\tilde{u}(x+h)|+|w(x+h)|+|\tilde{u}(x)|+|w(x)|)^{\alpha-1}
\times(|\tilde{u}(x+h)-\tilde{u}(x)|+|w(x+h)-w(x)|).\label{long2}
\end{align}
\eqref{long1}, \eqref{long2} and the embedding $H_y^s\hookrightarrow L_y^\infty$ now yield
\begin{align*}
&\,\||\tilde{u}+w|^\alpha(\tilde{u}+w)-|\tilde{u}|^\alpha \tilde{u}\|_{\dot{H}_y^s}^2\nonumber\\
\lesssim &\,\int_0^{2\pi}\int_{\R}\frac{|w(x+h)-w(x)|^2(\|\tilde{u}\|_{L_y^\infty}^{2\alpha}+\|w\|_{L_y^\infty}^{2\alpha})}{|h|^{1+2s}}\,dhdx\nonumber\\
+&\,\int_0^{2\pi}\int_{\R}
\frac{(|\tilde{u}(x+h)-\tilde{u}(x)|^2+|w(x+h)-w(x)|^2)(\|\tilde{u}\|_{L_y^\infty}^{2\alpha-2}+\|w\|_{L_y^\infty}^{2\alpha-2})\|w\|^2_{L_y^\infty}}{|h|^{1+2s}}\,dhdx\nonumber\\
\lesssim&\,\|w\|^2_{\dot{H}_y^s}(\|\tilde{u}\|_{L_y^\infty}^{2\alpha}+\|w\|_{L_y^\infty}^{2\alpha})
+\|w\|^2_{\dot{H}_y^s}(\|\tilde{u}\|^2_{\dot{H}_y^s}+\|w\|^2_{\dot{H}_y^s})(\|\tilde{u}\|_{L_y^\infty}^{2\alpha-2}+\|w\|_{L_y^\infty}^{2\alpha-2})
\nonumber\\
\lesssim&\, \|\tilde{u}\|^{2\alpha}_{H_y^s}\|w\|^2_{H_y^s}+\|w\|^{2\alpha+2}_{H_y^s}.
\end{align*}
This in turn implies
\begin{align*}
\||\tilde{u}+w|^\alpha(\tilde{u}+w)-|\tilde{u}|^\alpha \tilde{u}\|_{\dot{H}_y^s}\lesssim
\|\tilde{u}\|^{\alpha}_{H_y^s}\|w\|_{H_y^s}+\|w\|^{\alpha+1}_{H_y^s}.
\end{align*}
On the other hand, a simple application of H\"older's inequality also yields
\begin{align*}
\||\tilde{u}+w|^\alpha(\tilde{u}+w)-|\tilde{u}|^\alpha \tilde{u}\|_{L_y^2}\lesssim
\|\tilde{u}\|^{\alpha}_{H_y^s}\|w\|_{H_y^s}+\|w\|^{\alpha+1}_{H_y^s}
\end{align*}
and we conclude
\begin{align}\label{final long}
\||\tilde{u}+w|^\alpha(\tilde{u}+w)-|\tilde{u}|^\alpha \tilde{u}\|_{H_y^s}\lesssim
\|\tilde{u}\|^{\alpha}_{H_y^s}\|w\|_{H_y^s}+\|w\|^{\alpha+1}_{H_y^s}.
\end{align}
We also notice that $w$ satisfies the NLS
\begin{align*}
i\pt_t w+\Delta_{x,y}w+|\tilde{u}+w|^\alpha(\tilde{u}+w)-|\tilde{u}|^\alpha\tilde{u}+e=0.
\end{align*}
Using the exotic Strichartz estimates given in Lemma \ref{exotic strichartz} and \eqref{final long} we infer that for $t\in(0,T)$
\begin{align}\label{cazenave}
\|w\|_{L_t^\ba L_x^\br H_y^s(-t,t)}\leq(M+1)\vare +M\|\|\tilde{u}\|^{\alpha}_{H_y^s}\|w\|_{H_y^s}+\|w\|^{\alpha+1}_{H_y^s}\|_{L_t^{\bb'}L_x^{\bs'}(-t,t)}
\end{align}
for some $M>0$. We point out that \eqref{cazenave} is exactly (4.20) in \cite{focusing_sub_2011}. Hence following the proof of \cite[Prop. 4.7]{focusing_sub_2011}, words by words, we deduce the desired claim by setting
\begin{align*}
\vare(A)\leq 2^{-\frac{1}{\alpha}}[(2M+1)\Phi(A^\alpha)]^{-\frac{\alpha+1}{\alpha}},
\end{align*}
where $\Phi$ is the function given by \cite[Lem. 8.1]{focusing_sub_2011}. Notice that in order to apply \cite[Lem. 8.1]{focusing_sub_2011}, the quantities
\begin{align*}
\gamma:=\ba,\quad\rho:=\frac{\ba}{\alpha},\quad\beta=\bb'
\end{align*}
should satisfy $1\leq\beta<\gamma\leq \infty$ and $1\leq\rho<\infty$. By direct calculation, this is clearly fulfilled when $\br$ is sufficiently close to $\alpha+1$ ($d=1$) resp. $\frac{\alpha(\alpha+1)d}{\alpha+2}$ ($d\geq 2$), the latter being guaranteed by Lemma \ref{exotic strichartz}.
\end{proof}

\subsection{Profile decomposition}
We derive in this subsection a profile decomposition for a bounded sequence in $H_{x,y}^1$. To begin with, we firstly record an inverse Strichartz inequality.

\begin{lemma}[Inverse Strichartz inequality]\label{refined l2 lemma 1}
Let $(f_n)_n\subset H_{x,y}^1$. Suppose that
\begin{align}\label{425}
\lim_{n\to\infty}\|f_n\|_{H_{x,y}^1}=A<\infty\quad\text{and}\quad\lim_{n\to\infty}\|e^{it\Delta_{x}} f_n\|_{L_t^\ba L_{x,y}^\br (\R)}=\vare>0.
\end{align}
Then up to a subsequence, there exist $\phi\in H_{x,y}^1$ and $(t_n,x_n)_n\subset\R\times\R^d$ such that
\begin{align}\label{cnls l2 refined strichartz basic 5}
e^{it_n\Delta_x}f_n(x+x_n,y)\rightharpoonup &\,\,\phi(x,y)\text{ weakly in $H_{x,y}^1$}.
\end{align}
Set $\phi_n:=e^{-it_n\Delta_x}\phi(x-x_n,y)$. Let $s\in(\frac{1}{2},1-s_\alpha)$ be some given number, where $s_\alpha:=\frac{d}{2}-\frac{2}{\alpha}\in(0,\frac12)$. Then there exists some positive $\kappa\ll 1$ and $\theta,\beta\in(0,1)$ such that for $D\in\{1,\pt_{x_i},\pt_y\}$ we have
\begin{align}
\lim_{n\to\infty}&(\|f_n\|^2_{H_{x,y}^1}-\|f_n-\phi_n\|^2_{H_{x,y}^1})=\|\phi\|^2_{H_{x,y}^1}\gtrsim
A^{-\frac{2(\theta+\beta(1-\theta))}{(1-\beta)(1-\theta)}-\frac{d}{\kappa}}\vare^{\frac{2}{(1-\beta)(1-\theta)}+\frac{d}{\kappa}},
\label{cnls l2 refined strichartz decomp 1}\\
\lim_{n\to\infty}&(\|Df_n\|^2_{L_{x,y}^2}-\|D(f_n-\phi_n)\|^2_{L_{x,y}^2}-\|D\phi_n\|^2_{L_{x,y}^2})=0.
\label{cnls l2 refined strichartz decomp 3}
\end{align}
\end{lemma}

\begin{proof}
For $R\geq 1$ let $\chi_R:\R^d\to[0,1]$ be the indicator function of the ball $\{\xi\in\R^d:|\xi|\leq R\}$. We then define $f_n^R$ through its symbol $\mathcal{F}_x(f_n^R)=\chi_R \mathcal{F}_x(f_n)$. By Strichartz estimate and the embedding $H_y^s\hookrightarrow L_y^\br$ we obtain
\begin{align*}
&\,\|e^{it\Delta_{x}}(f_n-f_n^R)\|_{L_t^\ba L_{x,y}^\br (\R)}\nonumber\\
\lesssim&\,\|e^{it\Delta_{x,y}}e^{-it\Delta_{y}}(f_n-f_n^R)\|_{L_t^\ba L_{x}^\br H_y^s (\R)}
\lesssim \|f_n-f_n^R\|_{H_x^{s_\alpha} H_y^s}\nonumber\\
\lesssim&\, \|f_n-f_n^R\|_{L_{x,y}^2}+\|f_n-f_n^R\|_{L_x^2 \dot{H}_y^s}+\|f_n-f_n^R\|_{\dot{H}_x^{s_\alpha} L_y^2 }+\|f_n-f_n^R\|_{\dot{H}_x^{s_\alpha} \dot{H}_y^{s}}\nonumber\\
=:&\,I+II+III+IV.
\end{align*}
Writing the Hilbert norms via Fourier transform we deduce
\begin{align}
&\,(I+III)^2\nonumber\\
\lesssim&\,\sum_{k\in\Z}\int_{|\xi|\geq R}|\mathcal{F}_{x,y}(f_n)(\xi,k)|^2\,d\xi
+\sum_{k\in\Z}\int_{|\xi|\geq R}|\xi|^{2s_\alpha}|\mathcal{F}_{x,y}(f_n)(\xi,k)|^2\,d\xi\nonumber\\
\lesssim &\,(R^{-2}+R^{-2(1-s_\alpha)})\sum_{k\in\Z}\int_{|\xi|\geq R}|\xi|^2|\mathcal{F}_{x,y}(f_n)(\xi,k)|^2\,d\xi\nonumber\\
\lesssim&\,(R^{-2}+R^{-2(1-s_\alpha)})\|f_n\|^2_{\dot{H}_x^1 L_y^2}\lesssim (R^{-2}+R^{-2(1-s_\alpha)})A^2.\label{or1}
\end{align}
For $II$ and $IV$, since $s\in(\frac{1}{2},1-s_\alpha)$, we can find some positive $\kappa\ll 1$ such that $s\in(\frac{1}{2},1-(s_\alpha+\kappa))$. Then using H\"older we infer that
\begin{align}
&\,(II+IV)^2\nonumber\\
\lesssim&\,\sum_{k\in\Z}k^{2s}\int_{|\xi|\geq R}|\mathcal{F}_{x,y}(f_n)(\xi,k)|^2\,d\xi
+\sum_{k\in\Z}k^{2s}\int_{|\xi|\geq R}|\xi|^{2s_\alpha}|\mathcal{F}_{x,y}(f_n)(\xi,k)|^2\,d\xi\nonumber\\
\lesssim &\,(R^{-2(s_\alpha+\kappa)}+R^{-2\kappa})\sum_{k\in\Z}k^{2s}\int_{|\xi|\geq R}|\xi|^{2(s_\alpha+\kappa)}|\mathcal{F}_{x,y}(f_n)(\xi,k)|^2\,d\xi\nonumber\\
\lesssim&\, (R^{-2(s_\alpha+\kappa)}+R^{-2\kappa})\sum_{k}k^{2(1-(s_\alpha+\kappa))}\|\mathcal{F}_{y} (f_n)(k)\|^{2(1-(s_\alpha+\kappa))}_{L_x^2}
\|\nabla_x\mathcal{F}_{y} (f_n)(k)\|^{2(s_\alpha+\kappa)}_{L_x^2}
\nonumber\\
\lesssim&\,(R^{-2(s_\alpha+\kappa)}+R^{-2\kappa})\|(k\|\mathcal{F}_{y} (f_n)(k)\|_{L_x^2})_k\|^{2(1-(s_\alpha+\kappa))}_{\ell_k^2}
\times \|(\|\nabla_x\mathcal{F}_{y} (f_n)(k)\|_{L_x^2})_k\|^{2(s_\alpha+\kappa)}_{\ell_k^2}\nonumber\\
\sim&\,(R^{-2(s_\alpha+\kappa)}+R^{-2\kappa})\|f_n\|_{L_x^2 \dot{H}_y^1}^{2(1-(s_\alpha+\kappa))}
\|\nabla_xf_n\|_{L_{x,y}^2}^{2(s_\alpha+\kappa)}\lesssim (R^{-2(s_\alpha+\kappa)}+R^{-2\kappa})A^2.\label{or2}
\end{align}
Hence there exists some $C>0$ independent of $R$, $\vare$ and $A$ such that for all sufficiently large $n$
\begin{align}
\|e^{it\Delta_{x}}(f_n-f_n^R)\|_{L_t^\ba L_{x,y}^\br (\R)}\leq CR^{-\kappa}A.
\end{align}
Let $R=\bg(\frac{4AC}{\vare}\bg)^{\frac{1}{\kappa}}$. Then by \eqref{425} we obtain
\begin{align}\label{large R}
\liminf_{n\to\infty}\|e^{it\Delta_{x}}f_n^R\|_{L_t^\ba L_{x,y}^\br (\R)}\geq \frac{\vare}{4}.
\end{align}
Next, define $\bq$ in a way such that $(\bq,\br)$ is an $L^2$-admissible pair. Also let $\theta,\beta\in(0,1)$ be given such that $\ba^{-1}=\theta \bq^{-1}$ and $\br^{-1}=\beta 2^{-1}$. Then using interpolation and Strichartz we obtain
\begin{align*}
\vare&\lesssim\liminf_{n\to\infty}\|e^{it\Delta_{x}} f^R_n\|_{L_t^\ba L_{x,y}^\br (\R)}\lesssim
\liminf_{n\to\infty}\bg(\|e^{it\Delta_{x}} f^R_n\|^\theta_{L_t^\bq L_{x,y}^\br (\R)}
\|e^{it\Delta_{x}} f^R_n\|^{1-\theta}_{L_t^\infty L_{x,y}^\br (\R)}\bg)\nonumber\\
&\lesssim \liminf_{n\to\infty}\bg(\|e^{it\Delta_{x}} f^R_n\|^\theta_{L_t^\bq L_{x}^\br H_y^s (\R)}
\|e^{it\Delta_{x}} f^R_n\|^{1-\theta}_{L_t^\infty L_{x,y}^\br (\R)}\bg)\nonumber\\
&\lesssim \liminf_{n\to\infty}\bg(\|f^R_n\|^\theta_{L_x^2 H_y^s }
\|e^{it\Delta_{x}} f^R_n\|^{\beta(1-\theta)}_{L_t^\infty L_{x,y}^2 (\R)}
\|e^{it\Delta_{x}} f^R_n\|^{(1-\beta)(1-\theta)}_{L_{t,x,y}^\infty(\R)}\bg)\nonumber\\
&\lesssim A^{\theta+\beta(1-\theta)} \liminf_{n\to\infty} \|e^{it\Delta_{x}} f^R_n\|^{(1-\beta)(1-\theta)}_{L_{t,x,y}^\infty(\R)},
\end{align*}
which in turn implies
\begin{align*}
\liminf_{n\to\infty}\|e^{it\Delta_{x}} f^R_n\|_{L_{t,x,y}^\infty(\R)}\gtrsim A^{-\frac{\theta+\beta(1-\theta)}{(1-\beta)(1-\theta)}}\vare^{\frac{1}{(1-\beta)(1-\theta)}}.
\end{align*}
Hence there exist $(t_n,x_n,y_n)_n\subset\R\times\R^d\times\T$ such that
\begin{align*}
\liminf_{n\to\infty} |e^{it_n\Delta_x}f_n^R(x_n,y_n)|\gtrsim A^{-\frac{\theta+\beta(1-\theta)}{(1-\beta)(1-\theta)}}\vare^{\frac{1}{(1-\beta)(1-\theta)}},
\end{align*}
or equivalently
\begin{align}\label{loew bound}
\liminf_{n\to\infty} |\int_{\R^d}(\mathcal{F}_x^{-1}\chi_R)(-z)e^{it_n\Delta_z}f_n(x_n+z,y_n)\,dz|\gtrsim A^{-\frac{\theta+\beta(1-\theta)}{(1-\beta)(1-\theta)}}\vare^{\frac{1}{(1-\beta)(1-\theta)}}.
\end{align}
Since $\T$ is bounded, we may w.l.o.g. assume that $y_n\equiv 0$. Next, define
\begin{align*}
h_n(x,y):= e^{it_n\Delta_x}f_n(x+x_n,y)
\end{align*}
One easily verifies that $\|h_n\|_{H_{x,y}^1}=\|f_n\|_{H_{x,y}^1}$ and by the $H_{x,y}^1$-boundedness of $(f_n)_n$ we know that there exists some $\phi\in H_{x,y}^1$ such that $h_n\rightharpoonup \phi$ weakly in $H_{x,y}^1$. \eqref{loew bound}, the weak convergence of $h_n$ to $\phi$, H\"older and the embedding $H_y^1 \hookrightarrow L_y^\infty$ yield
\begin{align*}
A^{-\frac{\theta+\beta(1-\theta)}{(1-\beta)(1-\theta)}}\vare^{\frac{1}{(1-\beta)(1-\theta)}}&\lesssim
|\int_{\R^d}(\mathcal{F}_x^{-1}\chi_R)(-z)\phi(z,0)\,dz|\lesssim \|\chi_R\|_{L_x^2}\|\phi(\cdot,0)\|_{L_x^2}\nonumber\\
&\lesssim A^{\frac{d}{2\kappa}}\vare^{-\frac{d}{2\kappa}}\|\phi\|_{L_y^\infty L_x^2}\leq
A^{\frac{d}{2\kappa}}\vare^{-\frac{d}{2\kappa}}\|\phi\|_{L_x^2 L_y^\infty } \nonumber\\
&\lesssim A^{\frac{d}{2\kappa}}\vare^{-\frac{d}{2\kappa}}\|\phi\|_{L_x^2 H_y^1}\leq
A^{\frac{d}{2\kappa}}\vare^{-\frac{d}{2\kappa}}\|\phi\|_{H_{x,y}^1}
\end{align*}
and the lower bound estimate in \eqref{cnls l2 refined strichartz decomp 1} follows. Since $H_{x,y}^1$ is a Hilbert space, we infer that for $D\in\{1,\pt_{x_i},\pt_y\}$
\begin{align*}
\|D(h_n-\phi)\|_{L_{x,y}^2}+\|D\phi\|_{L_{x,y}^2}=\|Dh_n\|_{L_{x,y}^2}+o_n(1).
\end{align*}
The equalities in \eqref{cnls l2 refined strichartz decomp 1} and \eqref{cnls l2 refined strichartz decomp 3} now follow from undoing the transformation form $h_n$ to $f_n$ and $\phi$ to $\phi_n$.
\end{proof}

\begin{remark}
By redefining the symmetry parameters suitably we may w.l.o.g. assume that
$$t_n\equiv 0\quad\text{or}\quad t_n\to \pm\infty$$
as $n\to\infty$.
\end{remark}

\begin{lemma}[Energy Pythagorean expansion]
Let $(f_n)_n$ and $(\phi_n)_n$ be the functions from Lemma \ref{refined l2 lemma 1}. Then
\begin{align}
&\|f_n\|_{\alpha+2}^{\alpha+2}=\|\phi_n\|_{\alpha+2}^{\alpha+2}+\|f_n-\phi_n\|_{\alpha+2}^{\alpha+2}+o_n(1).\label{decomp tas}
\end{align}
\end{lemma}

\begin{proof}
Assume first $t_n\to\pm\infty$. For $\beta>0$ let $\psi\in C_c^\infty (\R^d)\otimes C_{\mathrm{per}}^\infty(\T)$ such that $\|\phi-\psi\|_{H_{x,y}^1}\leq\beta$, where $\phi$ is the same function given by Lemma \ref{refined l2 lemma 1}. Define also $\psi_n:=e^{-it_n\Delta_x}\psi(x-x_n,y)$. Then by dispersive estimate on $\R^d$ we deduce
\begin{align*}
\|\psi_n\|_{L_{x,y}^{\alpha+2}}\lesssim |t_n|^{-\frac{\alpha d}{2(\alpha+2)}}\|\psi\|_{L_{y}^{\alpha+2} L_x^{\frac{\alpha+2}{\alpha+1}}}\to 0.
\end{align*}
Now let $\zeta\in C^\infty(\R^{d+1};[0,1])$ be a cut-off function such that $\mathrm{supp}\,\zeta\subset \R^{d}\times [-2\pi,2\pi]$ and $\zeta\equiv 1$ on $\R^{d}\times [-\pi,\pi]$. Then by Sobolev's inequality on $\R^d\times\T$, product rule and periodicity along the $y$-direction we deduce
\begin{align}
\|\psi_n-\phi_n\|_{L_{x,y}^{\alpha+2}(\R^d\times\T)}
&\leq\|\zeta(\psi_n-\phi_n)\|_{L_{x,y}^{\alpha+2}(\R^{d+1})}
\lesssim \|\zeta(\psi_n-\phi_n)\|_{H_{x,y}^1(\R^{d+1})}\nonumber\\
&\lesssim \|\psi_n-\phi_n\|_{H_{x,y}^1(\R^d\times\T)}\leq\beta,\label{argument gn}
\end{align}
which in turn implies $\|\phi_n\|_{L_{x,y}^{\alpha+2}}=o_n(1)$. Therefore by triangular inequality
\begin{align*}
|\|f_n\|_{\alpha+2}-\|f_n-\phi_n\|_{\alpha+2}|\leq \|\phi_n\|_{\alpha+2}=o_n(1)
\end{align*}
and \eqref{decomp tas} follows. Assume now $t_n\equiv 0$. Notice that by truncation arguments, the weak $H_{x,y}^1$-convergence also implies almost everywhere convergence. Thus by Brezis-Lieb lemma we obtain
\begin{align*}
\|h_n\|_{\alpha+2}^{\alpha+2} =\|\phi\|_{\alpha+2}^{\alpha+2}+\|h_n-\phi\|_{\alpha+2}^{\alpha+2}+o_n(1),
\end{align*}
where $h_n$ is the function given by Lemma \ref{refined l2 lemma 1}. \eqref{decomp tas} follows then by undoing the transformation.
\end{proof}

\begin{lemma}[Linear profile decomposition]\label{linear profile}
Let $(\psi_n)_n$ be a bounded sequence in $H_{x,y}^1$. Then up to a subsequence, there exist nonzero linear profiles
$(\tdu^j)_j\subset H_{x,y}^1$, remainders $(w_n^k)_{k,n}\subset H_{x,y}^1$, parameters $(t^j_n,x^j_n)_{j,n}\subset\R\times\R^d$ and $K^*\in\N\cup\{\infty\}$, such that
\begin{itemize}
\item[(i)] For any finite $1\leq j\leq K^*$ the parameter $t^j_n$ satisfies
\begin{align}
t^j_n\equiv 0\quad\text{or}\quad \lim_{n\to\infty}t^j_n= \pm\infty.
\end{align}

\item[(ii)]For any finite $1\leq k\leq K^*$ we have the decomposition
\begin{align}\label{cnls decomp lemma}
\psi_n=\sum_{j=1}^k T_n^j\phi^j(x,y)+w_n^k=:\sum_{j=1}^k e^{-it_n\Delta_x}\phi^j(x-x_n,y)+w_n^k.
\end{align}

\item[(iii)] The remainders $(w_n^k)_{k,n}$ satisfy
\begin{align}\label{cnls to zero wnk lemma}
\lim_{k\to K^*}\lim_{n\to\infty}\|e^{it\Delta_x}w_n^k\|_{L_t^\ba L_{x,y}^\br(\R)}=0.
\end{align}

\item[(iv)] The parameters are orthogonal in the sense that
\begin{align}\label{cnls orthog of pairs lemma}
|t_n^k-t_n^j|+|x_n^k-x_n^j|\to\infty
\end{align}
for any $j\neq k$.

\item[(v)] For any finite $1\leq k\leq K^*$ and $D\in\{1,\pt_{x_i},\pt_y\}$ we have the energy decompositions
\begin{align}
\|D\psi_n\|_{L_{x,y}^2}^2&=\sum_{j=1}^k\|D(T_n^j\tdu^j)\|_{L_{x,y}^2}^2+\|Dw_n^k\|_{L_{x,y}^2}^2+o_n(1),\label{orthog L2 lemma}\\
\|\psi_n\|_{\alpha+2}^{\alpha+2}&=\sum_{j=1}^k\|T_n^j\tdu^j\|_{\alpha+2}^{\alpha+2}
+\|w_n^k\|_{\alpha+2}^{\alpha+2}+o_n(1)\label{cnls conv of h lemma}.
\end{align}
\end{itemize}
\end{lemma}

\begin{proof}
We construct the linear profiles iteratively and start with $k=0$ and $w_n^0:=\psi_n$. We assume initially that the linear profile decomposition is given and its claimed properties are satisfied for some $k$. Define
\begin{align*}
\vare_{k}:=\lim_{n\to\infty}\|e^{it\Delta_x}w_n^k\|_{L_t^\ba L_{x,y}^\br(\R)}.
\end{align*}
If $\vare_k=0$, then we stop and set $K^*=k$. Otherwise we apply Lemma \ref{refined l2 lemma 1} to $w_n^k$ to obtain the sequence $(\tdu^{k+1},w_n^{k+1},t_n^{k+1},x_n^{k+1})_{n}.$ We should still need to check that the items (iii) and (iv) are satisfied for $k+1$. That the other items are also satisfied for $k+1$ follows directly from the construction of the linear profile decomposition. If $\vare_k=0$, then item (iii) is automatic; otherwise we have $K^*=\infty$. Using \eqref{cnls l2 refined strichartz decomp 1} and \eqref{orthog L2 lemma} we obtain
\begin{align*}
\sum_{j\in \N}A^{-\frac{2(\theta+\beta(1-\theta))}{(1-\beta)(1-\theta)}-\frac{d}{\kappa}}\vare^{\frac{2}{(1-\beta)(1-\theta)}+\frac{d}{\kappa}}
\lesssim \sum_{j\in \N}\|\tdu^j\|^2_{H_{x,y}^1}
=\sum_{j\in \N}\lim_{n\to\infty}\|T_n^j \phi^j\|^2_{H_{x,y}^1}
\leq \lim_{n\to\infty}\|\psi_n\|^2_{H_{x,y}^1}= A_0^2,
\end{align*}
where $A_j:=\lim_{n\to\infty}\|w_n^j\|_{H_{x,y}^1}$. Hence
\begin{align*}
A^{-\frac{2(\theta+\beta(1-\theta))}{(1-\beta)(1-\theta)}-\frac{d}{\kappa}}\vare^{\frac{2}{(1-\beta)(1-\theta)}+\frac{d}{\kappa}}\to 0\quad\text{as $j\to\infty$}.
\end{align*}
By \eqref{orthog L2 lemma} we know that $(A_j)_j$ is monotone decreasing, thus also bounded. The boundedness of $(A_j)_j$ then implies  $\vare_j\to 0$ as $j\to\infty $ and the proof of item (iii) is complete. Finally, we take item (iv). Assume inductively that item (iv) does not hold for some $j<k$ but holds for all pairs $(i_1,i_2)$ with $i_1<i_2\leq j$. We may w.l.o.g assume that the inductive basis is satisfied, otherwise from the following contradiction proof (where no inductive assumption is invoked in the base case) the algorithm already stops at $k=0$. By construction of the profile decomposition we have
\begin{align*}
w_n^{k-1}=w_n^j-\sum_{l=j+1}^{k-1} e^{-it_n^l\Delta_x} \tdu^l(x-x_n^l).
\end{align*}
Then by definition of $\tdu^k$ we know that
\begin{align*}
\tdu^k&=\wlim_{n\to\infty}e^{it_n^k\Delta_x}w_n^{k-1}(x+x_n^k,y)\nonumber\\
&=\wlim_{n\to\infty}e^{it_n^k\Delta_x}w_n^{j}(x+x_n^k,y)-\sum_{l=j+1}^{k-1}\wlim_{n\to\infty}e^{i(t_n^k-t_n^l)\Delta_x}\tdu^l(x+x_n^k-x_n^l,y),
\end{align*}
where the weak limits are taken in the $H_{x,y}^1$-topology. We aim to show $\tdu^k$ is zero, which leads to a contradiction and proves item (iv). For the first summand, we obtain that
\begin{align*}
e^{it_n^k\Delta_x}w_n^{j}(x+x_n^k,y)
=(e^{i(t_n^k-t_n^j)\Delta_x}\tau_{x_n^k-x_n^j})[e^{it_n^j\Delta_x}w_n^{j}(x+x_n^j,y)],
\end{align*}
where $\tau_{z}f(x,y):=f(x+z,y)$. On the one hand, the failure of item (iv) will lead to the strong convergence of the adjoint of $e^{i(t_n^k-t_n^j)\Delta_x}\tau_{x_n^k-x_n^j}$ in $H_{x,y}^1$. On the other hand, by construction of the profile decomposition (see the proof of Lemma \ref{refined l2 lemma 1}) we have
\begin{align*}
e^{it_n^j\Delta_x}w_n^{j}(x+x_n^j,y)\rightharpoonup 0\quad\text{in $H_{x,y}^1$}
\end{align*}
and we conclude that the first summand weakly converges to zero in $H_{x,y}^1$. Now we consider the single terms in the second summand. We can rewrite each single summand to
\begin{align*}
e^{i(t_n^k-t_n^l)\Delta_x}\tdu^l(x+x_n^k-x_n^l,y)
=(e^{i(t_n^k-t_n^j)\Delta_x}\tau_{x_n^k-x_n^j})[e^{i(t_n^j-t_n^l)\Delta_x}\tdu^l(x+x_n^j-x_n^l,y)].
\end{align*}
By the previous arguments it suffices to show that
\begin{align*}
I_n:=e^{i(t_n^j-t_n^l)\Delta_x}\tdu^l(x+x_n^j-x_n^l,y)\rightharpoonup 0\quad\text{in $H_{x,y}^1$}.
\end{align*}
Due to the inductive hypothesis we know that item (iv) is satisfied for the pair $(j,l)$. Suppose first $|t_n^j-t_n^l|\to \infty$. Then the weak convergence of $I_n$ to zero in $H_{x,y}^1$ follows immediately from the dispersive estimate. Hence we may simply assume that
$$\lim_{n\to\infty}(t_n^j-t_n^l)\in\R\quad\text{and}\quad\lim_{n\to\infty}|x_n^j-x_n^l|=\infty.$$
In this case, we utilize the fundamental fact that the symmetry group composing by unbounded translations in $\R^d$ weakly converges to zero as operators in $H_{x,y}^1$ to deduce the claim. This completes the desired proof of item (iv).
\end{proof}

\begin{remark}\label{remark interpolation}
Let $s\in(\frac12,1-s_\alpha)$ with $s_\alpha=\frac{d}{2}-\frac{2}{\alpha}\in (0,\frac12)$. Using interpolation, Strichartz, the embedding $L_y^\br\hookrightarrow L_y^2$ and arguing as in \eqref{or1} and \eqref{or2} we obtain
\begin{align}
&\,\lim_{k\to K^*}\lim_{n\to\infty}\|e^{it\Delta_{x,y}}w_n\|_{\diag H_y^{s}(\R)}\nonumber\\
\lesssim&\lim_{k\to K^*}\lim_{n\to\infty}\bg(\|e^{it\Delta_{x}}w_n\|^{1-\frac{s}{1-s_\alpha}}_{\diag L_y^2(\R)}
\|w_n\|^{\frac{s}{1-s_\alpha}}_{H_x^{s_\alpha} H_y^{1-s_\alpha}}\bg)\nonumber\\
\lesssim&\lim_{k\to K^*}\lim_{n\to\infty}\bg(\|e^{it\Delta_{x}}w_n\|^{1-\frac{s}{1-s_\alpha}}_{\diag L_y^\br(\R)}
\|w_n\|^{\frac{s}{1-s_\alpha}}_{H_{x,y}^1}\bg)
=0.\label{interpolation remainder}
\end{align}
\end{remark}

\subsection{The MEI-functional and its properties}
In this subsection we introduce the mass-energy-indicator (MEI) functional $\mD$ and state some of its very useful properties which play a fundamental role for setting up an inductive hypothesis of a contradiction proof. This was firstly introduced in \cite{killip_visan_soliton} for the study of the focusing-defocusing 3D cubic-quintic NLS and further applied in \cite{ArdilaDipolar,killip2020cubicquintic,Luo_JFA_2022,Luo_DoubleCritical,Luo_Waveguide_MassCritical,CubicQuinticPotential} for different models. Particularly, the MEI-functional is very useful for building up a multi-directional inductive hypothesis scheme and matches perfectly to the underlying framework in this paper. To begin with, we firstly define the domain $\Omega\subset \R^2$ by
\begin{align}
\Omega&:=\bg((-\infty,0]\times \R\bg)\cup\bg\{(c,h)\in\R^2:c\in(0,\infty),h\in(-\infty,m_c)\bg\},
\end{align}
where $m_c$ is defined by \eqref{def of mc}. Then we define the MEI-functional $\mD:\R^2\to [0,\infty]$ by
\begin{align}\label{cnls MEI functional}
\mD(c,h)=\left\{
             \begin{array}{ll}
             h+\frac{h+c}{\mathrm{dist}((c,h),\Omega^c)},&\text{if $(c,h)\in \Omega$},\\
             \infty,&\text{otherwise}.
             \end{array}
\right.
\end{align}
For $u\in H_{x,y}^1$, define $\mD(u):=\mD(\mM(u),\mH(u))$. We also define the set $\mA$ by
\begin{align}
\mA&:=\{u\in H_{x,y}^1:\mH(u)<m_{\mM(u)},\,\mK(u)>0\}.
\end{align}
By conservation of mass and energy we know that if $u$ is a solution of \eqref{nls}, then $\mD(u(t))$ is a conserved quantity, thus in the following we simply write $\mD(u)=\mD(u(t))$ as long as $u$ is a solution of \eqref{nls}. In the following lemma we shall also show that the NLS-flow leaves a solution starting from the set $\mA$ invariant. We will therefore write $u\in\mA$ for a solution $u$ of \eqref{nls} if $u(t)\in\mA$ for some $t$ in the lifespan of $u$.

\begin{lemma}\label{invariance from mA}
Let $u$ be a solution of \eqref{nls} and assume that there exists some $t$ in the lifespan of $u$ such that $u(t)\in\mA$. Then $u(t)\in\mA$ for all $t$ in the maximal lifespan of $u$.
\end{lemma}

\begin{proof}
Assume therefore the contrary that there exists some $s$ in the lifespan of $u$ such that $\mK(u(s))\leq 0$. By continuity of $u$ and conservation of energy we know that there exists some $t'$ lying between $t$ and $s$ such that $\mK(u(t'))=0$ and $\mH(u(t'))<m_c$. This however contradicts the minimality of $m_c$.
\end{proof}

Next, we prove that for $u\in\mA$, the energy $\mH(u)$ is equivalent to $\|\nabla_{x,y}u\|_2^2$.

\begin{lemma}\label{lemma coercivity}
Let $u\in\mA$. Then
\begin{align}
\bg(\frac12-\frac{2}{\alpha d}\bg)\|\nabla_{x,y} u\|_2^2&\leq \mH(u)\leq\frac{1}{2}\|\nabla_{x,y} u\|_2^2\label{inq coercivity}.
\end{align}
\end{lemma}

\begin{proof}
On the one hand, since the nonlinear potential energy $\|u\|_{\alpha+2}^{\alpha+2}$ has negative sign in $\mH(u)$, we have $\mH(u)\leq \frac{1}{2}\|\nabla_{x,y} u\|_2^2$. On the other hand, using $\mK(u)>0$ for $u\in\mA$ we deduce
\begin{align*}
\mH(u)>\mH(u)-\frac{2}{\alpha d}\mK(u)=\bg(\frac12-\frac{2}{\alpha d}\bg)\|\nabla_{x} u\|_2^2+\|\pt_y u\|_2^2
\geq \bg(\frac12-\frac{2}{\alpha d}\bg)\|\nabla_{x,y} u\|_2^2.
\end{align*}
\end{proof}

We end this subsection by giving some useful properties of the MEI-functional.

\begin{lemma}\label{cnls killip visan curve}
Let $u,u_1,u_2$ be functions in $H_{x,y}^1$. The following statements hold true:
\begin{itemize}
\item[(i)] $u\in\mA \Leftrightarrow\mD(u)\in(0,\infty)$.

\item[(ii)] Let $u_1,u_2\in \mA$ satisfy $\mM(u_1)\leq \mM(u_2)$ and $\mH(u_1)\leq \mH(u_2)$, then $\mD(u_1)\leq \mD(u_2)$. If in addition either $\mM(u_1)<\mM(u_2)$ or $\mH(u_1)<\mH(u_2)$, then $\mD(u_1)<\mD(u_2)$.

\item[(iii)] Let $\mD_0\in(0,\infty)$. Then
\begin{gather}
m_{\mM(u)}-\mH(u)\gtrsim_{\mD_0} 1\label{small of unaaa},\\
\mH(u)+\mM(u)\lesssim_{\mD_0}\mD(u)\label{mei var2}
\end{gather}
uniformly for all $u\in \mA$ with $\mD(u)\leq \mD_0$.
\end{itemize}
\end{lemma}

\begin{proof}
That $u\in\mA$ implies $\mD(u)\in(0,\infty)$ follows immediately from Lemma \ref{lemma coercivity}. Now assume that $\mD(u)\in(0,\infty)$. If $\mH(u)\leq 0$, then
$$ \frac{1}{2}\|\nabla_x u\|_2^2-\frac{1}{\alpha+2}\|u\|_{\alpha+2}^{\alpha+2}=\mH(u)-\frac12\|\pt_y u\|_2^2\leq\mH(u)\leq 0,$$
which in turn implies
$$  \mK(u)=\|\nabla_x u\|_2^2-\frac{\alpha d}{2(\alpha+2)}\|u\|_{\alpha+2}^{\alpha+2}\leq \bg(2-\frac{\alpha d}{2}\bg)(\alpha+2)^{-1}\|u\|_{\alpha+2}^{\alpha+2}\leq 0.$$
This however contradicts Corollary \ref{cor lower bound} and Step 1 in the proof of Theorem \ref{thm existence of ground state}. Thus $\mH(u)>0$. Since $\mD(u)<\infty$, we also infer that $\mH(u)<m_{\mM(u)}$. By definition of $m_{\mM(u)}$ and Step 1 in the proof of Theorem \ref{thm existence of ground state} we conclude that $u\in \mA$. This completes the proof of (i).

Next, (ii) follows from the monotonicity and continuity of the curve $c\mapsto m_c$ deduced from Lemma \ref{monotone lemma}.

Finally, we take (iii). Assume first that \eqref{small of unaaa} does not hold. Then we could find a sequence $(u_n)_n\subset\mA$ such that $\sup_{n\in\N}\mD(u_n)\leq \mD_0$ and
\begin{align}\label{small of un}
m_{\mM(u_n)}-\mH(u_n)=o_n(1).
\end{align}
Since $(\mM(u),m_{\mM(u)})\notin \Omega$, we have
\begin{align}
\mathrm{dist}\bg(\bg(\mM(u_n),\mH(u_n)\bg),\Omega^c\bg)\leq m_{\mM(u_n)}-\mH(u_n)=o_n(1).\label{dist upper}
\end{align}
By (i) we know that $\mH(u_n)\in (0, m_{\mM(u_n)})$. If $\liminf_{n\in\N}\mM(u_n)\gtrsim 1$, then
\begin{align}\label{on1 contra}
\liminf_{n\to\infty}\mD(u_n)\gtrsim \frac{1}{o_n(1)},
\end{align}
contradicting $\sup_{n\in\N}\mD(u_n)\leq \mD_0$. If $\mM(u_n)=o_n(1)$, then by the monotonicity of $c\mapsto m_c$ and \eqref{small of un} we know that $\liminf_{n\to\infty}\mH(u_n)\gtrsim 1$ and similarly we derive the contradiction \eqref{on1 contra} again. This completes the proof of \eqref{small of unaaa}. It remains to show \eqref{mei var2}. By definition of $\mD(u)$ we already have $\mH(u)\leq \mD(u)$. For $\mM(u)$, assume first $\mM(u)\geq 1$. By Lemma \ref{monotone lemma} we have $m_{\mM(u)}\leq m_1$. Now using the definition of $\mD(u)$, the fact $\mH(u)\geq 0$ and \eqref{dist upper} we obtain
\begin{align*}
\mD(u)\geq \frac{\mM(u)}{m_{\mM(u)}-\mH(u)}\geq \frac{\mM(u)}{m_{1}}.
\end{align*}
This implies $\mM(u)\leq m_1\mD(u)$. Next, assume $\mM(u)<1$ and $\mH(u)\leq m_1$. Then
\begin{align*}
\mathrm{dist}\bg(\bg(\mM(u),\mH(u)\bg),\Omega^c\bg)\leq (1-\mM(u))+(m_1-\mH(u))\leq (1-\mM(u))+m_1
\end{align*}
and consequently
\begin{align*}
\mD_0\geq \mD(u)\geq \frac{\mM(u)}{(1-\mM(u))+m_1}.
\end{align*}
Rearranging terms we obtain
\begin{align*}
\mM(u)\leq\frac{1+m_1}{1+\mD_0}\mD_0\leq \frac{1+m_1}{1+\mD_0}\mD(u).
\end{align*}
Finally, assume $\mM(u)<1$ and $\mH(u)>m_1$. Then
\begin{align*}
\mathrm{dist}\bg(\bg(\mM(u),\mH(u)\bg),\Omega^c\bg)\leq 1-\mM(u)
\end{align*}
and proceeding as before we conclude that
\begin{align*}
\mM(u) \leq \frac{1}{1+\mD_0}\mD(u).
\end{align*}
This completes the desired proof.
\end{proof}

\subsection{Existence of a minimal blow-up solution}
Having all the preliminaries we are ready to construct a minimal blow-up solution of \eqref{nls}. According to Remark \ref{remark interpolation} we fix some $s\in(\frac12,1-s_\alpha)$ with $s_\alpha=\frac{d}{2}-\frac{2}{\alpha}\in (0,\frac12)$. Define
\begin{align*}
\tau(\mD_0):=\sup\bg\{\|u\|_{\diag H_y^{s}(I_{\max})}:
\text{ $u$ is solution of \eqref{nls}, }\mD(u)\in (0,\mD_0)\bg\}
\end{align*}
and
\begin{align}\label{introductive hypothesis}
\mD^*&:=\sup\{\mD_0>0:\tau(\mD_0)<\infty\}.
\end{align}
By Lemma \ref{lemma small data}, \ref{lemma scattering norm}, \ref{lemma coercivity}, and \ref{cnls killip visan curve} we know that $\mD^*>0$ and $\tau(\mD_0)<\infty$ for sufficiently small $\mD_0$. Therefore we simply assume $\mD^*<\infty$, relying on which we derive a contradiction. This in turn ultimately implies $\mD^*=\infty$ and the proof of Theorem \ref{main thm} will be complete in view of Lemma \ref{cnls killip visan curve}. By the inductive hypothesis we can find a sequence $(u_n)_n$ which are solutions of \eqref{nls} with $(u_n(0))_n\subset {\mA}$ and maximal lifespan $(I_{n})_n$ such that
\begin{gather}
\lim_{n\to\infty}\|u_n\|_{\diag H_y^{s}((\inf I_n,0])}=\lim_{n\to\infty}\|u_n\|_{\diag H_y^{s}([0, \sup I_n))}=\infty,\label{oo1}\\
\lim_{n\to\infty}\mD(u_n)=\mD^*.\label{oo2}
\end{gather}
Up to a subsequence we may also assume that
\begin{align*}
(\mM(u_n),\mH(u_n))\to(\mM_0,\mH_0)\quad\text{as $n\to\infty$}.
\end{align*}
By continuity of $\mD$ and finiteness of $\mD^*$ we know that
\begin{align*}
\mD^*=\mD(\mM_0,\mH_0),\quad
\mM_0\in(0,\infty),\quad
\mH_0\in[0,m_{\mM_0}).
\end{align*}
From Lemma \ref{lemma coercivity} and \ref{cnls killip visan curve} it follows that $(u_n(0))_n$ is a bounded sequence in $H_{x,y}^1$, hence Lemma \ref{linear profile} and Remark \ref{remark interpolation} are applicable for $(u_n(0))_n$: There exist nonzero linear profiles
$(\tdu^j)_j\subset H_{x,y}^1$, remainders $(w_n^k)_{k,n}\subset H_{x,y}^1$, parameters $(t^j_n,x^j_n)_{j,n}\subset\R\times\R^d$ and $K^*\in\N\cup\{\infty\}$, such that
\begin{itemize}
\item[(i)] For any finite $1\leq j\leq K^*$ the parameter $t^j_n$ satisfies
\begin{align}
t^j_n\equiv 0\quad\text{or}\quad \lim_{n\to\infty}t^j_n= \pm\infty.
\end{align}

\item[(ii)]For any finite $1\leq k\leq K^*$ we have the decomposition
\begin{align}\label{cnls decomp}
u_n(0)=\sum_{j=1}^k T_n^j\phi^j(x,y)+w_n^k=:\sum_{j=1}^k e^{-it_n\Delta_x}\phi^j(x-x_n,y)+w_n^k.
\end{align}

\item[(iii)] The remainders $(w_n^k)_{k,n}$ satisfy
\begin{align}\label{cnls to zero wnk}
\lim_{k\to K^*}\lim_{n\to\infty}\|e^{it\Delta_{x,y}}w_n^k\|_{L_t^\ba L_{x}^\br H_y^s(\R)}=0.
\end{align}

\item[(iv)] The parameters are orthogonal in the sense that
\begin{align}\label{cnls orthog of pairs}
|t_n^k-t_n^j|+|x_n^k-x_n^j|\to\infty
\end{align}
for any $j\neq k$.

\item[(v)] For any finite $1\leq k\leq K^*$ and $D\in\{1,\pt_{x_i},\pt_y\}$ we have the energy decompositions
\begin{align}
\|D(u_n(0))\|_{L_{x,y}^2}^2&=\sum_{j=1}^k\|D(T_n^j\tdu^j)\|_{L_{x,y}^2}^2+\|Dw_n^k\|_{L_{x,y}^2}^2+o_n(1),\label{orthog L2}\\
\|u_n(0)\|_{\alpha+2}^{\alpha+2}&=\sum_{j=1}^k\|T_n^j\tdu^j\|_{\alpha+2}^{\alpha+2}
+\|w_n^k\|_{\alpha+2}^{\alpha+2}+o_n(1)\label{cnls conv of h}.
\end{align}
\end{itemize}
We now define the nonlinear profiles as follows:
\begin{itemize}
\item For $t^k_\infty=0$, we define $u^k$ as the solution of \eqref{nls} with $u^k(0)=\tdu^k$.

\item For $t^k_\infty\to\pm\infty$, we define $u^k$ as the solution of \eqref{nls} that scatters forward (backward) to $e^{-it\Delta_{x,y}}\tdu^k$ in $H_{x,y}^1$.
\end{itemize}
In both cases we define
\begin{align*}
u_n^k:=u^k(t-t^k_n,x-x_n^k,y).
\end{align*}
Then $u_n^k$ is also a solution of \eqref{nls}. In both cases we have for each finite $1\leq k \leq K^*$
\begin{align}\label{conv of nonlinear profiles in h1}
\lim_{n\to\infty}\|u_n^k(0)-T_n^k \tdu^k\|_{H_{x,y}^1}=0.
\end{align}

In the following, we establish a Palais-Smale type lemma which is essential for the construction of the minimal blow-up solution.

\begin{lemma}[Palais-Smale-condition]\label{Palais Smale}
Let $(u_n)_n$ be a sequence of solutions of \eqref{nls} with maximal lifespan $I_n$, $u_n\in\mA$ and $\lim_{n\to\infty}\mD(u_n)=\mD^*$. Assume also that there exists a sequence $(t_n)_n\subset\prod_n I_n$ such that
\begin{align}\label{precondition}
\lim_{n\to\infty}\|u_n\|_{\diag H_y^{s}((\inf I_n,\,t_n])}=\lim_{n\to\infty}\|u_n\|_{\diag H_y^{s}([t_n,\,\sup I_n)}=\infty.
\end{align}
Then up to a subsequence, there exists a sequence $(x_n)_n\subset\R^2$ such that $(u_n(t_n, \cdot+x_n,y))_n$ strongly converges in $H_{x,y}^1$.
\end{lemma}

\begin{proof}
By time translation invariance we may assume that $t_n\equiv 0$. Let $(u_n^j)_{j,n}$ be the nonlinear profiles corresponding to the linear profile decomposition of $(u_n(0))_n$. Define
\begin{align*}
\Psi_n^k:=\sum_{j=1}^k u_n^j+e^{it\Delta_{x,y}}w_n^k.
\end{align*}
We shall show that there exists exactly one non-trivial ``bad'' linear profile, relying on which the desired claim follows. We divide the remaining proof into four steps.
\subsubsection*{Step 1: Decomposition of energies of the linear profiles}
We first show that for a given nonzero linear profile $\phi^j$ we have
\begin{align}
\mH(T_n^j\phi^j)&> 0,\label{bd for S}\\
\mK(T_n^j\phi^j)&> 0\label{pos of K}
\end{align}
for all sufficiently large $n=n(j)\in\N$. Since $\phi^j\neq 0$ we know that $T_n^j\phi^j\neq 0$ for all sufficiently large $n$. Suppose now that \eqref{pos of K} does not hold. Up to a subsequence we may assume that $\mK(T_n^j \phi^j)\leq 0$ for all sufficiently large $n$. Recall the energy functional $\mI$ defined by \eqref{def of mI}. Using \eqref{orthog L2} and \eqref{cnls conv of h} we infer that
\begin{align}
\mI(u_n(0))&=\sum_{j=1}^k\mI(T_n^j\tdu^j)+\mI(w_n^k)+o_n(1)\label{conv of i}.
\end{align}
By the non-negativity of $\mI$, \eqref{conv of i} and \eqref{small of unaaa} we know that there exists some sufficiently small $\delta>0$ depending on $\mD^*$ and some sufficiently large $N_1$ such that for all $n>N_1$ we have
\begin{align}\label{contradiction1}
\tm_{\mM(T_n^j\phi^j)}\leq\mI(T_n^j\phi^j)\leq \mI(u_n(0))+\delta
\leq\mH(u_n(0))+\delta\leq m_{\mM(u_n(0))}-2\delta,
\end{align}
where $\tm$ is the quantity defined by \eqref{mtilde equal m}. By continuity of $c\mapsto m_c$ we also know that for sufficiently large $n$ we have
\begin{align}\label{contradiction3}
m_{\mM(u_n(0))}-2\delta\leq m_{\mM_0}-\delta.
\end{align}
Using \eqref{orthog L2} we deduce that for any $\vare>0$ there exists some large $N_2$ such that for all $n>N_2$ we have
\begin{align*}
\mM(T_n^j\phi^j)\leq \mM_0+\vare.
\end{align*}
From the continuity and monotonicity of $c\mapsto m_c$ and Step 1 in the proof of Theorem \ref{thm existence of ground state}, we may choose some sufficiently small $\vare$ to see that
\begin{align}\label{contradiction2}
\tm_{\mM(T_n^j\phi^j)}=m_{\mM(T_n^j\phi^j)}\geq m_{\mM_0+\vare}\geq m_{\mM_0}-\frac{\delta}{2}.
\end{align}
Now \eqref{contradiction1}, \eqref{contradiction3} and \eqref{contradiction2} yield a contradiction. Thus \eqref{pos of K} holds, which combining with Lemma \ref{lemma coercivity} also yields \eqref{bd for S}. Similarly, for each $j\in\N$ we have
\begin{align}
\mH(w_n^j)&> 0,\label{bd for S wnj} \\
\mK(w_n^j)&> 0\label{pos of K wnj}
\end{align}
for sufficiently large $n$.
\subsubsection*{Step 2: Decoupling of nonlinear profiles}
In this step, we show that the nonlinear profiles are asymptotically decoupled in the sense that
\begin{align}\label{smallness aaa}
\lim_{n\to\infty} \|\la u_n^i,u_n^j\ra_{H_y^s}\|_{L_t^{\frac{\ba}{2}}L_x^{\frac{\br}{2}}(\R)}=0
\end{align}
for any fixed $1\leq i,j\leq K^*$ with $i\neq j$, provided that
$$\limsup_{n\to\infty}\,(\|u_n^i\|_{\diag H_y^{s}(\R)}+\|u_n^j\|_{\diag H_y^{s}(\R)})<\infty.$$
We claim that for any $\beta>0$ there exists some $\psi^i_\beta,\psi_\beta^j\in C_c^\infty(\R\times\R^d)\otimes C_{\mathrm{per}}^\infty(\T)$ such that
\begin{align*}
\bg\|u^i_n-\psi^i_\beta(t-t^i_n,x-x^i_n,y)\bg\|_{\diag H_y^{s}(\R)}\leq \beta,\\
\bg\|u^j_n-\psi^j_\beta(t-t^j_n,x-x^j_n,y)\bg\|_{\diag H_y^{s}(\R)} \leq \beta.
\end{align*}
Indeed, we may simply choose some $\psi^i_\beta,\psi^j_\beta\in C_c^\infty(\R\times\R^d)$ such that
\begin{align*}
\|u^i-\psi^i_\beta\|_{\diag H_y^{s}(\R)}\leq \beta,\,\|u^j-\psi^j_\beta\|_{\diag H_y^{s}(\R)}\leq \beta,
\end{align*}
where $u^k$ are solutions of \eqref{nls} with $u^k(0)=\phi^k$ ($t^k_\infty=0$) or $u^k$ scatters forward (backward) to $e^{-it\Delta_{x,y}}\phi^k$ in $H_{x,y}^1$ ($t^k_n\to\pm\infty$). Define now
$$ \Lambda_n (\psi_\beta^\ell):=\psi^\ell_\beta(t-t^\ell_n,x-x^\ell_n,y)$$
for $\ell\in\{i,j\}$. Using H\"older we infer that
\begin{align*}
\|\la u_n^i,u_n^j\ra_{H_y^s}\|_{L_t^{\frac{\ba}{2}}L_x^{\frac{\br}{2}}(\R)}\lesssim \beta+\|\la \Lambda_n (\psi_\beta^i),\Lambda_n (\psi_\beta^j)\ra_{H_y^s}\|_{L_t^{\frac{\ba}{2}}L_x^{\frac{\br}{2}}(\R)}.
\end{align*}
Since $\beta$ can be chosen arbitrarily small, it suffices to show
\begin{align}\label{step2a1}
\lim_{n\to\infty}\|\la \Lambda_n (\psi_\beta^i),\Lambda_n (\psi_\beta^j)\ra_{H_y^s}\|_{L_t^{\frac{\ba}{2}}L_x^{\frac{\br}{2}}(\R)}=0.
\end{align}
However, since $\psi^i_\beta$ and $\psi^j_\beta$ are compactly supported in $\R_t\times\R^d_x$, \eqref{step2a1} follows immediately from the orthogonality of the parameters $(t_n^k)_{k,n}$ and $(x_n^k)_{k,n}$ given by \eqref{cnls orthog of pairs}.
\subsubsection*{Step 3: Existence of at least one bad profile}
First we prove that there exists some $1\leq J\leq K^*$ such that for all $j\geq J+1$ and all sufficiently large $n$, $u_n^j$ is global and
\begin{align}\label{uniform bound of unj}
\lim_{n\to\infty}\sum_{J+1\leq j\leq K^*}\|u_n^j\|^2_{\diag H_y^{s}(\R)}\lesssim 1.
\end{align}
Indeed, using \eqref{orthog L2} we infer that
\begin{align}\label{small initial data}
\lim_{k\to K^*}\lim_{n\to\infty}\sum_{j=1}^k\|T_n^j \tdu^j\|^2_{H_{x,y}^1}<\infty.
\end{align}
Then \eqref{uniform bound of unj} follows from Lemma \ref{lemma small data} and \ref{lemma scattering norm}. We now claim that there exists some $1\leq J_0\leq J$ such that
\begin{align}
\limsup_{n\to\infty}\|u_n^{J_0}\|_{\diag H_y^{s}(\R)}=\infty.
\end{align}
We argue by contradiction and assume that
\begin{align}\label{uniform bound of unj small}
\limsup_{n\to\infty}\|u_n^j\|_{\diag H_y^{s}(\R)}<\infty\quad\forall\,1\leq j\leq J.
\end{align}
To proceed, we first show
\begin{align}\label{kkkk uniform bound of unj}
\sup_{J+1\leq k\leq K^*}\lim_{n\to\infty}\bg\|\sum_{j=J+1}^k u_n^j\bg\|_{\diag H_y^{s}(\R)}
\lesssim 1.
\end{align}
Indeed, using triangular inequality, \eqref{smallness aaa} and \eqref{uniform bound of unj} we immediately obtain
\begin{align*}
&\,\sup_{J+1\leq k\leq K^*}\lim_{n\to\infty}\bg\|\sum_{j=J+1}^k u_n^j\bg\|_{\diag H_y^s(\R)}\nonumber\\
\lesssim&\, \sup_{J+1\leq k\leq K^*}\lim_{n\to\infty}\bg(\bg(\sum_{j=J+1}^k
\|u_n^j\|^2_{\diag H_y^{s}(\R)}\bg)^{\frac12}
+\bg(\sum_{i,j=J+1,i\neq j}^k\|\la u_n^i, u_n^j\ra_{H_y^s}\|_{L_t^{\frac{\ba}{2}}L_x^{\frac{\br}{2}}(\R)}\bg)^{\frac12}\bg)\lesssim 1.
\end{align*}
Combining \eqref{kkkk uniform bound of unj} with \eqref{uniform bound of unj small} we deduce
\begin{align}\label{super uniform}
\sup_{1\leq k\leq K^*}\lim_{n\to\infty}\bg\|\sum_{j=1}^k u_n^j\bg\|_{\diag H_y^{s}(\R)}
\lesssim  1.
\end{align}
Therefore, using \eqref{orthog L2}, \eqref{conv of nonlinear profiles in h1} and Strichartz we confirm that the conditions \eqref{cond 1} and \eqref{cond 3} are satisfied for sufficiently large $k$ and $n$, where we set $u=u_n$ and $\tilde{u}=\Psi_n^k$ in Lemma \ref{lem stability cnls}. As long as we can show that \eqref{cond 2} is satisfied for the chosen $s\in(\frac{1}{2},1-s_\alpha)$, we are able to apply Lemma \ref{lem stability cnls} to obtain the contradiction
\begin{align}\label{contradiction 1}
\limsup_{n\to\infty}\|u_n\|_{\diag H_y^{s}(\R)}<\infty.
\end{align}
Direct calculation shows
\begin{align*}
e&=\,i\pt_t\Psi_n^k+\Delta_{x,y}\Psi_n^k+|\Psi_n^k|^{2}\Psi_n^k\nonumber\\
&=\bg(\sum_{j=1}^k (i\pt_tu_n^j+\Delta_{x,y} u_n^j)+|\sum_{j=1}^ku_n^j|^{\alpha}\sum_{j=1}^k u_n^j\bg)
+\bg(|\Psi_n^k|^{\alpha}\Psi_n^k-|\Psi_n^k-e^{it\Delta_{x,y}}w_n^k|^{\alpha}(\Psi_n^k-e^{it\Delta_{x,y}}w_n^k)\bg)\nonumber\\
&=:I_1+I_2.
\end{align*}
In the following we show the asymptotic smallness of $I_1$ and $I_2$. Since $u_n^j$ solves \eqref{nls}, we can rewrite $I_1$ to
\begin{align*}
I_1
=-\bg(\sum_{j=1}^k|u_n^j|^{\alpha}u_n^j-\bg|\sum_{j=1}^ku_n^j\bg|^{\alpha}\sum_{j=1}^ku_n^j\bg)
\lesssim \sum_{i,j=1,i\neq j}^k|u_n^i|^{\alpha}|u_n^j|=:I_{11}.
\end{align*}
By H\"older we have
\begin{align*}
\||u_n^i|^\alpha|u_n^j|\|_{L_t^{\bb'}L_x^{\bs'}L_y^2(\R)}&\lesssim \|u_n^{i}u_n^{j}\|^{\frac12}_{L_t^{\frac{\ba}{2}}L_t^{\frac{\br}{2}}L_y^1(\R)}
\|u_n^i\|^{\alpha-\frac12}_{\diag L_y^\infty}\|u_n^j\|^{\frac12}_{\diag L_y^\infty}\nonumber\\
&\lesssim \|u_n^iu_n^j\|^{\frac12}_{L_t^{\frac{\ba}{2}}L_t^{\frac{\br}{2}}L_y^1}
\|u_n^i\|^{\alpha-\frac12}_{\diag H_y^1}\|u_n^j\|^{\frac12}_{\diag H_y^1}.
\end{align*}
Then \eqref{uniform bound of unj}, \eqref{uniform bound of unj small} and \eqref{smallness aaa} imply
\begin{align*}
\lim_{k\to K^*}\lim_{n\to\infty}\| I_{11}\|_{L_t^{\bb'}L_x^{\bs'}L_y^2(\R) }=0.
\end{align*}
On the other hand, by Lemma \ref{lemma scattering norm} and the telescoping arguments given in the proof of Lemma \ref{lem stability cnls} we obtain for $s'\in(s,1-s_\alpha)$
\begin{align*}
\||u_n^i|^\alpha u_n^j\|_{L_t^{\bb'}L_x^{\bs'}H_y^{s'}(\R)}\lesssim \|u_n^i\|_{\diag H_y^{s'}(\R)}^\alpha\|u_n^j\|_{\diag H_y^{s'}(\R)}\lesssim1.
\end{align*}
Combining with the inequality $\|f\|_{H_y^{s}}\leq\|f\|_{L_y^2}^{1-s/s'}\|f\|_{H_y^{s'}}^{s/s'}$ we infer that
\begin{align*}
\lim_{k\to K^*}\lim_{n\to\infty}\| I_{11}\|_{L_t^{\bb'}L_x^{\bs'}H_y^{s}(\R)}=0
\end{align*}
Next, we prove the asymptotic smallness of $I_2$. Direct calculation shows
\begin{align*}
I_2=O\bg(\Psi_n^k(e^{it\Delta_{x,y}}w_n^k)^\alpha+(\Psi_n^k)^\alpha e^{it\Delta_{x,y}}w_n^k+(e^{it\Delta_{x,y}}w_n^k)^{\alpha+1}\bg).
\end{align*}
But then \eqref{super uniform}, \eqref{cnls to zero wnk}, Lemma \ref{fractional on t}, H\"older and the telescoping arguments given in the proof of Lemma \ref{lem stability cnls} immediately yield
\begin{align*}
\lim_{k\to K^*}\lim_{n\to\infty}\| I_2\|_{L_t^{\bb'}L_x^{\bs'}H_y^{s}(\R)}=0
\end{align*}
and Step 3 is complete.
\subsubsection*{Step 4: Reduction to one bad profile and conclusion}
From Step 3 we conclude that there exists some $1\leq J_1\leq K^*$ such that
\begin{align}
\limsup_{n\to\infty}\|u_n^j\|_{\diag H_y^s(\R)}&=\infty\quad \forall \,1\leq j\leq J_1,\label{infinite}\\
\limsup_{n\to\infty}\|u_n^j\|_{\diag H_y^s(\R)}&<\infty\quad \forall \,J_1+1\leq j\leq K^*.
\end{align}
Using \eqref{orthog L2} and \eqref{cnls conv of h} we infer that for any finite $1\leq k\leq K^*$
\begin{align}
\mM_0&=\sum_{j=1}^k \mM(T_n^j\tdu^j)+\mM(w_n^k)+o_n(1),\label{mo sum}\\
\mH_0&=\sum_{j=1}^k \mH(T_n^j\tdu^j)+\mH(w_n^k)+o_n(1)\label{eo sum}.
\end{align}
If $J_1>1$, then using \eqref{mo sum}, \eqref{eo sum}, the asymptotic positivity of energies deduced from Step 1 and Lemma \ref{cnls killip visan curve} we know that $\limsup_{n\to\infty}\mD(u_n^1)<\mD^*$, which violates \eqref{infinite} due to the inductive hypothesis. Thus $J_1=1$ and
$$ u_n(0,x)=e^{-it_n^1 \Delta_{x,y}}\tdu^1(x-x_n^1)+w_n^1(x).$$
Similarly, we must have $\mM(w_n^1)=o_n(1)$ and $\mH(w_n^1)=o_n(1)$, otherwise we could deduce again the contradiction \eqref{contradiction 1} using Lemma \ref{lem stability cnls}. Combining with Lemma \ref{cnls killip visan curve} we conclude that $\|w_n^1\|_{H_{x,y}^1}=o_n(1)$. Finally, we exclude the cases $t^1_n\to\pm \infty$. We only consider the case $t_n^1\to \infty$, the case $t_n^1 \to -\infty$ can be similarly dealt. Indeed, using Strichartz we obtain
\begin{align*}
\|e^{-it\Delta_{x,y}}u_n(0)\|_{\diag H_y^s([0,\infty))}\lesssim
\|e^{-it\Delta_{x,y}}\tdu^1\|_{\diag H_y^s([t^1_n,\infty))}+\|w_n^1\|_{H^1}\to 0.
\end{align*}
Using Lemma \ref{lemma small data} we derive the contradiction \eqref{contradiction 1} again. This completes the desired proof.
\end{proof}

\begin{lemma}[Existence of a minimal blow-up solution]\label{category 0 and 1}
Suppose that $\mD^*\in(0,\infty)$. Then there exists a global solution $u_c$ of \eqref{nls} such that $\mD(u_c)=\mD^*$ and
\begin{align}
\|u_c\|_{\diag H_y^s((-\infty,0])}=\|u_c\|_{\diag H_y^s([0,\infty))}=\infty.
\end{align}
Moreover, $u_c$ is almost periodic in $H_{x,y}^1$ modulo $\R_x^d$-translations, i.e. the set $\{u(t):t\in\R\}$ is precompact in $H_{x,y}^1$ modulo translations w.r.t. the $x$-variable.
\end{lemma}

\begin{proof}
As discussed at the beginning of this section, under the assumption $\mD^*<\infty$ one can find a sequence $(u_n)_n$ of solutions of \eqref{nls} that satisfies the preconditions of Lemma \ref{Palais Smale}. We apply Lemma \ref{Palais Smale} to infer that $(u_n(0))_n$ (up to modifying time and space translation) is precompact in $H_{x,y}^1$. We denote its strong $H_{x,y}^1$-limit by $\psi$. Let $u_c$ be the solution of \eqref{nls} with $u_c(0)=\psi$. Then $\mD(u_c(t))=\mD(\psi)=\mD^*$ for all $t$ in the maximal lifespan $I_{\max}$ of $u_c$ (recall that $\mD$ is a conserved quantity).

We first show that $u_c$ is a global solution. We only show that $s_0:=\sup I_{\max}=\infty$, the negative direction can be similarly proved. If this does not hold, then by Lemma \ref{lemma small data} there exists a sequence $(s_n)_n\subset \R$ with $s_n\to s_0$ such that
\begin{align*}
\lim_{n\to\infty}\|u_c\|_{\diag H_y^s((-\inf I_{\max},s_n])}=\lim_{n\to\infty}\|u_c\|_{\diag H_y^s([s_n,\sup I_{\max}))}=\infty.
\end{align*}
Define $v_n(t):=u_c(t+s_n)$. Then \eqref{precondition} is satisfied with $t_n\equiv 0$. We then apply Lemma \ref{Palais Smale} to the sequence $(v_n(0))_n$ to conclude that there exists some $\varphi\in H_{x,y}^1$ such that, up to modifying the space translation, $u_c(s_n)$ strongly converges to $\varphi$ in $H_{x,y}^1$. But then using Strichartz we obtain for $s\in(\frac12,1-s_\alpha)$
\begin{align*}
\|e^{it\Delta_{x,y}}u_c(s_n)\|_{\diag H_y^s([0,s_0-s_n))}=\|e^{it\Delta_{x,y}}\varphi\|_{\diag H_y^s([0,s_0-s_n))}+o_n(1)=o_n(1).
\end{align*}
By Lemma \ref{lemma small data} we can extend $u_c$ beyond $s_0$, which contradicts the maximality of $s_0$. Now by \eqref{oo1} and Lemma \ref{lem stability cnls} it is necessary that
\begin{align}\label{blow up uc}
\|u_c\|_{\diag H_y^s((-\infty,0])}=\|u_c\|_{\diag H_y^s([0,\infty))}=\infty.
\end{align}

We finally show that the orbit $\{u_c(t):t\in\R\}$ is precompact in $H_{x,y}^1$ modulo ($\R^2_x$)-translations. Let $(\tau_n)_n\subset\R$ be an arbitrary time sequence. Then \eqref{blow up uc} implies
\begin{align*}
\|u_c\|_{\diag H_{y}^s((-\infty,\tau_n])}=\|u_c\|_{\diag H_{y}^s([\tau_n,\infty))}=\infty.
\end{align*}
The claim follows by applying Lemma \ref{Palais Smale} to $(u_c(\tau_n))_n$.
\end{proof}

We end this section by establishing some useful properties of the minimal blow-up solution $u_c$.

\begin{lemma}\label{lemma property of uc}
Let $u_c$ be the minimal blow-up solution given by Lemma \ref{category 0 and 1}. Then
\begin{itemize}
\item[(i)] There exists a center function $x:\R\to \R^d$ such that for each $\vare>0$ there exists $R>0$ such that
\begin{align}
\int_{|x+x(t)|\geq R}|\nabla_{x,y} u_c(t)|^2+|u_c(t)|^2+|u_c(t)|^{\alpha+2}\,dxdy\leq\vare\quad\forall\,t\in\R.
\end{align}

\item[(ii)]There exists some $\delta>0$ such that $\inf_{t\in\R}\mK(u_c(t))=\delta$.
\end{itemize}
\end{lemma}

\begin{proof}
(i) is an immediate consequence of the Gagliardo-Nirenberg inequality on $\R^d\times\T$ and the almost periodicity of $u_c$ in $H_{x,y}^1$. We follow the same lines in the proof of \cite[Prop. 10.3]{killip_visan_soliton} to show (ii). Assume the contrary that the claim does not hold. Then we can find a sequence $(t_n)_n\subset\R$ such that $\mK(u_c(t_n))\to 0$. By the almost periodicity of $u_c$ in $H_{x,y}^1$ we can find some $v_0$ such that $u(t_n)$ converges strongly to $v_0$ in $H_{x,y}^1$ (modulo $x$-translations). Combining with the continuity of $\mK$ in $H_{x,y}^1$ we infer that
\begin{align*}
\mD(v_0)&=\mD(u_c)=\mD^*\in(0,\infty),\\
\mK(v_0)&=\lim_{n\to\infty}\mK(u_c(t_n))=0.
\end{align*}
This is however impossible due to Lemma \ref{cnls killip visan curve} (i).
\end{proof}

\subsection{Extinction of the minimal blow-up solution}
In this subsection we close the proof of Theorem \ref{main thm} by showing the contradiction that the minimal blow-up solution $u_c$ must be equal to zero. We shall firstly record a key result concerning the growth of the center function $x(t)$ given in Lemma \ref{lemma property of uc}. The proof of the result relies on the conservation of momentum and was initially given in \cite{non_radial}. Similar results under the framework of MEI-functional and in the setting of NLS on a waveguide were proved in \cite{killip_visan_soliton} and \cite{HaniPausader} respectively. The proof in our case is a straightforward combination of the arguments in \cite{non_radial,killip_visan_soliton,HaniPausader} and we thus omit the details here.

\begin{lemma}\label{holmer}
Let $x(t)$ be the center function given by Lemma \ref{lemma property of uc}. Then $x(t)$ obeys the decay condition $x(t)=o(t)$ as $|t|\to\infty$.
\end{lemma}

We are now ready to prove Theorem \ref{main thm}.

\begin{proof}[Proof of Theorem \ref{main thm}]
We firstly prove the statement concerning the global well-posedness of \eqref{nls} below ground states. Since \eqref{nls} is energy-subcritical, it is well-known (see for instance \cite{TzvetkovVisciglia2016}) that global well-posedness is equivalent to the statement that for all $n\in\N$ we have
\begin{align}
\sup_{t\in[-n,n]}\|\nabla_{x,y}u(t)\|_{2}<\infty,
\end{align}
which follows immediately from Lemma \ref{cnls killip visan curve}.

Let us now restrict ourselves to the case $d\leq 4$ and prove the scattering result below ground states. As mentioned previously, we show the contradiction that the minimal blow-up solution $u_c$ given by Lemma \ref{category 0 and 1} is equal to zero. Let $\chi:\R^d\to\R$ be a smooth radial cut-off function satisfying
\begin{align*}
\chi=\left\{
             \begin{array}{ll}
             |x|^2,&\text{if $|x|\leq 1$},\\
             0,&\text{if $|x|\geq 2$}.
             \end{array}
\right.
\end{align*}
For $R>0$, we define the local virial action $z_R(t)$ by
\begin{align*}
z_{R}(t):=\int R^2\chi\bg(\frac{x}{R}\bg)|u_c(t,x)|^2\,dxdy.
\end{align*}
Direct calculation yields
\begin{align}
\pt_t z_R(t)=&\,2\,\mathrm{Im}\int R\nabla_x \chi\bg(\frac{x}{R}\bg)\cdot\nabla_x u_c(t)\bar{u}_c(t)\,dxdy,\label{final4}\\
\pt_t^2 z_R(t)=&\,4\int \pt^2_{x_j x_k}\chi\bg(\frac{x}{R}\bg)\pt_{x_j} u_c\pt_{x_k}\bar{u}_c\,dxdy-\frac{1}{R^2}\int\Delta_x^2\chi\bg(\frac{x}{R}\bg)|u_c|^2\,dxdy\nonumber\\
-&\,\frac{2\alpha }{\alpha+2}\int\Delta_x\chi\bg(\frac{x}{R}\bg)|u_c|^{\alpha+2}\,dxdy.
\end{align}
We then obtain
\begin{align*}
\pt_t^2 z_R(t)=8\mK(u_c(t))+A_R(u_c(t)),
\end{align*}
where
\begin{align*}
A_R(u_c(t))=&\,4\int\bg(\pt^2_{x_j}\chi\bg(\frac{x}{R}\bg)-2\bg)|\pt_{x_j} u_c|^2\,dxdy+4\sum_{j\neq k}\int_{R\leq|x|\leq 2R}\pt_{x_j}\pt_{x_k}\chi\bg(\frac{x}{R}\bg)\pt_{x_j} u_c\pt_{x_k}\bar{u}_c\,dxdy\nonumber\\
-&\,\frac{1}{R^2}\int\Delta_x^2\chi\bg(\frac{x}{R}\bg)|u_c|^2\,dxdy
-\frac{2\alpha}{\alpha+2}\int\bg(\Delta_x\chi\bg(\frac{x}{R}\bg)-2d\bg)|u_c|^{\alpha+2}\,dxdy.
\end{align*}
We have the rough estimate
\begin{align*}
|A_R(u(t))|\leq C_1\int_{|x|\geq R}|\nabla_x u_c(t)|^2+\frac{1}{R^2}|u_c(t)|^2+|u_c(t)|^{\alpha+2}\,dxdy
\end{align*}
for some $C_1>0$. By Lemma \ref{lemma property of uc} we know that there exists some $\delta>0$ such that
\begin{align}\label{small extinction ff}
\inf_{t\in\R}(8\mK(u_c(t)))\geq 8\delta=:2\eta_1>0.
\end{align}
From Lemma \ref{lemma property of uc} it also follows that there exists some $R_0\geq 1$ such that
\begin{align*}
\int_{|x+x(t)|\geq R_0}|\nabla_{x,y} u_c(t)|^2+|u_c(t)|^2+|u_c(t)|^{\alpha+2}\,dxdy\leq \frac{\eta_1}{C_1}.
\end{align*}
Thus for any $R\geq R_0+\sup_{t\in[t_0,t_1]}|x(t)|$ with some to be determined $t_0,t_1\in[0,\infty)$, we have
\begin{align}\label{final3}
\pt_t^2 z_R(t)\geq \eta_1
\end{align}
for all $t\in[t_0,t_1]$. By Lemma \ref{holmer} we know that for any $\eta_2>0$ there exists some $t_0\gg 1$ such that $|x(t)|\leq\eta_2 t$ for all $t\geq t_0$. Now set $R=R_0+\eta_2 t_1$. Integrating \eqref{final3} over $[t_0,t_1]$ yields
\begin{align}\label{12}
\pt_t z_R(t_1)-\pt_t z_R(t_0)\geq \eta_1 (t_1-t_0).
\end{align}
Using \eqref{final4}, Cauchy-Schwarz and Lemma \ref{cnls killip visan curve} we have
\begin{align}\label{13}
|\pt_t z_{R}(t)|\leq C_2 \mD^*R= C_2 \mD^*(R_0+\eta_2 t_1)
\end{align}
for some $C_2=C_2(\mD^*)>0$. \eqref{12} and \eqref{13} give us
\begin{align*}
2C_2 \mD^*(R_0+\eta_2 t_1)\geq\eta_1 (t_1-t_0).
\end{align*}
Setting $\eta_2=\frac{\eta_1}{4C_2\mD^*}$, dividing both sides by $t_1$ and then sending $t_1$ to infinity we obtain $\frac{1}{2}\eta_1\geq\eta_1$, which implies $\eta_1\leq 0$, a contradiction. This completes the proof.
\end{proof}

\subsection{Finite time blow-up below ground states}
In the final subsection we give the proof of Theorem \ref{thm blow up}.

\begin{proof}[Proof of Theorem \ref{thm blow up}]
The proof makes use of the classical Glassey's virial arguments \cite{Glassey1977}. In the context of normalized ground states, we shall invoke the same idea from the proof of \cite[Thm. 1.5]{BellazziniJeanjean2016} to show the claim. We firstly prove the following statement: for $\phi\in H_{x,y}^1$ satisfying $\mH(\phi)<m_{\mM(\phi)}$ and $\mK(\phi)<0$ one has
\begin{align}\label{energy trapping}
\mK(\phi)\leq \mH(\phi)-m_{\mM(\phi)}.
\end{align}
Indeed, from the proof of Lemma \ref{monotoneproperty} we know that there exists some $t^*\in(0,1)$ such that $\mK(\phi^{t^*})=0$ and $\frac{d}{ds}(\mH(\phi^{s}))(s)\geq \frac{d}{ds}(\mH(\phi^{s}))(1)=\mK(\phi)$ for $s\in(t^*,1)$. Then
\begin{align*}
\mH(\phi)&=\mH(\phi^1)=\mH(\phi^{t^*})+\int_{t^*}^1\frac{d}{ds}(\mH(\phi^{s}))(s)\,ds\geq\mH(\phi^{t^*})+(1-t^*)\frac{d}{ds}(\mH(\phi^{s}))(1)\nonumber\\
&=\mH(\phi^{t^*})+(1-t^*)\mK(\phi)> m_{\mM(\phi)}+\mK(\phi),
\end{align*}
which implies \eqref{energy trapping}. Next, using the same arguments as in the proof of Lemma \ref{invariance from mA} we infer that $\mK(u(t))<0$ for all $t$ in the maximal lifespan of $u$, thus also $\mK(u(t))\leq m_{\mM(u)}-\mH(u)$. We now define
\begin{align*}
V(t):=\int |x|^2|u(t)|^2\,dxdy.
\end{align*}
By using the same approximation arguments as in the proof of \cite[Prop. 6.5.1]{Cazenave2003} we know that $|x|u(t)\in L_{x,y}^2$ for all $t$ in the maximal lifespan of $u$. Direct calculation (which is similar to the one given in the proof of Theorem \ref{main thm}) yields
\begin{align*}
\pt_{t}^2V(t)=8\mK(u(t))\leq 8(m_{\mM(u)}-\mH(u))<0.
\end{align*}
This particularly implies that $t\mapsto V(t)$ is a positive and concave function simultaneously. Hence the function $t\mapsto V(t)$ can not exist for all $t\in\R$ and the desired claim follows.
\end{proof}

\subsubsection*{Acknowledgements}
The author acknowledges the funding by Deutsche Forschungsgemeinschaft (DFG) through the Priority Programme SPP-1886 (No. NE 21382-1).

\addcontentsline{toc}{section}{References}

\end{document}